\documentclass[11pt]{article}
\usepackage{mymacros}

\usepackage{tikz} 
\usepackage{tikz-cd}
\usepackage{lipsum}

\usepackage{dutchcal}
\usepackage{standalone}
\usepackage[normalem]{ulem}

\titleformat{\subsection}
{\normalfont\normalsize\bfseries}{\thesubsection}{1em}{}
\titleformat{\subsubsection}[runin]
{\normalfont\small\bfseries}{\thesubsubsection}{0.8em}{}

\title{Central Limit Theorem for Multi-Point Functions of the 2D Discrete Gaussian Model at high temperature}

\author{
  Jiwoon Park\footnote{Department of Pure Mathematics and
   Mathematical Statistics, University of Cambridge, Centre for
   Mathematical Sciences, Wilberforce Road, Cambridge, CB3 0WB, UK.
  \url{https://orcid.org/0000-0002-1159-2676}.   {\tt jp711@cantab.ac.uk}.}
}

\date{\vspace*{-2em}}

\definecolor{magenta}{rgb}{0.55, 0.0, 0.55}
\definecolor{darkmagenta}{rgb}{0.85, 0.0, 0.55}

\makeatletter
\newcommand{\customlabel}[2]{%
   \protected@write \@auxout {}{\string \newlabel {#1}{{#2}{\thepage}{#2}{#1}{}} }%
   \hypertarget{#1}{#2}
}
\makeatother

  {\list{}{\leftmargin=0.15in\rightmargin=0.15in}\item[]}%
  {\endlist}

\newcommand{\Tay}{\operatorname{Tay}}
\newcommand{\Rem}{\operatorname{Rem}}
\newcommand{\Loc}{\operatorname{Loc}^{(2)}}
\newcommand{\Loco}{\operatorname{Loc}^{(0)}}

\newcommand{\Eplus}{\mathbb{E}}
\newcommand{\Comp}{\operatorname{Comp}}

\newcommand{\one}{{\bf 1}}

\newcommand{\Et}{\mathbb{E}_{(\tau)}}

\newcommand{\kt}{\mathfrak{t}}
\newcommand{\htau}{h_{\kt}}
\newcommand{\ctau}{c_{\kt}}

\newcommand{\kM}{\mathfrak{M}}
\newcommand{\kK}{\mathfrak{K}}
\newcommand{\kU}{\mathfrak{U}}

\newcommand{\kC}{\mathfrak{C}}
\newcommand{\kg}{\mathfrak{g}}

\newcommand{\sg}{\operatorname{SG}}
\newcommand{\gsg}{\operatorname{G-SG}}
\newcommand{\dg}{\operatorname{DG}}
\newcommand{\gff}{\operatorname{GFF}}

\newcommand{\Her}{\operatorname{He}}

\newcommand{\regularityone}{\,\textnormal{R}1}
\newcommand{\regularitytwo}{\,\textnormal{R}2}
\newcommand{\Gammaone}{\,\Gamma1}
\newcommand{\Gammatwo}{\,\Gamma2}
\newcommand{\Gammathree}{\,\Gamma3}
\newcommand{\Gammafour}{\,\Gamma4}
\newcommand{\assumpf}{\,\textnormal{A}'_{\f}}
\newcommand{\assumpfa}{\,\textnormal{A}_{\f}}
\newcommand{\assumpu}{\,\textnormal{A}'_{u}}
\newcommand{\rginitial}{\,\Phi_{\operatorname{IC}}}

\newcommand{\varphip}{\varphi_{+}}

\newcommand{\Conn}{\operatorname{Conn}}

\newcommand{\f}{\mathcal{f}}
\newcommand{\cz}{\mathcal{z}}
\newcommand{\rr}{\mathcal{r}}

\renewcommand{\bar}[1]{\overline{#1}}
\newcommand{\alphaLoc}{\alpha^{(2)}_{\operatorname{Loc}}}
\newcommand{\alphaLoco}{\alpha^{(0)}_{\operatorname{Loc}}}

\usepackage[titles]{tocloft}
\def\newpageonoff{0} 
\def\tableofcontentsonoff{0} 

\begin{document}
\iftrue

\maketitle

\ifx\newpageonoff\undefined
{\red command undefined!!}
\else
  \if\tableofcontentsonoff1
  \tableofcontents
  \fi
\fi

\abstract{
We study microscopic observables of the Discrete Gaussian model (i.e., the Gaussian free field restricted to take integer values) at high temperature using the renormalisation group method. 
In particular, we show the central limit theorem for the two-point function of the Discrete Gaussian model 
by computing the asymptotic of the moment generating function
$\big\langle e^{\cz  (\sigma (0) - \sigma (y))} \big\rangle_{\beta, \Z^2}^{\dg}$ for $\cz \in \C$ sufficiently small.
The method we use has direct connection with the multi-scale polymer expansion used in \cite{dgauss1, dgauss2}, where it was used to study the scaling limit of the Discrete Gaussian model.
The method also applies to multi-point functions and lattice models of sine-Gordon type studied in \cite{MR634447}. 
}

\section{Introduction}

The Discrete Gaussian (DG) model is an effective interface model whose height variables are restrained to take integer values and the thermal fluctuation is controlled by the Dirichlet energy.
Specifically, let $\Lambda \subset \Z^{d}$ be a square domain such that $0\in \Lambda$. 
The DG measure $\P^{\dg}_{\beta, \Lambda}$ on $\Omega^{\dg}_{\Lambda} = \{  \sigma \in (2\pi \Z)^{\Lambda} : \sigma (0) = 0 \}$ with periodic boundary condition is defined by
\begin{align}
\P^{\dg}_{\beta, \Lambda} (\sigma) = \frac{1}{Z_{\beta, \Lambda}^{\dg}} \exp \Big( -\frac{1}{4 \beta} \sum_{x\sim y \in \Lambda} (\sigma_x - \sigma_y)^2  \Big)
\end{align}
where $\beta >0$ is the temperature parameter and $Z_{\beta, \Lambda}^{\dg}$ is a normalisation constant.
Also, $x\sim y$ indicates they are adjacent to each other when the domain $\Lambda$ is equipped with periodic structure and $\sum_{x\sim y}$ counts every edge of $\Lambda$ twice.
The expectation in the measure is denoted either $\E^{\dg}_{\beta, \Lambda} (\cdot)$ or $\langle \cdot \rangle^{\dg}_{\beta, \Lambda}$.
The exponent can be equivalently written as $-\frac{1}{2\beta} (\sigma, -\Delta \sigma)$ where $\Delta f(x) = \sum_{y: y\sim x} f(y) - f(x)$.

The case $d=2$ has been drawing attention from the probability and mathematical physics communities due to the presence of the localisation-delocalisation phase transition and its relation to other models such as the lattice Coulomb gas model and the Villain XY model.
For some introductory aspects of this phase transition and duality, one may consult \cite{1711.04720,RB16-MATH253x}.
When $\beta \ll 1$ (low temperature),  Peierls' argument shows that $\sup_{x\in \Lambda} \operatorname{Var} (\sigma_x)$ is bounded uniformly as $|\Lambda| \rightarrow \infty$, see  \cite{Brandenberger:1982aa}. 
Whereas for the case $\beta \gg 1$ (high temperature), 
Fr\"{o}hlich and Spencer were the first to use the spin-wave method in \cite{MR634447} to prove delocalisation, namely $\operatorname{Var} (\sigma_x - \sigma_y) \geq c(\beta) \log \norm{x-y}_2$ for some $c(\beta) > 0$ and uniformly over $|\Lambda|$. 
Since then,  the delocalisation was proved on more general graphs
\cite{2110.09498, 2012.09687} using new strategies, and a number of works \cite{2002.12284, 2101.05139, 2012.01400, MR2251117, 1907.08868,arxiv.2211.14365} developed different views on how to understand the delocalised phase.  
Also, it was proved recently in \cite{dgauss1, dgauss2} that the scaling limit of the DG model with $\beta \gg 1$ is in fact a usual continuum Gaussian Free Field with renormalised temperature using the renormalisation group argument.

Delocalisations are widely observed phenomenon in 2D random surface models, as seen in (the non-exhaustive lists) \cite{Brascamp2004, MR3395146, Ioffe2002,MR3182484, arxiv.2202.13578} for continuous-valued models and \cite{MR3606736,MR4121614,2012.13750,MR4315657,arxiv.2206.12058,  MR3369909} for discrete-valued models. 
Delocalisation in 2D is closely related to the absence of symmetry breaking in 2D systems with continuous symmetry,  the Mermin-Wagner theorem \cite{PhysRevLett.17.1133, Mermin1967}.  Such type of results can also be found in \cite{MR632763,Pfister1981,MR0424106}.
One may also consult \cite{Velenik-LectureNotes, MR2228384} for different probabilistic models,  questions, approaches and diverse lists of references on this phenomenon. 

In this paper, we aim to further develop the method of \cite{dgauss1, dgauss2} to understand the multi-point functions of the DG model in the delocalised phase, and thereby establish the central limit theorem for microscopic observables.

\subsection{Results}

\subsubsection{Infinite volume function.}

Our main theorem about the two-point function of the DG model is stated in terms of the infinite volume DG measure. 
However, in the high-temperature phase, a translation invariant infinite volume Gibbs measure with variance delocalisation does not exist,  cf. \cite{2101.05139}.
To resolve this, one usually considers the gradient Gibbs measure that only takes the gradient of the field $\sigma$ into account. 
Formally,  let us define $\P^{\dg}_{\beta, \Z^2}$ to be the probability measure on the quotient space $(2\pi\Z)^{\Z^2} / \sim$ (where $\sigma \sim \sigma'$ whenever $\sigma \equiv \sigma' + 2\pi c$ for some $c \in \Z$ independent of the lattice site)
such that $\langle F (\sigma) \rangle^{\dg}_{\beta,\Z^2} = \lim_{n\rightarrow \infty} \langle F(\sigma) \rangle^{\dg}_{\beta, \Lambda (n)}$ for some $(\Lambda(n))_{n\in \N}$, 
a sequence of tori increasing to $\Z^2$ as subsets of $\Z^2$, 
and $F$ is a bounded measurable function that only depends on finite number of points in $\Z^2$.
It is shown in \cite{MR2251117} that the measure obtained in this way is the unique (translation-)ergodic gradient Gibbs measure with tilt 0 in the high-temperature phase.  

In the high-temperature phase, the results of Fr\"ohlich-Spencer-Park (\!\! \cite{MR634447, MR0456220}, see also \cite{1711.04720}) show that this measure is approximately Gaussian in the sense that
\begin{align}
	\langle e^{ r(\beta) (f, \phi) } \rangle^{\gff}
	\leq 
 	\langle e^{\beta^{-1/2} (f, \sigma)} \rangle^{\dg}_{\beta} 
	\leq 
 	\langle e^{(f, \phi)} \rangle^{\gff} = e^{\frac{1}{2} (f, (-\Delta)^{-1} f)}
 	\label{eq:Frohlich-Spencer-Park}
\end{align}
for sufficiently large $\beta$, 
some $r (\beta) \in (0,1)$ such that $r(\beta) \rightarrow 1$ as $\beta \rightarrow \infty$
and $f \in \R^{\Z^2}$ of compact support such that $\sum_{x} f(x) =0$,
where $\phi \sim \P^{\gff}_{\Lambda}$, 
meaning that $\phi$ has the distribution of the Gaussian free field. 
When $f = \delta_0 - \delta_y$ for $y\in \Z^2$, 
then 
\begin{align}
	 \frac{1}{2} (\delta_0 - \delta_y, (-\Delta)^{-1} (\delta_0 - \delta_y))
		&= \frac{ \log \norm{y}_2 }{2 \pi} \cz^2  + C + O(\norm{y}^{-\alpha'})	,
		\label{eq:mgf_gff}
\end{align}
for some $\alpha' > 0$ (see \cite{MR4043225}, for example).
This shows the field is delocalised, but does not provide
sufficient amount of control when it comes to proving the scaling limit (proved in \cite{dgauss1, dgauss2}),
distributional convergence of the rescaled field at a point (Corollary~\ref{cor:CLT}) or determining the scaling dimension of  $e^{i\eta \sigma_y}$ (Corollary~\ref{cor:cos_correlation}).
In this paper, we aim to extend the method of \cite{dgauss1, dgauss2} to prove both of these results as corollaries of the following theorem.

\begin{theorem}
\label{thm:main_theorem}

There exist $\beta_0$ and $\htau >0$ such that, 
whenever $\beta \geq \beta_0$ and $\cz \in \D_{\htau} = \{ z \in \C : |z| < \htau  \}$,
\begin{align}
\begin{split}
	\log \langle e^{ \beta^{-1/2} \cz (\sigma_0- \sigma_y)} \rangle^{\dg}_{\beta, \Z^2} = \frac{ \log \norm{y}_2 }{2 \pi (1 + s_0^c (\beta))} \cz^2  + f_{\beta} (\cz, y)
	\label{eq:main_theorem}
\end{split}
\end{align}
for some $s_0^c (\beta) = O(e^{-c_f \beta})$ with $c_f >0$
and for some complex analytic function $f_{\beta} (\cdot ,  y) : \D_{\htau} \rightarrow \C$ that is bounded uniformly in $y$.
In fact, there is $f_{\beta} (\cdot, \infty)$ such that $\norm{f_{\beta} (\cdot, y) - f_{\beta} (\cdot, \infty) }_{L^{\infty} (\D_{\htau})} = O(\norm{y}_2^{-\alpha})$ for some $\alpha > 0$.
\end{theorem}

Note that $\htau$, the radius of validity of \eqref{eq:main_theorem}, 
must have a finite upper bound because the function $\cz \mapsto \langle e^{\beta^{-1/2} \cz (\sigma_0 - \sigma_y)} \rangle_{\beta, \Z^2}^{\dg}$ is periodic under translations by $2 \pi \beta^{1/2} i$.

The main message is twofold.
First,  the dominating term of the moment generating function in the limit $\norm{y}_2 \rightarrow \infty$ is just a constant multiple of that of the Gaussian free field given in \eqref{eq:mgf_gff}.
In other words, 
\begin{align}
	\log \langle e^{ \beta^{-1/2} \cz (\sigma_0- \sigma_y)} \rangle^{\dg}_{\beta, \Z^2} 
		= \frac{1}{1+ s_0^c (\beta)}	\log \langle e^{\cz (\phi_0 - \phi_y)} \rangle_{\Lambda}^{\gff} + 	C' + O(\norm{y}^{-\alpha})
\end{align}
so the divergence of log-moment generating function on a neighbourhood of 0 is entirely due to the second moment.
This can also be compared to \cite{MR3182484}, where the strong central limit theorem for continuous valued models is proved in this sense (but on the whole real line, not for complex $\cz$ as above).
Also, the theorem shows a polynomial decay of the remainder $f_{\beta} (\cdot, y) - f_{\beta} (\cdot, \infty)$.  
There are two different contributions on the remainder term, one decaying part of the lattice Laplacian $(-\Delta)^{-1} (0,y)$ and one specific to the model. 

Second, there is a deviation of the DG model from the Gaussian free field by a factor $s_0^c (\beta)$.  
This correction is identical to the one observed in the scaling limit, \cite{dgauss1}.
It is often explained to arise from the competition between the spin-wave (upper bound) and spin-vortex fluctuations (lower bound).
Our proof does not provide a lower bound,  but comparing \eqref{eq:main_theorem} with \cite[Corollary~6.2, Corollary~2.2]{2012.01400},  we see that $s_0^c (\beta)$ also satisfies $s_0^c (\beta) \geq \frac{\beta}{2\pi^2} e^{-\beta / 2}$ for $\beta \geq (\beta_0 \vee 4\pi^2/3)$.

Theorem~\ref{thm:main_theorem} can be generalised to the sine-Gordon type models of \cite{MR634447} (see Section~\ref{sec:relation_to_sine_gordon} and Remark~\ref{remark:sg_also_works}),  multi-point functions (see Section~\ref{sec:multi_point_functions}) and a larger domain of $\cz$ (the domain $S_{R, \htau}$ in Theorem~\ref{thm:main_theorem_generalised}).
But we confine most of our discussion to the case of the Discrete Gaussian model and the two-point function for $\cz \in \D_{\htau}$ in the interest of clarity.

Our main theorem complements an earlier result about the scaling limits in \cite{dgauss2} where the DG model is proved to converge to the full plane Gaussian free field.
To be precise,  \cite{dgauss2} considers observables that are `mesoscopic' in the sense that it is given as a sequence of functions $(f_n)_n \subset \R^{\Z^2}$ where $f_n(x)$ approximate $n^{-2} f(n^{-1} x)$ where $f \in C_c^{\infty} (\R)$ with $\int_{\R^2} f = 0$. 
Then $\langle e^{(f_n, \sigma)} \rangle_{\beta, \Z^2}^{\dg}$ is shown to converge to the moment generating function of the Gaussian free field.  To obtain Theorem~\ref{thm:main_theorem}, 
the method of \cite{dgauss2} requires two major extensions. 
The first is to refine the analysis to `microscopic' observables and the second is to extend the domain of the observables in the complex direction. 
The first amounts to controlling extra free energy coming from the point observables, which are not observed in the scaling limits. 
These points will be explained again in Section~\ref{sec:the_rg_method_explained}.

\subsubsection{Implications.}

An immediate consequence of Theorem~\ref{thm:main_theorem} is obtained by expanding the moment generating function as a Taylor series in $\cz$, so
\begin{align}
	\var_{\beta, \Z^2}^{\dg} \big[  \sigma_0 - \sigma_y  \big] = \frac{\beta}{2\pi (1+ s_0^c (\beta))} \log \norm{y}_2 + c_{\beta} + O_{\beta} \big( \norm{y}_2^{-\alpha'} \big)	\label{eq:variance}
\end{align}
for some $c_{\beta} \in \R$.
This shows the unboundedness of the variance directly. 
Also, we can find the natural scale associated to the size of the spin field by observing \eqref{eq:variance}, so let $t_{\beta} (y) = (\beta \log \norm{y}_2 )^{-1/2}$. 
After rescaling the spin field in Theorem~\ref{thm:main_theorem} by $t_{\beta} (y)$,  we obtain the precise distribution of the rescaled spins.

\begin{corollary}[Central limit theorem] \label{cor:CLT}

Under the assumptions of Theorem~\ref{thm:main_theorem}, for $\cz \in \C$,  
\begin{align}
\begin{split}
	\log \langle e^{ \cz t_{\beta} (y)  (\sigma_0 - \sigma_y)} \rangle^{\dg}_{\beta, \Z^2} = \frac{1}{2 \pi (1 + s_0^c (\beta))} \cz^2 + O_{\beta} \Big( \frac{\cz^2}{\log \norm{y}_2}  \Big)	.
\end{split}
\end{align}
\end{corollary}

When the field is allowed to take continuous values and the interaction potential satisfies uniform convexity condition, the same observable as in Corollary~\ref{cor:CLT} was shown to satisfy the central limit theorem in \cite{MR3182484}, 
and recently,  the local central limit theorem was also proved in \cite{arxiv.2202.13578}.
These results are obtained by applying stochastic homogenisation on the Helffer-Sj\"ostrand representation of the field. 
This technique has wide-ranging applications in this context,  but it is not directly applicable to integer valued systems.  
There are also results on some specific integer valued models. 
The height function for the dimer model was studied extensively in \cite{KenyonSheffield2003} and the central limit theorem was proved in \cite{Laslier2015}.
Also, a central limit theorem (with unknown scaling factors) was obtained for the square ice model in \cite{arxiv.2206.12058} using the Russo-Seymour-Welsh estimate on the level lines.

Another interesting related observable is the cosine correlation. 

\begin{corollary}[Cosine correlation] \label{cor:cos_correlation}
Under the assumptions of Theorem~\ref{thm:main_theorem} and if $\eta \in (-\htau, \htau)$, then
\begin{align}
\begin{split}
	\big\langle \cos\big( \beta^{-1/2} \eta  (\sigma_0 - \sigma_y) \big) \big\rangle^{\dg}_{\beta, \Z^2} = C (\beta, \eta) \norm{y}_{2}^{-\frac{\eta^2}{2\pi (1+s_0^c (\beta))}} \Big( 1 + O_{\beta} \big( \norm{y}_2^{-\alpha} \big) \Big)
\end{split}
\end{align}
for some $C(\eta,\beta)$ analytic in $\eta$. 
\end{corollary}

The interpretation is not so clear for the DG model.  But using the analogue for the generalised sine-Gordon models, the cosine correlation corresponds to the correlation function of two test charges $\pm \eta$ inserted into the 2D lattice multi-component Coulomb gas system dual to the generalised sine-Gordon model (also see \cite{1311.2237}).  For this reason, this observable is also called the fractional charge correlation or the electric correlator. 

The polynomial decay of the fractional charge correlation also characterises the delocalisation, and the polynomial lower and upper bounds were established in \cite{MR634447,MR610687}.
In the localised phase, the Debye screening (cf. \cite{MR574172, MR923850} for the lattice sine-Gordon model and \cite{Brandenberger:1982aa} for the DG model) induces exponential decay of the truncated charge correlation.  The same type of result was studied for the dimer model in \cite{Pinson2004, MR3369909}.
A similar result was also obtained for the interacting dimer model in \cite{MR3606736} but only after smoothing the point observables in an appropriate way. 
For the lattice sine-Gordon model at the critical point,  the fractional charge correlation was computed in \cite{1311.2237}.

\subsubsection{Finite volume functions.}

Theorem~\ref{thm:main_theorem} is proved by first observing what happens in finite volumes.
In what follows,
$\Lambda_N$ is $L^N \times L^N$ two-dimensional discrete torus with distinguished point 0, where $N \in \Z_{>0}$ and $L \in \Z_{\geq 2}$.
We consider $(\Lambda_N)_{N > 0}$ as nested subsets of $\Z^2$ with $0\in \Lambda_N$ and $\Lambda_N$ is represented by a corresponding subset of $\Z^2$. However, the metric $\operatorname{dist}_p (x,y)$ $(p\in [1,\infty])$ is that of the discrete torus, 
and also often is denoted $\norm{x-y}_{p}$ (even though it is not a norm). 

The correlation function in the infinite volume can be reduced to a result on the finite discrete torus, $\Lambda_N$.
For the finite volume result, we try to make the statement as general as possible. 
An interesting generalisation can be obtained by considering a large class of two-cluster observables of form $\f$ as in the following. 
\begin{equation} \stepcounter{equation}
	\tag{\theequation $\assumpfa$} \label{quote:assumpfa}
\begin{split}
	\parbox{\dimexpr\linewidth-4em}{%
		$y\in \Z^2$ and $\f_1,  \f_2 : \Z^2 \rightarrow \R$ are two functions with compact supports such that
$\f_1 (0), \f_2 (0) \neq 0$ and $\sum_{x\in \Z^2} ( \f_1 (x) +  \f_2 (x) ) = 0$.
Set $\f = \f_1 + T_y \f_2$ where $T_y \f_2 (x) = \f_2 (x-y)$.  Then set
	}
	\\
	M = \max\{ \norm{\f_1}_{L^{\infty}}, \norm{\f_2}_{L^{\infty}} \}, 
	\qquad 
	\rho = \max\{ \operatorname{diam}(\operatorname{supp} (\f_1) ) , \operatorname{diam}(\operatorname{supp} (\f_2) ) \} \qquad\;
\end{split}
\end{equation}

For Theorem~\ref{thm:main_theorem}, we will only need the case $\f_1 = \delta_0 = - \f_2$ so $\f = \delta_0 - \delta_y$,
but for later use, we record an alternative requirement on $\f$.
\begin{equation} \stepcounter{equation}
	\parbox{\dimexpr\linewidth-4em}{%
		$\f$ satisfies \eqref{quote:assumpfa},  but is allowed to have $\sum_x \f (x) \neq 0$.
	}	\tag{\theequation $\assumpf$} \label{quote:assumpf}
\end{equation}
\vspace{0pt}

Although $\f$ is a function defined on $\Z^2$, for $L^N$ sufficiently large compared to $\norm{y}_{\infty} + \rho$, 
$\f$ can be always be considered as a function on $\Lambda_N$. 
To be more formal, we will let $\Lambda'_N = [-\frac{L^N -1 }{2} , \frac{L^N -1}{2} ]^2 \cap \Z^2 \subset \Z^2$ if $L$ is odd and $\Lambda'_N  = [-\frac{L (L^N - 1)}{2(L-1)}, \frac{(L-2) (L^N - 1)}{2(L-1)} ]^2 \cap \Z^2$ if $L$ is even.
Also, let $\iota_N : \Lambda_N \rightarrow \Z_2$ be an isometric embedding such that $\iota_N |_{\Lambda'_N} : \Lambda_N \rightarrow \Lambda'_N$ is a bijection and $\iota_N (0) =0$.
Whenever we refer to $\f$ on $\Lambda_N$ below,  it will always mean $\f \circ \iota_N$ in real. 

Then Theorem~\ref{thm:main_theorem} can be reduced to the following theorem. 
In the statement, translation invariant covariance matrix means $\kC : \Lambda \times \Lambda \rightarrow \R$ such that $\kC (x_1,x_2) = \kC (x_1+ z, x_2 +z)$ for any $z\in \Lambda$ and $(f, \kC f) \geq 0$ for any $f \in \R^{\Lambda}$ (when $\Lambda$ is either $\Lambda_N$, a two-dimensional torus, or $\Z^2$).
We also use $S_{a,b} = \{ x + y : x \in [-a,a],  y \in \D_{b}  \} \subset \C$,  an elongation of the domain $\D_{b}$ in the real direction. 
The theorem should be interpreted in the line of Theorem~\ref{thm:main_theorem}, that the moment generating function of the two-cluster function behaves like that of the Gaussian free field with lower order corrections, at high temperature.

\begin{theorem} \label{thm:main_theorem_generalised}
There exists a translation invariant covariance matrix $\mathfrak{C}_{\beta,\Lambda_N}$  and $\beta_0 (R) , \, L_0 (R) > 0$ for each $R\geq 0$
such that the following holds.
Let $\cz \in S_{R,\htau}$ with $\htau \equiv \htau (M, \rho, L) >0$,  $L \geq L_0$
$\beta \geq \beta_0$ and $\sigma \sim \P_{\beta, \Lambda_N}^{\dg}$.
Then for any $y$ and $\f = \f_1 + T_y \f_2$ satisfying \eqref{quote:assumpfa},
\begin{align}
\begin{split}
	\log \langle e^{\beta^{-1/2} \cz (\f, \sigma)  } \rangle^{\dg}_{\beta, \Lambda_N} &= \frac{1}{2} \cz^2 (\f,  \mathfrak{C}_{\beta, \Lambda_N}  \f ) \\
	& \qquad +  h^{(1)}_{\beta} [\f_1] (\cz) + h^{(1)}_{\beta} [\f_2] (\cz) + h^{(2)}_{\beta} [\f_1, \f_2] (\cz, y) + \psi_{\beta,  \Lambda_N} (\cz, y) 
\end{split}
\label{eq:thm_main_theorem_generalised}
\end{align}
where $h^{(a)}_{\beta}$ ($a \in \{1,2\}$) and $\psi_{\beta, \Lambda_N}$ are  analytic functions in $S_{R,\htau} \ni \cz$
satisfying the following:
\begin{itemize}
\item for $\f_2 \equiv 0$ and for any $\f_1$, we have $h^{(1)}_{\beta} [0] = h^{(2)}_{\beta} [\f_1 ,  0] = 0$;
\item $| h^{(2)}_{\beta} [\f_1, \f_2] (\cz, y)  | = O_{\beta, R} ( \norm{y}_2^{-\alpha} )$ as $\norm{y}_2 \rightarrow \infty$
uniformly in $\cz \in S_{R, \htau}$ for some $\alpha  >0$; 
\item $\psi_{\beta, \Lambda_N} (\cz , y) = O_{\beta, R} \big( |\Lambda_N|^{-\alpha}  \big)$ as $N\rightarrow \infty$ uniformly in $\cz \in S_{R, \htau}$ and $y$; and
\item there exists a translation invariant covariance matrix $\mathfrak{C}_{\beta, \Z^2} : \Z^2 \times \Z^2 \rightarrow \R$ such that $(\f,  \mathfrak{C}_{\beta, \Z^2} \f) = \lim_{N\rightarrow \infty} (\f, \mathfrak{C}_{\beta, \Lambda_N}\f)$ (independent of choice of $L$).
\end{itemize}
\end{theorem}

The rest of the section is organised as follows. 
In Section~\ref{sec:remarks},  we make notes on generalisation of our main theorems,  make a conjecture about a possible refinement of our result and record some consequences of it. 
In Section~\ref{sec:the_rg_method}, we briefly discuss the main method of the proof,
the renormalisation group (RG). 
The RG appears in various contexts in mathematical physics. 
The method was already used in \cite{dgauss1,dgauss2} to study the Discrete Gaussian model, so we also explain how the method of this paper differs from them. 
In Section~\ref{sec:notations}, we attempt to give a brief summary of the  notations.

\subsection{Remarks and conjectures}
\label{sec:remarks}

\subsubsection{Finite range Discrete Gaussian model.} 
\label{sec:finite_range_laplacian}

As observed in \cite{dgauss1}, these results can also be extended to any finite ranged Discrete Gaussian models defined by the measure
\begin{align}
\P^{J-\dg}_{\beta, \Lambda_N} = \frac{1}{Z^{J-\dg}_{\beta, \Lambda_N}} \exp\Big( - \frac{1}{4 |J| \beta} \sum_{x-y \in J} (\sigma_x - \sigma_y)^2 \Big), \qquad \sigma \in \Omega^{\dg}_{\Lambda_N}
\end{align}
for any $J \subset \Z^2$ with compact support that is invariant under lattice rotations and reflections and includes the nearest-neighbour sites of the origin.
However, the uniqueness of infinite volume measure $\langle \cdot \rangle^{J-\dg}_{\beta, \Z^2}$ is not known, 
so only a specific choice of infinite volume limit is used in \cite{dgauss2}.

\subsubsection{Relation to the lattice sine-Gordon model.}
\label{sec:relation_to_sine_gordon}

 By the observation made in \cite{dgauss1} (see also \cite{MR2917175,MR1777310}), the same results hold for the lattice sine-Gordon model defined by the measure
\begin{align}
\P^{\sg}_{\beta, \Lambda_N} = \frac{1}{Z_{\beta, \Lambda_N}^{\sg} } \exp\Big( - \frac{1}{2\beta} (\sigma, -\Delta \sigma) + z \sum_{x\in \Lambda_N} \cos( \sigma(x) )  \Big)
\end{align}
when $\beta > 8\pi$ and activity $z$ is sufficiently small.
Also by Remark~\ref{remark:sg_also_works}, the same holds for a wide class of models of sine-Gordon type (that includes the sine-Gordon model with any activity $z \in \R$ and $\beta$ chosen sufficiently large depending on $z$).

For the lattice sine-Gordon model, near the critical point,  Corollary~\ref{cor:cos_correlation} complements the method of Falco \cite{1311.2237}, where the same quantity is computed for 
$\eta \in (0,  \beta^{1/2})$.
For the case when $J$ has sufficiently long range for finite range models in Section~\ref{sec:finite_range_laplacian}, Falco's method of using the observable fields also extends the range of $\eta \in (0,  \beta^{1/2})$
when $\beta$ is near the critical value.  (It was seen in \cite{dgauss1, dgauss2} that the near-critical values of these models are within the scope.) 
However,  it does not give the required bounds for the usual Discrete Gaussian model where we can implement the renormalisation group method only for large values of $\beta$ (well above the critical point).  
Using the observable fields is actually more common in this context,  for example,  as in \cite{MR3345374, MR3459163, arxiv.2007.10869}.

\subsubsection{Multi-point functions.} 
\label{sec:multi_point_functions}

The proof of Theorem~\ref{thm:main_theorem_generalised}  also implies analogous results for the multi-point functions. 
Let $\vec{y} = (y_i)_{1\leq i \leq I} \subset \Z^2$,  
$
d_{\vec{y}} = \min \{ \norm{y_{i_1} - y_{i_2}}_2 :i_1 \neq i_2 \}
$
and $(\f_i)_{i=1}^I \subset \R^{\Z^2}$ be such that $\sum_i \sum_x \f_i (x) = 0$. 

\begin{theorem}
There is $\beta_0 (R)$ such that whenever $\beta \geq \beta_0 (R)$ and $\cz \in S_{R, \htau}$, we have
\begin{align}
\begin{split}
	\log \big\langle \exp\big( \beta^{-1/2} \cz {\textstyle ( \sum_{i=1}^I T_{y_i} \f_i, \sigma ) } \big) \big\rangle_{\beta, \Z^2}^{\dg} 
	&= \frac{1}{2} \cz^2 \big( {\textstyle \sum_{i=1}^I T_{y_i} \f_i,  \mathfrak{C}_{\beta, \Lambda_N} \sum_{i=1}^I T_{y_i} \f_i } \big) \\
	& \quad 
	+  \sum_{i=1}^I h^{(1)}_{\beta} [\f_i] (\cz) + h^{(2)}_{\beta} \big[(\f_i)_{i=1}^I \big] (\cz,  \vec{y})  \label{eq:multi-point_function}
\end{split}
\end{align}
where 
$h^{(2)}_{\beta} (\cz,  \vec{y}) = O_{\beta} (d_{\vec{y}}^{-\alpha})$ is a function analytic on a bounded domain containing $0$
and $h_{\beta}^{(1)}$ is the function from Theorem~\ref{thm:main_theorem_generalised}. 
\end{theorem}

The analogue of Theorem~\ref{thm:main_theorem} also holds true for any multi-point function,  i.e.,  for any $(y_i)_{i=1}^I \subset \Z^2$ and $(c_i)_{i=1}^I$ such that $\sum_{i} c_i = 0$, there exists $\beta_0 (R)$ such that whenever $\beta \geq \beta_0 (R)$ and $\cz \in S_{R, \htau}$, 
\begin{align}
\begin{split}
	& \log \langle e^{\beta^{-1/2}  \cz \sum_{i=1}^I c_i \sigma_{y_i} } \rangle_{\beta, \Z^2}^{\dg} =  \frac{\cz^2}{2\pi (1+s_0^c (\beta))} \mathfrak{L} (\vec{c}, \vec{y}) + f_{\beta} (\cz, \vec{y})
\end{split}
\end{align}
where
\begin{align}
	\mathfrak{L} (\vec{c}, \vec{y}) := - \sum_{i<j} c_{i} c_j \log \norm{y_i - y_j}_2
\end{align}
and $f_{\beta}(\cz,y)$ is some complex analytic functions such that $ f_{\beta} (0,y) = \partial_{\cz} f_{\beta} (0,y) =0$ and
$\sup_{\cz} |f_{\beta} (\cz, y) - f_{\beta} (\cz, \infty) | = O_{\beta} ( d_{\vec{y}}^{-\alpha} )$.
Results of form Corollary~\ref{cor:CLT},\ref{cor:cos_correlation} also follow. 
In particular, 
\begin{align}
 	\big\langle e^{i \beta^{-1/2} \eta \sum_{i=1}^I c_i \sigma_{y_i} } \big\rangle_{\beta,\Z^2}^{\dg} = C(\beta,\eta ) \prod_{i<j} \norm{y_i - y_j}_2^{\frac{\eta^2}{2\pi (1 + s_0^c )} c_i c_j} \Big( 1+ O\big( d_{\vec{y}}^{-\alpha} \big) \Big)
\end{align}
for $\eta$ small.  This is exactly as the Gaussian free field prediction (see \cite{kang2013gaussian} for an in-depth study and \cite[(2.6)]{MR3369909} for a brief introduction).

\subsubsection{Conjecture on gradient correlation.}
\label{sec:gradient_correlatoin}

If $\f_1 = \delta_{\mu} - \delta_0$ and $\f_2 = (\delta_{\nu} - \delta_0)$ for some $\mu, \nu \in \{ \pm e_1, \pm e_2 \}$,  then with a more careful construction, it is believable that there are $\alpha>2$ and $C' \in \R$ such that
\begin{align}
	h^{(2)}_{\beta} (\cz, y)= C' \nabla^{\mu} \nabla^{-\nu} (-\Delta)^{-1} (y,0) +  O(\norm{y}_{2}^{-\alpha} ) \quad \text{as $\norm{y}_2 \rightarrow \infty$}.
	\label{eq:spectral_decay_of_remainder}
\end{align}
although the current technology only gives access to $\alpha >0$. 
We expect that a refinement of the renormalisation group method would yield such an estimate.
If we had \eqref{eq:spectral_decay_of_remainder}, then by the proof of the theorem,
\begin{align}
	\operatorname{Cov} ( \nabla^{\mu} \sigma (0) , \nabla^{\nu} \sigma(y)  ) = C(\beta) \nabla^{\mu} \nabla^{-\nu} (-\Delta)^{-1} (y,0) + O_{\beta} ( \norm{y}_2^{-2-\epsilon}  )
\end{align}
for some $\epsilon>0$.
Due to the duality relation between the Villain XY model and the DG model, this would also imply the energy correlation of the Villain XY model.

For continuous-valued models, estimates of this type can be found in \cite{Delmotte2005OnET} when the potential is uniformly convex and in \cite{MR2976565, arxiv.2007.10869} when the potential is a small perturbation of a uniformly convex function.

\subsection{Renormalisation group method}
\label{sec:the_rg_method}

We outline the proof of the main theorems here. We first explain the RG method of \cite{dgauss1} and explain how it is modified in Section~\ref{sec:the_rg_method_explained}.

\subsubsection{Bulk RG.}
\label{sec:bulk_RG_explained}

The RG method, which lies at the kernel of our proof, 
was already used in \cite{dgauss1} to study the scaling limit. 
We briefly overview the method here. 

We first reformulate the DG model with a continuous-valued statistical physics model defined in Section~\ref{section:two-point_function_as_a_tilted_expectation},
so that the partition function is written alternatively as a Gaussian expectation
\begin{align}
	Z_{\beta, \Lambda}^{\dg} \propto \lim_{m^2 \downarrow 0} \E\big[ e^{ h_0 (\Lambda, \varphi)} \big], \qquad \varphi \sim \cN(0,  C(s,m^2))
	\label{eq:h_0_refomulation}
\end{align}
for a modified covariance matrix $C(s,m^2)$ (given later by \eqref{eq:C_m^2_definition}) and a modified Hamiltonian (or a potential function)
\begin{align}
	h_0 (\Lambda, \varphi) = \frac{1}{2} s |\nabla \varphi|^2_{\Lambda} + \sum_{x\in \Lambda} \sum_{q \geq 1} z_0^{(q)} \cos( q \beta^{1/2} \varphi(x) )		,
	\qquad \varphi \in \R^{\Lambda_N}
	.
	\label{eq:h_0_definition}
\end{align}
The first term is a reminiscence of the gradient interaction, and the second term arises from smoothing the periodic function $\sum_{m\in \Z} \delta_m$ on $\R$.  

Then the RG method is applied to control the right-hand side of \eqref{eq:h_0_refomulation},
outlined as the following. 
We first decompose the covariance matrix as $C(s,m^2) = \sum_{j=1}^N \Gamma_j$ where $\Gamma_j$ is roughly the fraction of $C(s,m^2)$ at length scale $L^{j}$ (see Section~\ref{sec:frd_brief} for detailed conditions).
This gives
\begin{align}
	Z_{\beta, \Lambda}^{\dg} \propto \lim_{m^2 \downarrow 0} \E^{\zeta_N}_{\Gamma_N} \cdots \E^{\zeta_{1}}_{\Gamma_1} [ e^{h_0 (\Lambda,  \zeta_1 + \cdots + \zeta_N )}  ]
\end{align}
for independent Gaussian fields $\zeta_j \sim \cN(0, \Gamma_j)$.
If we perform integrals upto length scale $L^j$ and normalise with some choice of $E_j$,
we obtain
\begin{align}
	e^{h_j (\Lambda, \varphip )} := e^{-E_j |\Lambda|} \E^{\zeta_{j}}_{\Gamma_{j}} \cdots \E^{\zeta_{1}}_{\Gamma_1} [ e^{h_0 (\Lambda,  \zeta_1 + \cdots + \zeta_j + \varphip )}  ]
	.
\end{align}
The exponent $h_j$ is called the \emph{effective potential} at scale $j$, since it replaces the role of $h_0$ after the integration.
The effective potential contains information about the spin model when we zoom in $L^j$ times and rescale the field size by a factor of $L^{[\varphi]} := L^{\frac{d-2}{2}}$ (in our case, $d=2$, so is just 1).
These are two scalings that determine the natural scaling of $h_j$.

The RG map is a function that describes the evolution of $h_j$'s. 
However, the dynamics is not always stable under arbitrary choice of $h_0$, so we identify those components of $h_j$ as
\begin{align}
	U_j (\Lambda, \varphip) = \frac{1}{2} s_j |\nabla \varphi|^2_{\Lambda} + \sum_{x\in \Lambda} \sum_{q\geq 1} L^{-2j} z_j^{(q)} \cos(q \beta^{1/2} \varphip (x) )
	,
\end{align}
and the remaining (stable) part as $K_j$. 
Notice that $U_j$ is already naturally rescaled:
for the first term, the volume scaling $L^{-2j}$ and scaling for the gradient $L^{2j}$ cancel, and for the second term, only the volume scaling is present.
Then formally,  the RG map is a function $(s_j, (z_j^{(q)})_q , K_j) \mapsto (s_{j+1}, (z_{j+1}^{(q)})_q,K_{j+1})$ for each $j \in \N_{\ge 0} $.
It can be seen that $s_j$ and $z_j^{(q)}$ can be chosen to satisfy
\begin{align}
	s_{j+1} = s_j + O(L^{-\alpha j}) , \qquad z_{j+1}^{(q)} = L^{2-\frac{\beta}{4\pi} q^2} z_j^{(q)} 	.
\end{align}
Using a stable manifold theorem for $\beta >0$ sufficiently large,
we can tune the value of initial $s = s_0$ so that the flow of $(U_j, K_j)_{j\geq 0}$ is stable and tends to 0 as $j\rightarrow \infty$.
This means that the effective potential in the limit $j\rightarrow \infty$ is null, so the spin system is only described by the background Gaussian $\sum_{k\geq j} \Gamma_{k}$.
The choice of $s_0$ is an important input from \cite{dgauss1}, and is stated again in Theorem~\ref{thm:tuning_s-v2}.

Rigorous implementation of this procedure is achieved by representing the effective potentials in terms of polymer expansion described in Section~\ref{sec:Renormalisation_group_coordinates}.

\subsubsection{Observable RG--novelties of this paper.}
\label{sec:the_rg_method_explained}

In this paper, we append two-point observables on top of the bulk RG flow. 
After reformulating the DG model by \eqref{eq:h_0_refomulation}, 
the moment generating function of the two-point function is replaced by the  expectation of $e^{h_0 (\Lambda, \varphi)}$ exponentially tilted by a two-cluster function.
This is the content of Lemma~\ref{lemma:first_reformulation_Z2}.
The exponential tilting also shifts the effective potentials, thus a new observable RG flow has to be constructed and analysed. 
However, the observable RG flow can be treated as a perturbation of the bulk RG flow.  Due to the robustness of analysis settled in \cite{dgauss1} and stability of the RG map with respect to the perturbation,  we can take the bulk RG flow as the reference point to prove stability of the observable RG flow. 

In Section~\ref{sec:overview_of_proof},  we first assume the key estimate (Proposition~\ref{prop:overview_of_the_proof}) on the perturbed system and prove the main theorems based on this estimate. 
The key estimate is finally proved in Section~\ref{sec:renormalisation-group-computation}, using the observable RG flow.  
Construction of the observable RG flow and its properties appear in Section~\ref{sec:polymer_activities_and_norms}--Section~\ref{sec:proof-of-theorem-thm-local-part-of-K},
and the stability is proved in Section~\ref{sec:renormalisation-group-computation}.
A detailed overview on the technical aspects of the proof will also appear in Section~\ref{sec:Renormalisation_group_coordinates}.

Controlling the observable RG flow can be compared to that of \cite{dgauss2}, where the scaling limit of the DG model was obtained.
In \cite{dgauss2}, a (discretised) smooth function was tested against the field,  making the perturbation of the observable RG flow negligible compared to untilted RG flow in the scaling limit. 
On the other hand, in this work, the deviation of the observable RG flow is not negligible and builds up extra free energy along the flow. 
This necessitates dealing with an extra non-regularity, and is the origin of $h^{(\alpha)}_{\beta}$'s of Theorem~\ref{thm:main_theorem_generalised}.

Another new aspect of this paper is that the analyticity of the effective potentials are tracked along the RG flow.  
(Note that the analyticity of the left-hand side of \eqref{eq:main_theorem} on $\cz \in (-\htau,\htau) + i\R \subset \C$ is automatic once we have \eqref{eq:main_theorem} just on $\cz \in (-\htau,\htau)$, by the Vitali's convergence theorem. The difficult part is having  the correct estimate.)
Usually, the results of types Corollary~\ref{cor:CLT} and Corollary~\ref{cor:cos_correlation} are derived from very different analysis, but in this paper, they can be treated in the same framework due to the analyticity.

\subsection{Notations}
\label{sec:notations}

\subsubsection{Derivatives.}

There are three different derivatives.
If $D\subset \C$ is some domain 
and we have $f(\varphi, \tau)$, a smooth function $\R^{\Lambda_N} \times D \rightarrow \C$, we will use $D f$ for the derivative in the argument $\varphi$ and $\partial_{\tau} f$ for the derivative in $\tau$. 
Given $\varphi \in \R^{\Lambda_N}$, $\nabla^{\mu} \varphi(x) = \varphi(x+ \mu) - \varphi(x)$ will denote the discrete derivative in direction $\mu \in \hat{e} := \{\pm e_1, \pm e_2 \}$ ($(e_1, e_2)$ is the stander basis of $\Z^2$).
Any $\vec{\mu} = (\mu_1, \cdots, \mu_n) \in \hat{e}^n$ is called a multi-index with $|\vec{\mu}| = n$, and defines
\begin{align}
\nabla^{\vec{\mu}} \varphi = \nabla^{\mu_n} \cdots \nabla^{\mu_1} \varphi, \qquad \nabla_j^{\vec{\mu}} = L^{2j} \nabla^{\vec{\mu}}.
\end{align}
For $n\geq 0$, $\nabla^n \varphi$ will denote the collection $(\nabla^{\vec{\mu}} \varphi)_{|\vec{\mu}| = n}$ and
\begin{align}
|\nabla \varphi (x)|^2 = \sum_{\mu \in \hat{e}} | \nabla^{\mu} \varphi (x)  |^2 
.
\end{align}
Then for $j\geq 0$ and $X\subset \Lambda_N$, we define $\norm{\nabla_j^n \varphi}_{L_j^2(X)}$, $\norm{\nabla_j^n \varphi}_{L_j^2(\partial X)}$, $\norm{\nabla_j^n \varphi}_{L^{\infty}(X)}$ and $\norm{\varphi}_{C_j^2}$
as in \cite[Definition~5.2]{dgauss1}:
\begin{align}
	&\norm{\nabla^n_j f}_{L^{\infty}(X)} = \max_{x\in X} \max_{\vec{\mu}}  |\nabla^{\vec{\mu}}_j f (x)| \label{eq:norminfty}\\
	&\norm{\nabla^n_j f}^2_{L^{2}_j (X)} = L^{-2j} \sum_{x\in X} \sum_{\vec{\mu}} 2^{-n}
	|\nabla_j^{\vec{\mu}} f (x)|^2  \label{eq:normL2}\\
	&\norm{\nabla^n_j f}^2_{L^{2}_j (\partial X)} = L^{-j} \sum_{x\in \partial X} \sum_{\vec{\mu}} 2^{-n} 
	|\nabla^{\vec{\mu}}_j f (x)|^2 \label{eq:normL2boundary}\\
	& \norm{f}_{C^2_j (X)} = \max_{n=0,1,2} \norm{\nabla^n_j f}_{L^{\infty}(X)} \label{eq:normC2}
	.
\end{align}

\subsubsection{Polymers.}

Our implementation of the RG method, called the polymer expansion, uses the notions of $j$-scale blocks ($\cB_j$) and polymers ($\cP_j$). These are fairly standard in this field, so one is likely to find out similar notations in similar contexts,  for example,  as in \cite{1910.13564,  MR2523458,  MR1777310}. 
Let
\begin{align}
	B_0^j =
		\big[ -\ell_j (L) ,  L^j - \ell_j (L) - 1  \big]^2  \cap \Z^2  
\end{align}
where $\ell_j (L) = \frac{L^j -1}{2}$ if $L$ is odd and $\ell_j (L) = \frac{L (L^j -1)}{2(L-1)}$ if $L$ is even.
Also in $\Lambda_N$, let $B_0^j$ denote an analogous object, whenever a distinguished point $0\in \Lambda_N$ is specified.
In particular, they satisfy $\operatorname{dist}_{\infty} (0,  (B_0^j)^c) \geq \frac{L^j }{3}$.
Either on $\Z^2$ or $\Lambda_N$, $\cB_j$ is the set of $L^j \Z^2$-translations of $B_0^j$ and $\cP_j$ is the set of any union of blocks in $\cB_j$.
(The empty set is also in $\cP_j$.)
By the choice of $B_0^j$, any $j+1$-polymer is also a $j$-polymer for any $j\geq 0$.
The following is a list of relevant definitions.
\begin{itemize}
\item For $X\in \cP_j$, denote $\cB_j (X)$ for the collection of $B\in \cB_j$ such that $B\subset X$. Also denote $|X|_j = |\cB_j (X) | = L^{-2j} |X|$.
\item $X, Y \in \cP_j$ are connected (or $X\sim Y$) if $\dist_{\infty} (X,Y) \leq 1$. The set of connected polymers is denoted $\Conn_j$ and the connected components are denoted $\Comp_j (X)$.

\item The set of small polymers, $\cS_j$, are the collection of $X\in \cP_j \backslash \{ \emptyset \}$ such that $|X|_j \leq 4$ and $X\in \Conn_j$. For $X\in \cP_j$, denote $X^*$ (small set neighbourhood of $X$) for the union of small polymers intersecting $X$. 
\item For $X\in \cP_j$, denote $\bar{X}$ for the smallest $j+1$-scale polymer containing $X$.
\end{itemize}

Although $\f_1$ and $\f_2$ are supported on $B_0^j$ by \eqref{quote:assumpfa}, it is not necessarily true that $T_y \f_2$ is supported on a single scale-$j$ block.  New notations are required to denote these blocks:
for $y = (y_1, y_2) \in \Z^2$, let
\begin{align}
	& B_{1}^j = \text{unique block in $\cB_j$ that contains ${\textstyle \big( y_1 - \ell_j (L) , y_2 - \ell_j (L) \big) }$} \\
	& B^j_2 = B^j_1 + L^j e_1, 
	\quad  B^j_3 = B^j_1 + L^j e_2,
	\quad B^j_4 = B^j_1 + L^j (e_1 + e_2 )
\end{align}
so $\supp (T_y u_{j,2}) \subset \cup_{k=1}^4 B^j_k$. 
For future use, we also define $j$-polymers
\begin{align}
	P^j_y := \cup_{k=0}^4 B^j_k ,
	\qquad Q_y^j = \cup_{k=1}^4 B_k^j
.
\label{eq:union_of_B^j_k's}
\end{align}

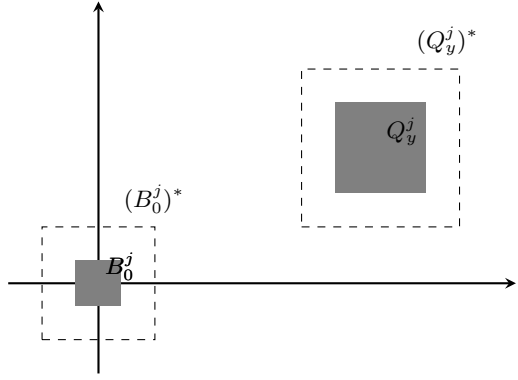
\begin{figure}[t] 
\centering
\begin{tikzpicture}[scale=1.5]

\draw[thick,-stealth] (-8.3,3) -- (-3.8,3);
\draw[thick,-stealth]  (-7.5,2.2) -- (-7.5,5.5);

\path [fill=gray, opacity=0.15] (-7.7,3.2) rectangle (-7.3,2.8);
\path [draw=black, opacity=0.6, dashed] (-8.0,3.5) rectangle (-7.0,2.5);
\node at (-7.35, 3.15) {{ \footnotesize $B_0^j$ }};
\node at (-7.05, 3.75) {{ \footnotesize $(B_0^j )^*$ }};

\filldraw[gray] (-5.1,4.05) circle (1pt) ; 
\path [fill=gray, opacity=0.15] (-5.4,3.8) rectangle (-4.6,4.6);
\path [draw=black, opacity=0.6, dashed] (-5.7,3.5) rectangle (-4.3,4.9);
\node at (-4.85, 4.35) {{ \footnotesize $Q_y^j$ }};
\node at (-4.45, 5.15) {{ \footnotesize $(Q_y^j)^*$ }};
\node at (-7.35, 3.15) {{ \footnotesize $B_0^j$ }};

\end{tikzpicture}
\caption{$y$ is the grey dot and $B_0^j$ and $Q_y^j$ are grey rectangles. The dotted rectangles are small set neighbourhoods of the grey rectangles, which will be used later.}
\end{figure}

Smooth functions on polymers are called polymer activities, defined as the following.

\begin{definition}
\label{def:polymer_activities}
Fix $\Lambda_N$, $\beta >0$ and $\htau >0$.

\begin{itemize}
\item Given $X \in \cP_j$, $\cN_j (X)$ is the set of smooth functions $\varphi \mapsto F (X, \varphi)$ of $\varphi \in \R^{\Lambda_N}$ that only depends on $\varphi |_{X^*}$
and $\cN_j$ is the collection $(F (X,  \cdot) \in \cN_j (X) : X \in \Conn_j )$.
$F \in \cN_j$ is said to obey lattice symmetries if 
$F(X, \varphi) = F(S X, \varphi \circ S)$ for $S$ either a translation or rotation or reflection (either in $\Lambda_N$ or $\Z^2$).
$F$ is said to be periodic if for any $n \in 2\pi \beta^{-1/2} \Z$, $F$ satisfies $F(X, \varphi + n \one ) = F(X, \varphi)$.
These are just as in \cite[Definition~5.1]{dgauss1}.

\item 
$\cN_{j, \kt} (X)$ is the subset of $C^{\infty} (\R^{\Lambda} \times \D_{\htau})$ such that $F(\varphi ; \tau)$ only depends on $(\varphi|_{X^*}, \tau)$ 
and $D^n F (\varphi ; \tau)$ is analytic in $\tau \in \D_{\htau}$ for each $n \geq 0$. 
Also, let $\cN_{j, \kt}$, the $\kt$-polymer activities, 
be the set of collections $F = (F(X, \cdot ; \cdot))_{X\in \Conn_j}$ such that $F(X, \cdot ; \cdot ) \in \cN_{j, \htau} (X)$ for each $X\in \Conn_j$.
\end{itemize}
\end{definition}

The analyticity of $D^n F (\cdot ; \tau)$ is is used only when $n=0$ for $F\in \cN_{j, \kt}$, 
but appending the case of general $n$ does not cause any extra difficulty. 
If $H$ is some function of $\cP_j$, then we also use the notions of polymer powers,
\begin{align}
H^X := \prod_{B\in \cB_{j} (X)} H(B),
\qquad
H^{[X]} := \prod_{Y \in \Comp_{j} (X)} H(Y) 
.
\label{eq:polymer_power_convention}
\end{align}
For $F$ in either $\cN_j$ or $\cN_{j,\kt}$,
we can define a natural extension to any polymer $X\in \cP_j$ by
\begin{align}
	F(X, \varphi) := F^{[X]} (\varphi) =  \prod_{Z \in \Comp_j (X)} F(Z,  \varphi) ,
	\qquad \varphi \in \R^{\Lambda}
	.
\label{eq:factorisation_of_polymer_activities}
\end{align}

\subsubsection{Gaussian integrals.}

We use special notations for Gaussian integrals:
for a finite set $V$,
a Gaussian random variable $X$ on $\R^{V}$ with mean $x \in \R^{V}$ and covariance $C \in \R^{V \times V}$ is denoted $X \sim \cN(x, C)$.
If $F : \R^{V} \rightarrow \R$ is a measurable function, then denote $\E^X_C [F(X)]$ for the Gaussian expectation with $X\sim \cN(0, C)$. 
The covariance $C$ and random variable $X$ are omitted if they are clear from the context, e.g., by just denoting $\E [F(X)]$.
Gaussian integrals are also called the fluctuation integrals. 

\ifx\newpageonoff\undefined
{\red command undefined!!}
\else
  \if\newpageonoff1
  \newpage
  \fi
\fi

\section{Two-point function as a tilted expectation}
\label{section:two-point_function_as_a_tilted_expectation}

This section shows how to reduce the two-point function in the DG model to a partition function of a modified statistical physics model with an external field and a shifted environment, 
which will be written in terms of an expectation with exponential tilting. 
Renormalisation group method will be used to control this tilted expectation,
so we also introduce basic elements that are required to construct the renormalisation group.

\subsection{Regularisation of the model}
\label{sec:regularisation_process}

$\langle e^{\beta^{-1/2} \cz (\f, \sigma)} \rangle^{\dg}_{\beta, \Lambda_N}$ is computed by first approximating the model with a massive model (with also the state space rescaled) defined by
\begin{align}
	\langle F(\sigma) \rangle_{\beta, m^2, \Lambda_N}
	= 
	\frac{\sum_{\sigma \in \Z^{\Lambda_N}_{\beta}} e^{-\frac{1}{2} (\sigma, (-\Delta +m^2) \sigma )} F(\sigma)}
	{\sum_{\sigma \in \Z^{\Lambda_N}_{\beta}} e^{-\frac{1}{2} (\sigma, (-\Delta +m^2) \sigma )}}
	\label{eq:massive_model_defi}
\end{align}
where $\Z_{\beta}^{\Lambda_N} = (2 \pi \beta^{-1/2} \Z)^{\Lambda_N}$.
Now $\sigma$ need not be fixed to be $\sigma (0) = 0$ and the model can be considered as a discrete model with covariance $(-\Delta +m^2)^{-1}$ acting on the configuration space $\Z_{\beta}^{\Lambda_N}$.
Then by \cite[Lemma~2.1]{dgauss1},
\begin{align}
\langle F( \beta^{-1/2} \sigma) \rangle^{\dg}_{\beta, \Lambda_N} = \lim_{m^2 \downarrow 0} \langle F(\sigma) \rangle_{\beta,m^2, \Lambda_N}
\label{eq:m^2_limit_justification}
\end{align}
whenever $F : \R^{\Lambda_N} \rightarrow \C$ is such that $F(\psi) = F(\psi + n \one)$ for any $n\in 2\pi \beta^{-1/2} \Z$, the constant field $\one (x) \equiv 1$
and $F$ is uniformly integrable with respect to the measures appearing in the equation.

Secondly, the underlying discreteness of the model is smoothened by integrating the `ultralocal' part of the covariance and the stiffness renormalisation is performed by extracting $\frac{1}{2} s (\varphi, -\Delta \varphi)$ from the covariance. 
Specifically, we take new covariances
\begin{align}
	C(m^2 ) = (-\Delta + m^2)^{-1} - \gamma \id,
	\qquad
	C(s,m^2) = (C(m^2)^{-1}  - s \Delta)^{-1}
\label{eq:C_m^2_definition}
\end{align}
($\gamma$ and $s$ are taken sufficiently small 
so that $C(s,m^2)$ is still positive definite--see \cite[Section~3]{dgauss1} for specific conditions and also Section~\ref{sec:choice_of_parameters} for the choice of $\gamma$)
and rewrite the moment generating function in the next lemma.

\begin{lemma}
\label{lemma:first_reformulation_Z2}
For all $\beta >0$,  $\gamma \in (0,1/3)$, $m^2 \in (0,1]$, $s$ small, 
$\cz \in \C$,
\begin{equation} 
	\big\langle e^{\cz (\f, \sigma)} \big\rangle_{\beta, m^2,\Lambda_N}
   		=  e^{\frac{\gamma}{2} \cz^2 (\f,  \tilde{\f})} \frac{ \E_{C(s,m^2)} \big[ e^{\cz (\tilde{\f}, \varphi )} Z_0(\varphi + \gamma \cz \f) \big] }{ \E_{C(s,m^2)} \big[ Z_0(\varphi) \big] } 
    	.
    \label{eq:reformulation-Z2-intermediate}
\end{equation}
where 
$\tilde{\f} = (1 + s\gamma \Delta ) \f$ and
\begin{align}
Z_0 (\varphi) 
:= e^{h_0 (\Lambda, \varphi)} := \exp \Big( \frac{1}{2} s_0 |\nabla \varphi|^2_{\Lambda} + \sum_{x\in \Lambda} \sum_{q\geq 1} z_0^{(q)} \cos \big(  q \beta^{1/2} \varphi(x) \big)  \Big)
\label{eq:Z_0_definition}
\end{align}
with $s_0 = s$ and some $(z_0^{(q)})_{q \geq 1}$ satisfying $| z_0^{(q)} (\beta) | = O( e^{- \frac{1}{4} \gamma \beta (1+q) } )$ for each $n \in \N$ and $\beta >0$.
\end{lemma}
\begin{proof}
We use the convolution formula \cite[(2.13)]{dgauss1}
\begin{align}
	e^{-\frac{1}{2} (\sigma, (-\Delta +m^2) \sigma)} 
		& \propto \int_{\R^{\Lambda}} e^{-\frac{1}{2\gamma} \sum_x (\varphi_x - \sigma_x)^2} e^{-\frac{1}{2} (\varphi, C(m^2)^{-1} \varphi)} d\varphi \nnb
		& = \int_{\R^{\Lambda}} e^{-\frac{1}{2\gamma} \sum_x (\varphi_x - \sigma_x)^2 + \frac{1}{2} s (\varphi, -\Delta \varphi) } e^{-\frac{1}{2} (\varphi, C(s,m^2)^{-1} \varphi)} d\varphi
\end{align}
where the second line follows simply by subtracting and adding $\frac{1}{2} s (\varphi, -\Delta \varphi)$ on $C(s,m^2)^{-1}$.
Plugging this into \eqref{eq:massive_model_defi} and performing the discrete sum, 
we arrive at a
new statistical physics model on configuration space $\R^{\Lambda_N}$:
\cite[(2.30)--(2.31)]{dgauss1} states that
\begin{align}
	\sum_{\sigma \in \Z^{\Lambda}_{\beta}} e^{-\frac{1}{2} (\sigma,  (-\Delta +m^2) \sigma ) } e^{\cz(\f, \sigma)} 
	\propto
	e^{\frac{\gamma}{2} \cz^2 (\f,\f) } \E^{\varphi}_{C(s,m^2)} \big[ e^{\cz (\f, \varphi)} e^{\frac{1}{2} s (\varphi, -\Delta \varphi) + \sum_{x\in \Lambda} U (\varphi (x) + \gamma \cz \f (x))} \big]
	\label{eq:reformulation_inter1}
\end{align}
where 
${U} (\theta) = \sum_{q=1}^{\infty} z_0 (\beta) \cos (q \beta^{1/2} \theta)$ for any $\theta \in \C$ and some $z_0^{(q)} \in \R$ with
\begin{align}
	| z_0 (\beta) | = O( e^{- \frac{1}{4} \gamma \beta (1+q) } ) , 
	\qquad q \in \N,  \; \beta >0
	.
\end{align}

After adding and subtracting $g (\varphi) := \frac{1}{2} s (\varphi + \gamma \cz \f ,  (-\Delta) (\varphi + \gamma \cz \f ) ) - \frac{1}{2} s (\varphi, - \Delta \varphi)$ in the exponent of the right-hand side of \eqref{eq:reformulation_inter1}, we obtain 
\begin{align}
	\sum_{\sigma \in \Z^{\Lambda}_{\beta}} e^{-\frac{1}{2} (\sigma,  (-\Delta +m^2) \sigma ) } e^{\cz(\f, \sigma)} 
	& \propto
	e^{\frac{\gamma}{2} \cz^2 (\f,\f)} \E^{\varphi} \big[ e^{\cz (\f,\varphi)} e^{-g (\varphi)} e^{\frac{1}{2} s (\varphi + \gamma \cz \f ,  (-\Delta) (\varphi + \gamma \cz \f ) ) +  \sum_{x\in \Lambda} {U} (\varphi (x) + \gamma \cz \f (x)) }   \big] \nnb
	& = e^{\frac{\gamma}{2} \cz^2 (\f,\f)} \E^{\varphi} \big[ e^{\cz (\f,\varphi) - g (\varphi)} Z_0 (\varphi + \gamma \cz f)  \big]
\end{align}
with $\varphi \sim \cN (0,C(s,m^2))$,
where the second line follows from \eqref{eq:Z_0_definition}.
Also, we can expand out $g (\varphi)$ to see that
\begin{align}
	\frac{\gamma}{2} \cz^2 (\f,\f) + \cz (\f,\varphi) + g (\varphi) 
		= \frac{\gamma}{2} \cz^2 (\f,  \tilde{\f}) + \cz (\tilde\f,  \varphi )
		,
\end{align}
which gives the desired conclusion.
\end{proof}

\begin{remark}[Sine-Gordon model]
\label{remark:sg_also_works}

At high temperature,
the procedure that shows \eqref{eq:reformulation_inter1} can also be applied to show a similar statement on a generalised class of models of sine-Gordon type defined by
\begin{align}
& \P^{\gsg}_{\beta,m^2} (d \phi) = \frac{1}{Z^{\gsg}_{\beta,m^2}} e^{- H^{\gsg}_{\beta,m^2} (\phi)  } d\phi, 
\\
& H^{\gsg}_{m^2} (\phi) = \frac{1}{2 \beta} (\phi,  (-\Delta + m^2) \phi) + \log \prod_{x\in \Lambda_N} \sum_{q\geq 0} \lambda_q \cos(q \phi(x) )
\end{align}
where $\lambda_q \in \R$ satisfies growth condition $|\lambda_q | \leq C e^{ (\eta + \beta \theta ) q^2 } \lambda_0$ for $\theta < \frac{1}{2} \gamma $ and when the sum $\sum_{q\geq 0}(\cdots)$ makes sense as a distribution.
A similar growth condition appeared in \cite{MR634447},
and the DG model can be considered to be the case $\lambda_0 = 1$, and $\lambda_q = 2$ for $q\neq 0$.

Now we explain why the growth condition $|\lambda_q| \leq C e^{(\eta + \beta \theta) q^2} \lambda_0$ with $\theta < \frac{\gamma}{2}$ is sufficient. 
Indeed,
\begin{align}
	\frac{1}{\sqrt{2\pi \gamma \beta}} \int_{\R} \lambda_q \cos(q x) e^{-\frac{1}{2\gamma \beta} (y-x)^2 } dx = \lambda_q e^{-\frac{1}{2} \beta \gamma q^2} \cos(q y ) 
	= 
	O \Big( e^{- \big( \frac{1}{2} \beta \gamma - \beta \theta  - \eta \big) q^2} \Big) \lambda_0 \cos(qy).
\end{align}
If we consider the unital Banach algebra of periodic functions
\begin{align}
	& \hat{\ell}^1 (c) = \{ f \in L^{\infty} (\R) : f(x + 2\pi) = f(x), \;\; \norm{f}_{\hat{\ell}^1 (c)} < \infty  \}, \\
	& \norm{f}_{\hat{\ell}^1 (c)} = \sum_{q \in \Z} e^{c |q|} |\hat{f} (q)|, \qquad \hat{f} (q) = \frac{1}{2\pi} \int_0^{2\pi} f(x) e^{-iq x} dx, 
\end{align}
then $f \mapsto \log (1 + f)$ is a well-defined continuously differentiable function $\{ \norm{f}_{\hat{\ell}^1 (c)} <  1 \} \rightarrow \hat{\ell} (c)$. Thus with sufficiently large $\beta$ and the choice $c = \beta(\gamma/2 - \theta)$, we see that
\begin{align}
\log \Big( \lambda_0 + \frac{1}{\sqrt{2\pi \gamma}} \int_{\R} \sum_{q \geq 1} \lambda_q \cos(\beta^{1/2} q x) e^{-\frac{1}{2\gamma} (y-x)^2 } dx \Big)  = \log \lambda_0 + \tilde{U} (y)
\end{align}
where
\begin{align}
\tilde{U} (y) = \sum_{q=0}^{\infty} \tilde{z}^{(q)} \cos( \beta^{1/2} q y), \qquad |\tilde{z}^{(q)} (\beta)| \leq c_1  e^{-c_2 (1+q) \beta}
\end{align}
for some $c_1, c_2 >0$ that depend on $\eta, \theta, \gamma$.  (See \cite[Lemma	~2.2]{dgauss1} for a similar argument.)
Along with the covariance decomposition $(-\Delta + m^2)^{-1} = C(m^2) + \gamma \operatorname{id}$,  this computation shows that
\begin{align}
	Z^{\gsg}_{\beta,m^2} 
	& \propto \int_{ (\R^{\Lambda_N} )^2 } e^{-\frac{1}{2} (\varphi, C(m^2)^{-1} \varphi) - \frac{1}{2 \gamma} |\zeta|^2} \prod_{x\in \Lambda_N} \sum_{q\geq 0} \lambda_q \cos (\beta^{1/2} q (\varphi(x) + \zeta(x) )) d\zeta d\varphi \nnb
	& \propto 
	\int_{\R^{\Lambda_N}} e^{-\frac{1}{2} (\varphi, C(m^2)^{-1} \varphi)} e^{\sum_x \tilde{U} (\varphi(x)) } d\varphi
	\nnb
	& \propto 
	\int_{\R^{\Lambda_N}} e^{-\frac{1}{2} (\varphi, C(s,m^2)^{-1} \varphi)} e^{\frac{1}{2} s(\varphi, (-\Delta) \varphi) +  \sum_x \tilde{U} (\varphi(x)) } d\varphi. 
\end{align}
The same procedure can also be applied to replace the moment generating function of \eqref{eq:reformulation_inter1}.
This can be used to prove a version of Lemma~\ref{lemma:first_reformulation_Z2} with ${z}^{(q)}$ replaced by $\tilde{z}^{(q)}$ when $\beta$ is sufficiently large.
\end{remark}

As a special case, one could think of the usual lattice sine-Gordon with any activity $z\in \R$,
i.e., the measure
\begin{align}
	\P^{\sg}_{\beta,m^2} (d\phi) = \frac{1}{Z^{\sg}_{\beta,m^2}} e^{-H^{\sg}_{\beta,m^2} (\phi)} d\phi, \qquad \phi \in \R^{\Lambda_N}
\end{align}
with Hamiltonian $H^{\sg}_{\beta,m^2} (\phi) = \frac{1}{2 \beta} (\phi,  (-\Delta + m^2) \phi) +  \sum_{x\in \Lambda_N} z \cos(\phi(x) )$.

\begin{corollary}
$\langle e^{\cz (\f,\phi)} \rangle_{\beta,m^2}^{\sg}$ satisfies the statement of Lemma~\ref{lemma:first_reformulation_Z2} for sufficiently large $\beta$.
\end{corollary}
\begin{proof}
Due to the previous remark, 
it is enough to check that the lattice sine-Gordon model is a generalised sine-Gordon model.
If we expand out the periodic part of the sine-Gordon Hamiltonian,
we obtain
\begin{align}
e^{z \cos(x)} 
=  \sum_{n=0}^{\infty} \sum_{k=0}^{n} \frac{(z/2)^n}{k! (n-k) !} \Re[ e^{i(2k-n)x} ]
& = \sum_{k=0}^{\infty} \sum_{q = -k}^{\infty} \frac{(z/2)^{2k + q}}{k! (k +q) !} \Re[ e^{i q x} ] \nnb
&= \sum_{k=0}^{\infty} \sum_{q=0}^{\infty} I(q,k) \frac{(z/2)^{2k + q}}{k! (k+q) !} \cos(q x)
\end{align}
where reparametrisation $n-2k =q$ was made in the second equality
and $I(q,k) = 2$ if $1 \leq q \leq k$ and $I(q,k) =1$ otherwise.
By making a trivial bound $\sum_{k < \infty} I(q,k)  \frac{q!}{k ! (k+q)!} (z/2)^{2k+q}  \leq 2 (z/2)^q e^{z^2/4}$, we see that the lattice sine-Gordon model with any activity $z\in \R$ satisfies the condition above. 
\end{proof}

\subsection{Finite range decomposition}
\label{sec:frd_brief}

We study the expectations in Lemma~\ref{lemma:first_reformulation_Z2} using a multi-scale analysis. 
The Gaussian field $\varphi \sim \cN (0, C(s,m^2))$ is decomposed into independent Gaussian random variables residing in different scales, and this is equivalent to the decomposition of the covariance $C(s,m^2)$. 
\cite[Section~3]{dgauss1} constructs a constant $\epsilon_s >0 $ and a collection of (rotation, reflection and translation invariant) covariance matrices $(\Gamma_{j} (m^2) )_{j\geq 0}$ and $\Gamma_N^{\Lambda_N} (m^2)$,  and $t_N (m^2) \in \R_{> 0}$ such that, on $\Lambda_N$,
\begin{align}
C  (s,m^2) = \sum_{j=1}^{N-1} \Gamma_{j} (m^2) + \Gamma_{N}^{\Lambda_N} (m^2) + t_N (m^2) Q_N
\label{eq:frd}
\end{align}
whenever $s \in [-\epsilon_s, \epsilon_s]$ and $m^2 \in (0,1]$
such that satisfy the following.
\begin{align} 
	& \parbox{\dimexpr\linewidth-4em}{%
	(1) $\Gamma_j (x,y) = 0$ whenever $\operatorname{dist}_{\infty} (x,y) \geq \frac{1}{4} L^j$ (giving the name \emph{finite range decomposition})
	and, 
	for any $n \in \N$ and multi-index $\vec{\mu} = (\mu_1, \cdots, \mu_n) \in \hat{e}^n = \{ \pm e_1, \pm e_2 \}^n$,
	}
	\nnb
	& \qquad \quad \sup_{x, y \in \Lambda_N} | \nabla^{\vec{\mu}} \Gamma_{j} (m^2) (x,y) | 
	\leq
	\begin{array}{ll}
	\begin{cases}
	C_{n} L^{-n (j - 1)} & \text{if} \;\; |\vec{\mu}| = n > 0 \\
	C_0 \log L & \text{if} \;\; |\vec{\mu}| = 0
	.
	\end{cases}
	\end{array} \label{eq:Gamma_j_decay}
	\\
	& \parbox{\dimexpr\linewidth-4em}{%
	The same bound holds for $\Gamma_N^{\Lambda_N}$ and the bounds are uniform in $m^2$.
	}
	\stepcounter{equation} \tag{\theequation $\Gammaone$} \label{quote:Gamma_one}
\end{align}
\begin{equation} \stepcounter{equation}
	\tag{\theequation $\Gammatwo$} \label{quote:Gamma_two}
	\parbox{\dimexpr\linewidth-4em}{%
	(2) $\Gamma_j (0,x)$'s are independent of $\Lambda_N$ if $j \leq N-1$. Note that this makes sense only because of \eqref{quote:Gamma_one}.
	}
\end{equation}
\begin{align} 
	& \parbox{\dimexpr\linewidth-4em}{%
	(3) $\Gamma_j (m^2)$ (respectively $\Gamma_N^{\Lambda_N} (m^2)$) is continuous in $m^2 \in (0,1]$ and attains limit $\Gamma_j (0) \equiv \lim_{m^2 \downarrow 0} \Gamma_j (m^2)$ (respectively $\Gamma_N^{\Lambda_N} (0) \equiv \lim_{m^2 \downarrow 0} \Gamma_N^{\Lambda_N} (m^2)$) that satisfies
	}
	\nnb
	& \qquad\quad \Gamma_{j} (m^2 = 0) (0,0) = \frac{1}{2\pi (1+s)} \big( \log L  + O(1) L^{-(j-1)} \big) 
	,
	\qquad j\leq N-1
	\\
	& \parbox{\dimexpr\linewidth-4em}{%
	with $O(1)$ independent of $j,L$
	(but this is not necessarily true for $\Gamma_N^{\Lambda_N} (0)$).
	}
	\stepcounter{equation} \tag{\theequation $\Gammathree$} \label{quote:Gamma_three}
\end{align}
\begin{equation} \stepcounter{equation}
	\tag{\theequation $\Gammafour$} \label{quote:Gamma_four}
	\parbox{\dimexpr\linewidth-4em}{%
	(4) $t_N (s,m^2) \in ( \max\{ m^{-2} - C L^{2N} ,  0 \} ,  m^{-2})$ for some $C >0$ and $Q_N$ is the matrix with all entries equal to $L^{-2N}$.
	}
\end{equation}

One immediate consequence of \eqref{quote:Gamma_three} is that the limit $m^2 \downarrow 0$ only needs to be taken carefully when integrating against the covariance $t_N (m^2) Q_N$.
Indeed,  by this decomposition,  the expectation $\E_{C(s,m^2)}$ can be decomposed as
\begin{align}
	\E^{\varphi}_{C(s,m^2)} [F (\varphi)] 
		= \E^{\phi^{(m^2)}} \E^{\zeta^{(m^2)}} \Big[ F \big( \phi^{(m^2)} + \zeta^{(m^2)} \big) \Big]
	\label{eq:E_Z_0_decomposition}
\end{align} 
where $\zeta^{(m^2)}$ and $\phi^{(m^2)}$ are independent Gaussian random variables with
\begin{align}
	\phi^{(m^2)} \sim \cN (0,t_N Q_N),  \qquad
	\zeta^{(m^2)} \sim \cN \big(0,  {\textstyle \sum_{j=1}^{N-1} \Gamma_j + \tilde{\Gamma}_N^{\Lambda_N}  } \big)	
\end{align}
so that $\zeta^{(m^2)} + \phi^{(m^2)} \sim \cN(0, C(s,m^2))$.
But since $\Gamma_j (m^2)$'s and $\Gamma_N^{\Lambda_N} (m^2)$ converge as $m^2 \downarrow 0$,  we can also take the same limit for $\zeta^{(m^2)}$ in \eqref{eq:E_Z_0_decomposition}.
(Certain integrability condition has to be verified to prove this limit, and taking $s$ in the definition of $Z_0$ sufficiently small will do this job.)
Hence the moment generating function will have its final alternative expression,
whereafter we always take $m^2 = 0$ in $\Gamma_j$ and $\Gamma_N^{\Lambda_N}$.

In the following, we denote the tilted measure (but \emph{not} necessarily a probability measure) by
\begin{align}
	\Et [F(\varphi)] = \frac{\E[ e^{\tau ( \tilde{\f},\varphi)} F(\varphi)]}{\E[ e^{\tau (\tilde{\f},\varphi)} ]},
	\qquad \tau \in \C
	\label{eq:Etau_definition}
\end{align}
for a random variable $\varphi$ on $\R^{\Lambda_N}$ (or $\R^{\Z^2}$) whenever the integrands are integrable.

\begin{proposition}
\label{prop:final_alternative_form_of_mgf}

For $\gamma > 0$ and $|s|$ sufficiently small,  $\f$ as in \eqref{quote:assumpfa},  $\tilde{\f} = ( 1 + s \gamma \Delta ) \f$,  $\beta >0$, $\rr \in \R$, $\tau \in \C$, $\cz = \rr + \tau$,  and $\Et$ defined by \eqref{eq:Etau_definition}, 
\begin{align}
	\langle e^{\beta^{-1/2} \cz (\f, \sigma)}  \rangle^{\dg}_{\beta, \Lambda_N} =  e^{\frac{1}{2} \cz^2 (\tilde{\f}, \tilde{C} (s)\tilde{\f}) + \frac{\gamma}{2} \cz^2 (\f, \tilde{\f})}
	\lim_{m^2 \downarrow 0} F_{N,m^2} [\f] (\rr, \tau)
	\label{eq:final_alternative_form_of_mgf}
\end{align}
where
\begin{align}
	 F_{N,m^2} [\f] (\rr, \tau) = \frac{\E^{\phi^{(m^2)}} \Et^{\bar{\zeta}} [ Z_0 (\phi^{(m^2)} + \bar{\zeta} + \cz \gamma \f + \rr {\tilde{C}(s)} \tilde{\f})] }{\E^{\phi^{(m^2)}} \E^{\bar{\zeta}} [ Z_0 (\phi^{(m^2)} + \bar{\zeta} )]}
	 \label{eq:F_N_m^2_definition}
\end{align}
with $\tilde{C} (s) = \sum_{j=1}^{N-1} \Gamma_{j}(m^2 = 0) + \Gamma_N^{\Lambda_N} (m^2 = 0)$,  $\phi^{(m^2)} \sim \cN(0, t_N (m^2) Q_N)$ and $\bar{\zeta} \sim \cN (0, \tilde{C} (s))$.
\end{proposition}
\begin{proof}
We start from \eqref{eq:m^2_limit_justification} and \eqref{eq:reformulation-Z2-intermediate},  saying
\begin{align}
	\langle e^{\beta^{-1/2} \cz (\f, \sigma)} \rangle^{\dg}_{\beta, \Lambda_N} 
	= \lim_{m^2 \downarrow 0} \langle e^{\cz (\f, \sigma)} \rangle_{\beta, m^2, \Lambda_N}
	= \lim_{m^2 \downarrow 0} e^{\frac{\gamma}{2} \cz^2 (\f,  \tilde{\f})} \frac{ \E_{C(s,m^2)} \big[ e^{\cz (\tilde{\f} , \varphi)} Z_0(\varphi + \gamma \cz \f) \big] }{ \E_{C(s,m^2)} \big[ Z_0(\varphi) \big] }
	\label{eq:final_alternative_form_of_mgf-interm1}
\end{align}
Since we have decomposition $\varphi = \phi^{(m^2)} + \zeta^{(m^2)}$ for some independent Gaussian random variables $\phi^{(m^2)} \sim \cN(0, t_N Q_N)$, $\zeta^{(m^2)} \sim \cN(0,  \sum_{j=1}^{N-1} \Gamma_j (m^2) + \Gamma_N^{\Lambda_N} (m^2))$,
and since $t_N Q_N \tilde{\f} \equiv 0$, if we can justify
\begin{align}
	\lim_{m^2 \downarrow 0} 
 \frac{\E^{\phi^{(m^2)}} \E^{\zeta^{(m^2)}} [ e^{\cz (\tilde{\f}, \zeta^{(m^2)} )} Z_0 (\phi^{(m^2)} + \zeta^{(m^2)} + \gamma \cz \f)] }
{\E^{\phi^{(m^2)}} \E^{\zeta^{(m^2)}} [ Z_0 (\phi^{(m^2)} + \zeta^{(m^2)})]} 
	=
	\lim_{m^2 \downarrow 0}  
	\frac{\E^{\phi^{(m^2)}} \E^{\bar{\zeta}} [ e^{\cz (\tilde{\f}, \bar{\zeta})} Z_0 (\phi^{(m^2)} + \bar{\zeta} + \gamma \cz \f )] }
	{\E^{\phi^{(m^2)}} \E^{\bar{\zeta}} [ Z_0 (\phi^{(m^2)} + \bar{\zeta})]} 
	\label{eq:final_alternative_form_of_mgf-mass_to_0_limit}
\end{align}
(for $\bar{\zeta} \sim \cN(0, \tilde{C} (s))$),
then we are done by a change of variables $\zeta \rightarrow \zeta + \rr \tilde{C}(s) \tilde{\f}$. 
But \eqref{eq:final_alternative_form_of_mgf-mass_to_0_limit} is implied by \eqref{quote:Gamma_three}.
\end{proof}

Having the DG regularised,  now $F_{N,m^2}[\f] (\rr, \tau)$ can be defined independent of the assumptions $\sum_{x\in \Lambda_N} \f (x) = 0$, although the limit $m^2 \downarrow 0$ cannot be justified in this case. 
However, the case $\sum_{x\in \Lambda_N} \f (x) \neq 0$ will also be used crucially in the course of our proof, and this motivates the utility of \eqref{quote:assumpf}.

\begin{remark}  \label{remark:shifted_environment_warning}
Even if have assumed that $F$ is an analytic function with nice decay properties so that 
\begin{align}
	\E_{(\tau), \tilde{C} (s)} [ F(\bar{\zeta} )] = \E_{\tilde{C} (s)} [ F(\bar{\zeta} + \tau \tilde{C} (s) \tilde{\f} )]
\end{align}
holds, the size of $\tilde{C}(s) \tilde{\f}$ diverges logarithmically in $|\Lambda_N|$.  We will see in later sections that such a divergence of the complex shift is not preferable in the technical point of view.
Thus we use an alternative approach in the following sections, where we also make multi-scale decomposition of the complex shifts.
\end{remark}

\subsection{Choice of parameters}
\label{sec:choice_of_parameters}

The RG map constructed in \cite{dgauss1} is sensitive to the choice of parameters that will appear in the following sections. 
Since the current work also depends on this RG map, 
the choice of these parameters will also play non-negligible role.
But this will not be at the centre of our discussion, 
so we will try to simplify our account on this problem in later chapters
once we define these parameters in this section. 

\begin{itemize}
\item $\gamma > 0$ is a constant related to the discrete Laplacian and $c_f = \frac{1}{4} \gamma$ is a constant related to the DG model. 
They are introduced in \cite[Section~3 and 7]{dgauss1}.

\item $L$ is a parameter defining the renormalisation group map. It is chosen $L \geq L_0 (M, \rho, \rr)$ where $L$ satisfies the assumptions \cite[Theorem~7.7]{dgauss1}, 
Lemma~\ref{lemma:extfield_bound},
Theorem~\ref{thm:local_part_of_K_j+1} and Proposition~\ref{prop:stability_of_dynamics}. 
Also it is restricted to take value of the form $L = \ell^{N'}$ for some $\ell$ and $N'$ both sufficiently large as in Lemma~\ref{lemma:G_change_of_scale_external_field}.

\item $A,c_2, c_4, c_w, c_{\kappa}, c_{h}, c_{\kt},\kappa, h, \htau >0$ appear in the definition of  norms introduced in Section~\ref{sec:polymer_activities_and_norms} and related definitions in Appendix~\ref{sec:norm_inequalities}.
We require $A \geq A_0 (L)$ to satisfy the assumptions of \cite[Theorem~7.7]{dgauss1}
and Proposition~\ref{prop:largeset_contraction_external_field}.
$c_2, c_4, c_w, c_{\kappa}, c_h$ are chosen according to \cite[Section~5-7]{dgauss1}, \cite[Section~3]{dgauss2} and Lemma~\ref{lemma:contraction_of_charge_q_term},
$\kappa= c_k / \log L$ and $h = \max\{ 1,  c_h \sqrt{\beta} \}$. 
$c_{\kt}$ and $\htau$ are fixed so that $\htau = c_{\kt} (\log L)^{-3/2}$ and satisfy Lemma~\ref{lemma:Et_F_bound}, Lemma~\ref{lemma:Et_by_complex_shift} and Lemma~\ref{lemma:F'_is_analytic}.

\item $\beta \geq \beta_0$ is chosen to satisfy the assumptions of \cite[Theorem~7.7]{dgauss1} and Proposition~\ref{prop:stability_of_dynamics}.

\item There are constants that might share the same symbol but differ from line to line,
e.g.,  $C$,  $C_1$,  $C_2$,  $C'$,  $C''$ and so on.
\end{itemize}

Some conditions are omitted in the list of references above. This is either because the condition is abundant or it is only a minor modification to the condition introduced in \cite{dgauss1,dgauss2}.
The main difference in the choice of parameters in this work with that of \cite{dgauss1} will be due to the introduction of new parameters $M, \rho, R$. 
However, this dependence might not always be apparent in the writing. 

We finally add a remark on $\gamma >0$.
The choice of $\gamma >0$ is related to the specific construction of the finite range decomposition in \cite[Section~3]{dgauss1}.
Its specific value does not play any important role as long as $C(s,m^2)$ defined by \eqref{eq:C_m^2_definition} admits decomposition of Section~\ref{sec:frd_brief}.

\ifx\newpageonoff\undefined
{\red command undefined!!}
\else
  \if\newpageonoff1
  \newpage
  \fi
\fi

\section{Overview of the proof}
\label{sec:overview_of_proof}

In the first half of this section, we prove the main theorems Theorem~\ref{thm:main_theorem} and Theorem~\ref{thm:main_theorem_generalised} assuming Proposition~\ref{prop:overview_of_the_proof}. 
In the second half of this section, we give a brief overview of how the renormalisation group argument is used to prove the proposition.

In view of Proposition~\ref{prop:final_alternative_form_of_mgf}, control of the moment generating function is just due to the control of the ratio $F_{N,m^2}$. 
Thus it is the objective of Proposition~\ref{prop:overview_of_the_proof} to show how $F_{N,m^2}$ is controlled. 
In the statement, the \emph{coalescence scale} $j_{0y}$ is used:
\begin{align}
	j_{0y} = \min \Big\{ j \geq 0 \; : \; (B^j_0)^{***} \cap Q_y^j  \neq \emptyset \Big\} 
	.
	\label{eq:joy_defi}
\end{align}
(Recall $Q^j_y =\cup_{i=1}^4 B^j_i$, and $(B_0^j)^{***}$ is taking the small set neighbourhood three times.)
The coalescence scale is roughly just $\log_L \operatorname{dist}_{\infty} (\supp (\f_{1}),  y+\supp ( \f_{2}) )$, but stated in the language of blocks.
If either $\f_{1} \equiv 0$ or $\f_{2} \equiv 0$,  we use convention $j_{0y} = \infty$
Also, for fixed $\f_1$ and $\f_2$, there exists $C >0$ such that
\begin{align}
	C^{-1} \norm{y}_2 \leq L^{ j_{0y}}  \leq C \norm{y}_2 
	\label{eq:charac_of_coalescence}
	.
\end{align}

\begin{proposition}
\label{prop:overview_of_the_proof}

Let $\f$ be as in \eqref{quote:assumpfa}. Then there exists $s_0^c (\beta) = O(e^{-c \beta})$ such that the following holds when $Z_0$ is defined by \eqref{eq:Z_0_definition} with $s = s_0^c (\beta)$:
under the assumptions of Theorem~\ref{thm:main_theorem_generalised},
if $\cz = \rr + \tau$ for some $\rr \in \R$ and $\tau \in \D_{\htau}$, then 
\begin{align}
	\lim_{m^2 \downarrow 0} F_{N,m^2} [\f] (\rr, \tau) =
	e^{\tilde{g}_{\infty}^{\rr}[\f_1] (\tau) + \tilde{g}_{\infty}^{\rr}[\f_2] (\tau) +  \sum_{j\geq j_{0y}}  \tilde{\kg}_{j}^{(2),\rr}[\f_1, T_y \f_2] (\tau)  } \big( 1 + \tilde{\psi}_{\Lambda_N}^{\rr} (\tau,y) \big)
\end{align}
where $\tilde{g}_{\infty}^{\rr} [\f_i] (\tau)$, $\tilde{\kg}_{j}^{(2),\rr} [\f_1, T_y \f_2] (\tau)$ and $\tilde{\psi}_{\Lambda_N}^{\rr} (\tau, y)$ are analytic functions in $\tau \in \D_{\htau}$ and satisfy
\begin{align}
	& \norm{ \tilde{\kg}_{j}^{(2),\rr} [\f_1, T_y \f_2]  }_{L^{\infty} (\D_{\htau})} \leq O(  L^{-\alpha j} ) ,
		\qquad  j \geq j_{0y},  
	\label{eq:overview_of_the_proof_1}
	\\
	& \norm{ {  \tilde{\psi}_{\Lambda_N}^{\rr}  (\cdot,  y)  } }_{L^{\infty} (\D_{\htau})} \leq O(  L^{-\alpha N} )
	\label{eq:overview_of_the_proof_2}
\end{align}
for some $\alpha >0$.
\end{proposition}

We add a short remark on $s_0^c (\beta)$. 
We import the choice of $s_0^c (\beta)$ from \cite{dgauss1}, 
where the family $(s_0^c (\beta) : \beta \ge \beta_0)$ is constructed as the stable manifold of a renormalisation group flow (that we will also recall in Section~\ref{sec:the_rg_method}).
Although it is not clear from the proof whether such value of $s_0^c (\beta)$ is unique or depends on $\gamma$, 
we know \textit{a posteriori} from Corollary~\ref{cor:CLT} that $s_0^c (\beta)$ is characterised by the limiting distribution of $\sigma_0 - \sigma_y$,  thus unique and independent of $\gamma$.

The proof will only appear in Section~\ref{sec:renormalisation-group-computation}.
We emphasize that each one-point free energy $\tilde{g}_{\infty}^{\rr} [\f_i]$ ($i=1,2$) is independent of choice of $y$. This is not an obvious consequence of the renormalisation group method, as it has strong dependence on the choice of the gridding,  i.e., the choice of configurations of $\cB_j$.

\subsection{Proof of the main theorems}

The main theorems are direct consequences of Proposition~\ref{prop:overview_of_the_proof}.

\begin{proof}[Proof of Theorem~\ref{thm:main_theorem_generalised}]

Let  $Z_0$ be defined by \eqref{eq:Z_0_definition} and $s = s_0^c (\beta)$ where $s_0^c (\beta)$ is as in Proposition~\ref{prop:overview_of_the_proof}.
Then by Proposition~\ref{prop:final_alternative_form_of_mgf} and Proposition~\ref{prop:overview_of_the_proof}, 
\begin{align}
	\log \langle e^{\beta^{-1/2} \cz (\f,\sigma)} \rangle^{\dg}_{\beta, \Lambda_N} & =
	\frac{1}{2} \cz^2 (\tilde{\f},  \tilde{C} (s_0^c (\beta)) \tilde{\f}) +  \frac{\gamma}{2} \cz^2 (\f, \tilde{\f}) \nnb
	&\qquad + \tilde{g}^{\rr}_{\infty} (\tau) [\f_1] + \tilde{g}^{\rr}_{\infty} (\tau) [\f_2] + \sum_{j=j_{0y} }^{\infty} \tilde{\kg}_j^{(2),\rr} [\f_1, T_y \f_2] (\tau)
	+ \log \big( 1+  \tilde{\psi}_{\Lambda_N}^{\rr} (\tau,y) \big) .
\end{align}
By \eqref{eq:overview_of_the_proof_1},\eqref{eq:overview_of_the_proof_2} and \eqref{eq:charac_of_coalescence}, they satisfy
\begin{align}
	\Big| \sum_{j=j_{0y} }^{\infty} \tilde{\kg}_j^{(2),\rr} [\f_1, T_y \f_2] \Big| \leq O(L^{-\alpha j_{0y}}) \leq O(\norm{y}_2^{-\alpha}), \qquad | \tilde{\psi}_{\Lambda_N}^{\rr} (\tau,y) | \leq O(L^{-\alpha N}) \leq O(|\Lambda_N|^{-\alpha})
	.
	\label{eq:main_thoerem_inter_bound}
\end{align}

If we let
\begin{align}
	H_N (\cz,y) & = \log \langle e^{\beta^{-1/2} \cz ( \f,\sigma)} \rangle^{\dg}_{\beta, \Lambda_N} - \frac{1}{2} \cz^2 (\tilde{\f},  \tilde{C} (s_0^c (\beta)) \tilde{\f}) -  \frac{\gamma}{2} \cz^2 (\f, \tilde{\f})
	\nnb
	& =  \tilde{g}^{\rr}_{\infty} [\f_1] (\tau)  + \tilde{g}^{\rr}_{\infty} [\f_2] (\tau)  + \sum_{j \geq j_{0y} } \tilde{\kg}_j^{(2),\rr} [\f_1, T_y \f_2] (\tau)
	+  \log (1 + \tilde{\psi}_{\Lambda_N}^{\rr} (\tau, y))  ,
\end{align}
then it is an analytic function of $\cz \in S_{R, \htau}$ (well-defined, because it is independent of the decomposition $\rr + \tau$). 
It also follows from taking limits $N, \norm{y}_2 \rightarrow \infty$ that
functions $h_{\beta}^{(1)}$, $h_{\beta}^{(2)}$ and $\psi_{\beta,N}$ defined by
\begin{align}
	& h^{(1)}_{\beta} [ \f_1 ] ( \cz ) + h^{(1)}_{\beta} [\f_2] (\cz) := \frac{\gamma}{2} \cz^2 \sum_{i=1,2} (\f_i, \tilde{\f}_i ) + \lim_{\norm{y}_2, N\rightarrow \infty} H_N (\cz, y) \label{eq:h_one_f_definition} \\
	& h^{(2)}_{\beta} [\f_1, \f_2] (\cz, y) := \frac{\gamma}{2} \cz^2  (\f_1 + T_y\f_2, \tilde{\f}_1 + T_y \tilde{\f}_2) - h^{(1)}_{\beta} [ \f_1 ] ( \cz ) - h^{(1)}_{\beta} [\f_2] (\cz)  + \lim_{N\rightarrow \infty} H_{N} (\cz, y) \\
	& \psi_{\beta, N} (\cz, y) := H_{N} (\cz, y) -  \lim_{N\rightarrow \infty} H_{N} (\cz, y)
\end{align}
are also analytic functions of $\cz$,  independent of the decomposition $\cz = \rr + \tau$.
\eqref{eq:h_one_f_definition} leaves infinitely many choices of each $h_{\beta}^{(1)} [\f_i] (\cz)$ ($i =1,2$), but we choose the branch where $h_{\beta}^{(1)} [\f_i] (\cz) =  \tilde{g}^{0}_{\infty} [\f_i] (\cz) + \frac{\gamma}{2} \cz^2 (\f_i, \tilde{\f}_i)$ for $\cz \in \D_{\htau}$.
Thus we obtain \eqref{eq:thm_main_theorem_generalised} with the above choices of $h_{\beta}^{(1)}, h_{\beta}^{(2)}, \psi_{\beta,N}$ and
\begin{align}
	\mathfrak{C}_{ \beta, \Lambda_N} = (1 + s_0^c (\beta) \gamma \Delta ) \tilde{C} (s_0^c (\beta)) (1 + s_0^c (\beta) \gamma \Delta )
	\label{eq:kC_definition}
 	.
\end{align}
The required estimates follow from \eqref{eq:main_thoerem_inter_bound}.
Also, for any $g = g_1 + T_y g_2$ that satisfies the assumptions of \eqref{quote:assumpf}, we have 
$| (g ,  \Gamma_j g) | = O\big( L^{-j} |y| \big) $ 
and 
$| (g, \tilde{\Gamma}_N^{\Lambda_N} g)  | = O\big( L^{-N} |y| \big)$ 
by \eqref{quote:Gamma_one}, so
$\lim_{N\rightarrow \infty} (g, \mathfrak{C}_{\beta, \Lambda_N} g)$ absolutely converges. 
Hence, if we let
\begin{align}
	\mathfrak{C}_{\beta, \Z^2} = \sum_{j=1}^{\infty} (1 + s_0^c (\beta) \gamma \Delta ) \Gamma_j (1 + s_0^c (\beta) \gamma \Delta )
	\label{eq:kC_Z2}
	,
\end{align}
then $(\f, \mathfrak{C}_{\beta, \Z^2} \f)$ is well-defined and equals $\lim_{N\rightarrow \infty} (\f, \mathfrak{C}_{\beta, \Lambda_N} \f)$.

\end{proof}

For the proof of Theorem~\ref{thm:main_theorem},
we will work in the Fourier space to compute the limit of $(f_y, \mathfrak{C}_{\Lambda_N,\beta} f_y)$ as $N \rightarrow \infty$.
As a preliminary step, 
we obtain the following lemmas.

\begin{lemma}
\label{lemma:main_thm_lemma1}
Let $f_y = \delta_0 - \delta_y$.
With $\mathfrak{C}_{\beta,\Z^2}$ as in \eqref{eq:kC_Z2},
\begin{align}
	(f_y, \mathfrak{C}_{\beta,\Z^2} f_y)
		= \frac{1}{4\pi^2} \int_{[-\pi, \pi]^2} dp \, | 1- e^{-i y \cdot p} |^2 |1- s_0^c (\beta) \gamma \lambda(p)|^2  \frac{ \lambda(p)^{-1} - \gamma }{1 +  s_0^c (\beta) \lambda(p) ( \lambda(p)^{-1} -\gamma )  } 
\end{align}
where $\lambda(p) = 4 - 2\cos (p_1) - 2\cos(p_2)$
is the Fourier transformation of $-\Delta$.
\end{lemma}
\begin{proof}
For brevity, we just denote $s = s_0^c (\beta)$.
Since $C(s,m^2)$, $\tilde{C} (s)$ and $\mathfrak{C}_{\beta, \Lambda_N}$ are translation invariant operations, 
we can compute their Fourier transforms.
We denote the Fourier transforms by applying hat($\hat{\cdot}$).
Recalling \eqref{eq:C_m^2_definition} for $C(s,m^2)$,
Proposition~\ref{prop:final_alternative_form_of_mgf} for $\tilde{C} (s) = \lim_{m^2 \downarrow 0} C(s, m^2) - t_N (m^2) Q_N$
and \eqref{eq:kC_definition} for $\mathfrak{C}_{\beta,\Lambda_N}$,
we see in order the Fourier transforms are
\begin{align}
	\hat{C} (s,m^2) (p) 
		&= \Big( \frac{1}{(\lambda (p) + m^2 )^{-1} - \gamma} + s \lambda (p) \Big)^{-1} \\
	\hat{C} (s_0^c) (p)
		&= \lim_{m^2 \downarrow 0} \Big(  \frac{(\lambda(p) +m^2)^{-1} - \gamma }{1 +  s \lambda(p) ( (\lambda(p) +m^2)^{-1} -\gamma )  } - t_N \delta_{0} (p) \Big) \\
	\hat{\mathfrak{C}}_{\beta,\Lambda_N} (p)
		&= |1- s \gamma \lambda(p)|^2  
	\lim_{m^2 \downarrow 0} \Big(  \frac{(\lambda(p) +m^2)^{-1} - \gamma }{1 +  s \lambda(p) ( (\lambda(p) +m^2)^{-1} -\gamma )  } - t_N \delta_{0} (p) \Big)
\end{align}
for $p \in \Lambda_N^* = 2\pi L^{-N} \Lambda_N$, the Fourier dual lattice of $\Lambda_N$.
Since $\hat{f_y} (p) = 1 - e^{-i y \cdot p}$,
we obtain
\begin{align}
	& (f_y, \mathfrak{C}_{\beta,\Lambda_N} f_y) \nnb
	&= \frac{1}{|\Lambda_N|} \sum_{p\in \Lambda_N^*} | 1- e^{-i y \cdot p} |^2 |1- s \gamma \lambda(p)|^2  
	\lim_{m^2 \downarrow 0}  \frac{(\lambda(p) +m^2)^{-1} - \gamma }{1 +  s \lambda(p) ( (\lambda(p) +m^2)^{-1} -\gamma )  } 
	.
\end{align}
We ignored $t_N \delta_{0} (p)$ since $1- e^{-i y \cdot p}$ vanishes at $p =0$.
In the limit $N \rightarrow \infty$,
we can simply replace the discrete sum by a 2-dimensional integral on $[-\pi , \pi )^2$,
thus
\begin{align}
\begin{split}
	& (f_y, \mathfrak{C}_{\beta,\Z^2} f_y) \\
		& \quad = \lim_{m^2 \downarrow 0} \frac{1}{4\pi^2} \int_{[-\pi, \pi]^2} dp \, | 1- e^{-i y \cdot p} |^2 |1- s \gamma \lambda(p)|^2  \frac{(\lambda(p) +m^2)^{-1} - \gamma }{1 +  s \lambda(p) ( (\lambda(p) +m^2)^{-1} -\gamma )  } 
	.
\end{split}
\end{align}
Finally, if we observe $| 1- e^{-i y \cdot p} |^2 = O(|y \cdot p|^2)$ and $\lambda(p) \ge c \norm{p}_2$ as $p \rightarrow 0$ and some $c >0$, 
we see that the limit $m^2 \downarrow 0$ can be exchanged with the integral,
giving the desired conclusion.
\end{proof}

\begin{lemma} \label{lemma:main_thm_lemma2}
Let $f_y = \delta_0 - \delta_y$.
Then for some constant $C_1$,
\begin{align}
	(f_y, \mathfrak{C}_{\beta,\Z^2} f_y) 
		= \frac{1}{1 + s_0^c (\beta)} (f_y, (-\Delta)^{-1} f_y) + C_1 + R(y)
\end{align}
where $R(y) = O(\norm{y}_2^{-2})$ as $\norm{y}_2 \rightarrow \infty$.
\end{lemma}
\begin{proof}
Again for brevity, denote $s = s_0^c (\beta)$.
Starting from Lemma~\ref{lemma:main_thm_lemma1},
the bound follows from a simple harmonic analysis.
To see this, first denote
\begin{align}
	{A} (p) = |1 - s \gamma \lambda(p) |^2 \frac{ \lambda(p)^{-1} - \gamma }{1 +  s \lambda(p) ( \lambda(p)^{-1} -\gamma )  } - \frac{1}{(1 + s) \lambda(p)} 
\end{align}
so that ${A} (p)$ is an analytic function invariant under $p \mapsto p + 2 \pi \vec{n}$ for $\vec{n} \in \Z^2$.
By Lemma~\ref{lemma:main_thm_lemma1},
\begin{align}
	(f_y, \mathfrak{C}_{\beta,\Z^2} f_y ) = \frac{1}{1 + s} (f_y,  (-\Delta)^{-1} f_y) + \frac{1}{4\pi^2} \int_{[-\pi,\pi]^2} dp \,  |1 - e^{-i y \cdot p}|^2 A(p)
	.
\end{align}
The constant part is simply given by $C_1 = \frac{1}{2\pi^2} \int_{[-\pi,\pi]^2} dp \,  {A} (p)$ and the rest is
\begin{align}
	R (y) = \frac{1}{4\pi^2} \int_{[-\pi,\pi]^2} dp \,  |1 - e^{-i y \cdot p}|^2 A(p) - C_1 = \frac{1}{4\pi^2} \int_{[-\pi,\pi]^2} dp \,   (e^{iy \cdot p} + e^{-iy \cdot p}) A(p)
	.
\end{align}
Since $-\Delta_p e^{iy \cdot p} = \norm{y}_2^2 e^{iy \cdot p}$ (where $-\Delta_p$ is the Laplacian acting on the $p$-variable),
\begin{align}
	\norm{y}_2^2 R(y)
		& = \frac{1}{4\pi^2} \int_{[-\pi,\pi]^2} dp \,   (-\Delta_p) (e^{iy \cdot p} + e^{-iy \cdot p}) A(p) \nnb
		& = \frac{1}{4\pi^2} \int_{[-\pi,\pi]^2} dp \,  (e^{iy \cdot p} + e^{-iy \cdot p})  (-\Delta_p) A(p)
\end{align}
where the second equality follows from integration by parts and the periodicity of $A(p)$.
This is bounded by $\norm{-\Delta_p A(p)}_{L^{\infty} ([-\pi,\pi]^2)}$, 
so we see that $R(y) = O(\norm{y}_2^{-2})$ as desired.
\end{proof}

The final input is a simple statement about the symmetry of $h_{\beta}^{(2)}$.

\begin{lemma} \label{lemma:main_thm_lemma3}
In the setting of Theorem~\ref{thm:main_theorem_generalised},
$h_{\beta}^{(2)} [\f_1, \f_2] (\cz, y)$ is an even function of $\cz$.
\end{lemma}
\begin{proof}
We apply Theorem~\ref{thm:main_theorem_generalised} with $\f = \f_1 + T_y \f_2$ and take $N\rightarrow \infty$.
Since the DG model is invariant under spin flip $\sigma \mapsto -\sigma$,
we see that $\langle e^{\beta^{-1/2} \cz (\f,\sigma)} \rangle_{\beta,\Z^2}^{\dg}$ is an even function of $\cz$.
If we insert $\f_2 =0$, 
we see that
\begin{align}
	h_{\beta}^{(1)} [\f_1] (\cz) = \log \langle e^{\beta^{-1/2} \cz (\f_1, \sigma)} \rangle_{\beta,\Z^2}^{\dg} - \frac{1}{2} \cz^2 (\f_1, \mathfrak{C}_{\beta,\Z^2} \f_1)
\end{align}
is also an even function and similarly, $h_{\beta}^{(2)} [\f_2] (\cz)$ is even.
Thus the remaining part $h_{\beta}^{(2)} [\f_1, \f_2] (\cz,y)$ is naturally also an even function of $\cz$.
\end{proof}

\begin{proof}[Proof of Theorem~\ref{thm:main_theorem}]

For $y \in \Z^2$, 
by Lemma~\ref{lemma:main_thm_lemma2} and \eqref{eq:mgf_gff},
\begin{align}
	(f_{y}, \mathfrak{C}_{\beta,\Z^2} f_{y}) 
	& = \frac{1}{1 + s_0^c (\beta)} \big( f_{y} ,  (-\Delta)^{-1} f_{y} \big) + C_1 + R(y) \nnb
	& = \frac{1}{\pi (1 + s_0^c (\beta))} \log \norm{y}_2 + C_2 + R_2 (y)
\end{align}
for some constants $C_1, C_2$ and $R_2 (y) = O(\norm{y}_2^{-\alpha'})$ for some $\alpha' >0$ as $\norm{y}_2 \rightarrow \infty$.
So Theorem~\ref{thm:main_theorem_generalised} with $\f = f_y = \delta_0 - \delta_y$ gives
\begin{align}
	\log \langle e^{\beta^{-1/2} \cz (\sigma(0)- \sigma(y))} \rangle_{\beta, \Z^2} & = C_3 (y) +  \frac{1}{2 \pi (1 + s_0^c (\beta))} \cz^2 \Big( \log \norm{y}_2 + C_4 + r_{\beta} (\cz,y)  \Big)
	\nnb
	& \quad + \big( h_{\beta}^{(1)} [\delta_0] (\cz) + h_{\beta}^{(1)} [-\delta_0] (\cz) - h_{\beta}^{(1)} [\delta_0] (0) - h_{\beta}^{(1)} [-\delta_0] (0 ) \big)
	\label{eq:main_theorem_conclusion}
\end{align}
for some function $C_3(y)$ and a constant $C_4$ where $C_3$ absorbed $\cz =0$ parts of $h_{\beta}^{(\alpha)}$'s and
\begin{align}
	\frac{r_{\beta} (\cz,y)}{2 \pi (1 + s_0^c (\beta))} = R_2 (y) + \frac{h_{\beta}^{(2)} [\delta_0, -\delta_0] (\cz,y) -  h_{\beta}^{(2)} [\delta_0, -\delta_0] (0,y) }{\cz^2} = O(\norm{y}_2^{-\alpha})
	,
\end{align}
which is again an analytic function of $\cz$ due to Lemma~\ref{lemma:main_thm_lemma3}.
But since the left-hand side of \eqref{eq:main_theorem_conclusion} vanishes when $\cz =0$,  we should have $C_3 (y) \equiv 0$.
We have the desired conclusion once we let
\begin{align}
	f_{\beta} (\cz, \infty) &=  \frac{C_4  \cz^2  }{2 \pi (1 + s_0^c (\beta))}  +  \big( h_{\beta}^{(1)} [\delta_0] (\cz) + h_{\beta}^{(1)} [-\delta_0] (\cz) - h_{\beta}^{(1)} [\delta_0] (0) - h_{\beta}^{(1)} [-\delta_0] (0 ) \big) \\
	f_{\beta} (\cz, y) &= \frac{\cz^2  }{2 \pi (1 + s_0^c (\beta))} r_{\beta} (\cz,y)  + f_{\beta} (\cz, \infty) .
\end{align}
\end{proof}

\subsection{Tilted expectation as shifted environment}

In the rest of the section, we give a brief overview of how Proposition~\ref{prop:overview_of_the_proof} is proved. 

As is explained in Remark~\ref{remark:shifted_environment_warning}, the complex shift will have to be dealt by making a multi-scale decomposition. Namely, we consider
\begin{align}
	\E^{\bar{\zeta}}_{(\tau),\tilde{C} (s)} [F(\bar{\zeta})] = \E^{\zeta_N}_{(\tau), \Gamma_N^{\Lambda_N}}  \E^{\zeta_{N-1}}_{(\tau),\Gamma_{N-1}} \cdots \E^{\zeta_1}_{(\tau), \Gamma_1} [F(\zeta_1 + \cdots + \zeta_N)]
	.
\end{align}
It will turn out in Lemma~\ref{lemma:Et_by_complex_shift} that, under mild conditions on some function $F_j$ analytic on a neighbourhood of the real space, that $\E^{\zeta_j}_{(\tau), \Gamma_{j + 1}} \big[  F_j ( \varphip + \zeta_j ) \big] = \E^{\zeta_j}_{\Gamma_{j + 1}} \big[  F_j ( \varphip + \zeta_j + \tau \Gamma_{j+1} \tilde{\f}  ) \big]$,
so the tilted expectation can be written as an ordinary expectation with shifted environment interchangeably. 
Hence we prepare a lemma that controls the complex shifts seen in each renormalisation group step.

\begin{definition} \label{def:extfield_def}
For $\f = \f_1 + T_y \f_2$ as in \eqref{quote:assumpf}, $s\in \R$
and $\alpha \in \{1,2 \}$,
let  $(u_{j, \alpha})_{j\geq 1}$ be defined by 
    \begin{align} \label{eq:extfield_def}
    u_{j,\alpha}  & =
    \begin{cases}
      \gamma \f_{\alpha} & \text{if} \;\;  j=0 \\
      \Gamma_{j} (1 +s\gamma\Delta )\f_{\alpha} & \text{if} \;\;  1 \leq j < N \\
      \Gamma_N^{\Lambda_N} (1 +s\gamma\Delta )\f_{\alpha} & \text{if} \;\;  j=N.
    \end{cases}
	\end{align}
	and $u_j = u_{j,1} + T_y u_{j,2}$.	
\end{definition}

Since $\tilde{C} (s) = \sum_{j=1}^{N-1} \Gamma_j + \Gamma_N^{\Lambda_N}$, 
if we let $\tilde{\f}_{\alpha} = (1+ s\gamma \Delta) \f_{\alpha}$, 
these definitions say
\begin{align}
{\textstyle \sum_{j=1}^{N} } u_{j, \alpha} =  \tilde{C} (s)   \tilde{\f}_{\alpha}, \quad \alpha \in \{1,2 \}. \label{eq:extfield_defining_property}
\end{align}
They also satisfy the following estimates,
where $M$ and $\rho$ are as in \eqref{quote:assumpfa}.

\begin{lemma} 
\label{lemma:extfield_bound}

If $L\geq 12(\rho +2)$,
then there exists $C_n  \equiv C_n (M,\rho) > 0$ for each $n\geq 0$ such that, 
for each multi-index $\vec{\mu} = (\mu_1, \cdots, \mu_n) \in \hat{e}^n$
and $\alpha \in \{1,2\}$,
\begin{align}
	\norm{\nabla^{\vec{\mu}} u_{j,\alpha}}_{L^{\infty}} \leq
	\begin{array}{ll}
	\begin{cases}
	C_0 M \rho^2  \log L & \text{if} \;\; n = |\vec{\mu} |=0 \\
	C_n M \rho^2 L^{-n  ( 0 \vee (j-1) )  }  & \text{if} \;\; n = |\vec{\mu} | \geq 1
	\end{cases}
	\end{array}
	\label{eq:u_j,alpha_bound}
\end{align}
and $\operatorname{supp} (u_{j,\alpha}) \subset B_0^j$, where $B_0^j$ is the unique $j$-scale block that contains $0$.
\end{lemma}

\begin{proof}
When $j=0$, the bounds are direct from the definition.  When $j \in \{1, \cdots, N-1\}$, for each $\alpha \in \{1,2\}$,
\begin{align}
	\norm{\nabla^{\vec{\mu}} \Gamma_j \f_{\alpha} }_{L^{\infty}} 
	= \sup_{x\in \Lambda_N} \big| \sum_{y\in \Lambda_N} \nabla^{\vec{\mu}} \Gamma_j (x-y,0) \f_{\alpha}(y) \big| 
	\leq M \rho^2 \norm{\nabla^{\vec{\mu}} \Gamma_j}_{L^{\infty}}
	,
\end{align}
while by \eqref{quote:Gamma_one}, 
$\norm{\nabla^{\vec{\mu}} \Gamma_j}_{L^{\infty}} \leq C L^{-n (j-1)}$ for $n\geq 1$ and $\norm{\Gamma_j}_{L^{\infty}} \leq C  \log L$ for $n=0$. 
The same method also bounds $\norm{ \nabla^{\vec{\mu}} \Gamma_j \Delta \f_{\alpha} }_{L^{\infty}}$, and these are enough to bound $\norm{\nabla^{\vec{\mu}} u_{j,\alpha}}_{L^{\infty}}$.
Finally, for $j = N$,  $\nabla^{\vec{\mu}} \Gamma_{N}^{\Lambda_N}$ satisfies bounds analogous to \eqref{eq:Gamma_j_decay},  so we have the desired bounds.

To see that $\operatorname{supp} (u_{j,\alpha}) \subset B_0^j$, 
recall $\Gamma_j (x,y) =0$ whenever $\operatorname{dist}_{\infty} (x,y) \geq \frac{1}{4} L^j$,  
so $u_{j,\alpha} (x) = 0$ whenever $\norm{x}_{\infty} \geq \rho + 2 + \frac{1}{4} L^j$. 
But by the definition of $B_0^j$, we have $\dist_{\infty} (0,  (B_0^j )^c) \geq \frac{L^j}{3}$, so $u_{j,\alpha} (x) = 0$ for any $x \not\in B_0^j$ and whenever $L \geq 12(\rho +2)$.
\end{proof}

Although the lemma is proved for general $|\vec{\mu}| = n \geq 0$, we actually only need this result for $n \leq 2$. 
That is, we only use the fact that
\begin{align}
	\norm{u_{j,\alpha} }_{C_j^2} 
	= \max_{n=0,1,2} \sup_{\vec{\mu} \in \hat{e}^n}  L^{nj}  \norm{\nabla^{\vec{\mu}} u_{j,\alpha} }_{L^{\infty}} 
\end{align}
is bounded by a uniform constant.
The following assumption on $(u_j)_{\geq 0}$'s can be used as well as \eqref{quote:assumpf} when $\f_j$'s 
are not mentioned directly. 

\begin{equation} \stepcounter{equation}
	\tag{\theequation $\assumpu$} \label{quote:assumpu}
	\parbox{\dimexpr\linewidth-4em}{%
	Let $M, \rho >0$.
	The sequence $(u_j)_{j\geq 1}$ has decomposition $u_j = u_{j, 1} + T_y u_{j, 2}$ where $u_{j,1}$ and $u_{j, 2}$ satisfy the following:
	for $\alpha \in \{1,2\}$,
	there exists $C > 0$
	such that $\norm{u_{j, \alpha}}_{C_{0 \vee (j-1)}^2} \leq C M \rho^2 \log L$ for each $j\leq N$,
	and $u_{j, \alpha}$ is supported on the unique $B^j_0 \in \cB_j$ 
	such that $0\in B^j_0$.
	}
\end{equation}

Note that $\supp (u_{j,2}) \subset B_0^j$ also implies
\begin{align}
	\supp (T_y u_{j,2}) \subset Q_y^j .
\end{align}
Thus by the definition of the coalescence scale, \eqref{eq:joy_defi}, 
we have $\supp (u_{j,1}) \, \cap \, \supp (T_y u_{j,2}) = \emptyset$ for $j < j_{0y}$.

\subsection{Effective potentials}
\label{sec:effpt}

An important implication of the covariance decomposition \eqref{eq:frd} is that it can be used to define the RG flow. 
Given $0^{th}$ scale partition function function $Z_0 : \C^{\Lambda_N} \rightarrow \C$, \eqref{eq:Z_0_definition}, 
the partition functions of scale $j$ are $Z^{\rr}_j : \R^{\Lambda_N} \times \D_{\htau} \rightarrow \C$ defined recursively by
\begin{align}
	Z_0^{\rr} (\varphi ; \tau) = Z_0 (\varphi + (\rr + \tau) u_0),
	\qquad \varphi \in \R^{\Lambda_N}
\end{align}
and for $\varphip \in \R^{\Lambda_N}$,
\begin{align}
& Z_{j+1}^{\rr} (\varphip ; \tau) 
= \begin{array}{ll}
\begin{cases}
\E_{\Gamma_{j+1}, (\rr + \tau)} [Z^{\rr}_j (\varphip + \zeta ; \tau)] 
& \text{if} \;\;  0 \leq j \leq N-2 \\
\E_{\Gamma_{N}^{\Lambda_N}, (\rr + \tau)} [Z^{\rr}_{N-1} (\varphip + \zeta  ; \tau)] 
& \text{if} \;\;  j  = N-1 \\
\end{cases}
\end{array}
\label{eq:bulk_RG_flow}
\end{align}
where the tilted expectations acts on $\zeta$.
Recall that $h_j := \log Z_j^{0} (\cdot ; 0)$ was called the effective potential, as it represents the renormalised theory at scale $j$.
Thus $\log Z_j^{\rr} (\cdot ; \tau)$ can be thought as a perturbation of the effective potential.
Note that
\begin{align}
	\E_{\Gamma_{j+1}, (\rr + \tau)} [Z_j^{\rr} (\varphip + \zeta ; \tau )] =	\E_{\Gamma_{j+1}, (\tau)} [Z_j^{\rr} (\varphip + \zeta + \rr u_{j+1} ; \tau )],
\end{align}
and similar holds for $\E_{\Gamma_N^{\Lambda} ,(\rr+ \tau)} [\cdots]$,
so by \eqref{eq:extfield_defining_property},
\begin{align}
& \E^{\zeta}_{ \tilde{C}^{\Lambda_N} (s), (\tau)} \big[ Z_0 (\varphip + \zeta + \cz \gamma \f + \rr {\tilde{C}(s)} \tilde{\f}) \big] = Z^{\rr}_N  (\varphip ; \tau) 
\label{eq:bulk_rg_final1}
\\
& \E^{\zeta}_{ \tilde{C}^{\Lambda_N} (s)} \big[ Z_0 (\varphip + \zeta ) \big]= Z^{0}_N (\varphip ; 0) .
\label{eq:bulk_rg_final2}
\end{align}
Thus $F_{N,m^2}$ of Proposition~\ref{prop:overview_of_the_proof} and $Z_N^{\rr}$ are related by
\begin{align}
	F_{N,m^2} [\f] (\rr,\tau) = \frac{ \E^{\varphi}_{t_N Q_N} Z_N^{\rr} (\varphi ; \tau) }{\E^{\varphi}_{t_N Q_N} Z_N^0 (\varphi ; 0)}
	.
\end{align}

\subsection{Renormalisation group coordinates}
\label{sec:Renormalisation_group_coordinates}

We seek for a representation of $Z^{\rr}_j (\varphi ; \tau)$ given by $j$-scale polymer expansion
\begin{align}
	Z^{\rr}_j (\varphi ; \tau) 
	= e^{-E_j |\Lambda_N | + g_j^{\rr} (\Lambda_N ; \tau) } \sum_{X\in \cP_j} e^{U_j (\Lambda \backslash X , \varphi + (\rr + \tau) u_{j})} K^{\rr}_j (X, \varphi ; \tau), 
	\quad 
	j\geq 1
	\label{eq:Z_j_generic_expression}
\end{align}
for RG coordinates $(E_j, g_j^{\rr}, U_j, K_j^{\rr})$, elements of normed spaces defined in Section~\ref{sec:polymer_activities_and_norms}. We give only a brief overview here:
$E_j$ is a real number, 
$g^{\rr}_j (\Lambda_N ; \tau)$ is an analytic function of $\tau \in \D_{\htau}$ and
\begin{align}
& U_j (X, \varphi) = \sum_{x\in X} \frac{1}{2} s_j |\nabla \varphi (x)|^2 +  W_j (X, \varphi) 
\label{eq:U_j_form}  \\
& W_j (X, \varphi) = \sum_{q\geq 1} L^{-2j} z_j^{(q)} \cos \big( q \beta^{1/2} \varphi(x) \big)
\label{eq:W_j_form}
\end{align}
for any $\varphi \in \C^{\Lambda}$ and some $s_j \in \R$, $(z_j^{(q)})_{q\geq 1} \subset \R$.
The polymer activity $K^{\rr}_j (X, \varphi ; \tau)$ parametrises the deviation of $Z^{\rr}_j$ from $e^{U_j}$ (up to a constant multiple).
The domain of $\varphi$-component of $K_j^{\rr}$ is initially $\R^{\Lambda}$, but it will be extended to some strip of $\C^{\Lambda_N}$ containing $\R^{\Lambda_N}$,
whose specification becomes important when we discuss the analyticity of $K_j^{\rr}$.

Section~\ref{sec:generic_rg_step} gives an explicit construction of the RG coordinates in scale $j+1$, given the coordinates in scale $j$. This is called the RG map. 
Namely, if there exists a specific choice of $(E_{j'}, g^{\rr}_{j'}, U_{j'}, K^{\rr}_{j'})_{j' \leq j}$ satisfying \eqref{eq:bulk_RG_flow} and \eqref{eq:Z_j_generic_expression}, 
then $(E_{j+1}, g^{\rr}_{j+1}, U_{j+1}, K^{\rr}_{j+1})$ 
with certain contraction properties is constructed via the RG map
\begin{align}
	\Phi^{\rr}_{j+1} : (U_j, K^{0}_j,K^{\rr}_j) \rightarrow (E_{j+1} - E_j , g^{\rr}_{j+1} - g^{\rr}_{j}, U_{j+1}, K^{0}_{j+1},K^{\rr}_{j+1}) 
	\label{eq:rg_map_first_view}
	.
\end{align}
The construction is based on operations defined in Section~\ref{sec:loc_and_reblocking}.
Key analytic properties of $\Phi^{\rr}_{j+1}$ are also stated Section~\ref{sec:generic_rg_step}, but we defer the technical  proof to Section~\ref{sec:proof-of-theorem-thm-local-part-of-K}. A reader not familiar with the content of \cite{dgauss1} may skip that section on the first read. 

In general, the existence of the maps $\Phi^{\rr}_{j+1}$ is not guaranteed for arbitrary large choice of $N$ and $j \leq N$.
The content of the first part of Section~\ref{sec:renormalisation-group-computation} is to prove such an existential theorem.
This theorem relies on the `stable manifold' theorem of
  \cite{dgauss1}--also see Section~\ref{sec:vacuum_energy}--where the existential theorem is obtained with
$\cz =\rr=\tau=0$,  $g_j^{\rr} \equiv 0$,$s_j, (z_j^{(q)})_{q\geq 0}, K_j^{\rr} \rightarrow 0$ (in norms specified later) and a specific initial condition.
Thus we only have to show the same convergence is also valid with the same initial condition and general $\rr \in \R$, $\tau \in \D_{\htau}$. 

The second half of Section~\ref{sec:renormalisation-group-computation} concludes the proof of Proposition~\ref{prop:overview_of_the_proof}.
By \eqref{eq:bulk_rg_final1}, \eqref{eq:bulk_rg_final2} and \eqref{eq:Z_j_generic_expression}, 
the ratio of partition functions in \eqref{eq:final_alternative_form_of_mgf} will converge to $\lim_{j\rightarrow \infty} \exp( g^{\rr}_{j} )$ as $N\rightarrow \infty$. 
Since the multi-scale polymer expansion \eqref{eq:Z_j_generic_expression} has strong dependence on the structure of the grids, the conclusion will be subject to the fact that the limit $\lim_{j\rightarrow \infty} g^{\rr}_{j}$ is free of the grid bias.
Indeed, this will be seen to be the case using a simple argument involving the translation invariance of the system.

\section{Polymer activities and norms}
\label{sec:polymer_activities_and_norms}

As in \eqref{eq:Z_j_generic_expression},  the partition function at scale $j$ is parametrised by coordinates $E_j, g_j, U_j$ and $K_j^{\rr}$.
While the form of $U_j$ is restricted by \eqref{eq:U_j_form},
$K^{\rr}_j$ can be a polymer activity of fairly high degree of freedom.  
In this section, we recall the norms defined on the polymer activities $U_j$ and $K^{\rr}_j (X, \cdot ; \cdot)$ from \cite{dgauss2}
and generalise them so we can encode information about analyticity in the parameter $\tau$.

\subsection{Norms}

We recall the norm on polymer activities from \cite{dgauss1} using the \emph{regulator} $G_j$ in the following definition.
A number of parameters, $A,h, c_{\kappa}, c_w,c_2 >0$ and $\kappa = c_{\kappa} (\log L)^{-1}$ will be appearing in the definition of the norm. 
The choice of these parameters are summarised in Section~\ref{sec:choice_of_parameters}.

\begin{definition}
\label{def:regulator}
For $X\in \cP_j$ and $\varphi \in \R^{\Lambda_N}$, define the regulator at scale $j$ by
\begin{align}
\begin{split}
	G_j ( X, \varphi ) &= \exp\Big( \kappa \big( \norm{\nabla_j \varphi }^2_{L^2_j (X)}
	+ c_2  \norm{\nabla_j \varphi }^2_{L^2_j (\partial X)} +  W_j(X,  \nabla^2_j \varphi )^2 \big) \Big) 
	\label{eq:regulator_definition}
\end{split}
\end{align}
where
\begin{align}
	W_j (X,  \nabla^2_j \varphi)^2 = \sum_{B\in \cB_j (X)} \norm{\nabla^2_j \varphi   }^2_{L^{\infty} (B^*)} .
\label{eq:W^r_j_definition}
\end{align}
Also define
\begin{align}
w_j (X, \varphi)^2 =  \sum_{B\in \cB_j (X)} \max_{n=1,2} \norm{\nabla^n_j \varphi}^2_{L^{\infty}} 
.
\end{align}
\end{definition}

The role of $w_j$ can be explained in terms of submultiplicativity as the following.
While
\begin{align}
	X\cap Y = \emptyset 
	\quad \not\Rightarrow \quad 
	G_j (X, \varphi) G_j (Y, \varphi) \leq C G_j (X \cup Y , \varphi)
\end{align}
for any $C >0$,
it is true, by \cite[Lemma~5.8]{dgauss1}, 
for $X, Y \in \cP_j$ such that $X \cap Y = \emptyset$,
\begin{align}
	e^{c_w \kappa w_j (X, \varphi)^2} G_j (Y, \varphi) \leq G_j (X \cup Y, \varphi)
	\label{eq:strong_regulator_key_bound}
\end{align}
whenever $c_w$ is chosen sufficiently small.  In particular, we have $e^{c_w \kappa w_j (X, \varphi )^2} \leq G_j (X , \varphi)$, and for this reason, 
we call $e^{c_w \kappa w_j (X, \varphi )^2}$ the \emph{strong regulator}. 

Then we define seminorms on polymer activities,
where we recall that $\cN_j$ defined in Definition~\ref{def:polymer_activities}.

\begin{definition}
For $F \in \cN_j$,  define 
\begin{align}
	\norm{D^n F(X,\varphi) }_{n, T_j (X, \varphi) } &= \sup\{ |D^n F (X,\varphi) (f_1, \cdots, f_n)|  \, : \, \norm{f_k}_{C_j^2 (X^*)} \leq 1 \; \text{for each} \; k \} \\
	\norm{F(X, \varphi)}_{h, T_j (X, \varphi)} &= \sum_{n=0}^{\infty} \frac{h^n}{n!} \norm{D^n F (X, \varphi)}_{n, T_j (X, \varphi)}
	.
\end{align}
\end{definition}

We also use polymer activities that are analytic in parameter $\tau$, and the corresponding seminorm and norms.

\begin{definition}
For $\ctau >0$ a constant,  let $\htau = \ctau (\log L)^{-3/2}$.
Define the space $W^+ (\D_{\htau})$ to be the space of holomorphic functions $f: \D_{\htau} \rightarrow \C$ such that the norm
\begin{align}
	\norm{f}_{\htau,  W} = \sum_{n=0}^{\infty} \frac{\htau^n}{n!} \Big| \partial_{\tau}^n f (\tau) |_{\tau = 0}  \Big|
\end{align}
is finite. 
For $A' >0$, $F \in \cN_{j, \kt}$ and $\vec{h} = (h, \htau)$, define
\begin{align}
	& \norm{F (X, \varphi ; \cdot )}_{\vec{h},  T_j (X, \varphi) } = \sum_{n=0}^{\infty} \frac{\htau^n}{n!} \big\| \partial_{\tau}^n F (X, \varphi; \tau) |_{\tau = 0} \big\|_{h, T_j (X, \varphi)} \\
	& \norm{F (X, \cdot ; \cdot)}_{\vec{h}, T_j (X) } = \sup_{\varphi \in \R^{\Lambda_N}} G_j (X, \varphi)^{-1} \norm{F (X, \varphi  ; \cdot)}_{\vec{h}, T_j (X, \varphi)}
	\label{eq:Xnorm_defi}	
	 \\
	& \norm{F}_{\vec{h}, T_j, A'} = \sup_{X\in \Conn_j (\Lambda_N)}  (A')^{|X|_j} \norm{F (X ; \cdot)}_{\vec{h}, T_j (X)}  . 
\end{align}
\end{definition}

While $\norm{\cdot}_{\htau, W}$ and $\norm{\cdot}_{\vec{h}, T_j (X,\varphi)}$ are independent of choice of $\tau$, 
we sometimes make $\tau$ explicit as an argument of $F$ when we want to make clear in which variable we are bounding the complex derivative.

When $A' = A$ 
(recall that $A$ is a constant fixed in Section~\ref{sec:choice_of_parameters}),
then the suffix $A'$ will often be omitted, i.e., $\norm{K_j (X ; \cdot)}_{\vec{h}, T_j} = \norm{K_j (X ; \cdot)}_{\vec{h}, T_j, A}$.
$\norm{F(X)}_{h, T_j (X)}$ and $\norm{F}_{h, T_j}$ (without $\tau$-dependence) are defined accordingly.

\begin{remark} \label{remark:new_norm}
Another way to think about the norm $\norm{K_j (X)}_{\vec{h}, T_j (X)}$ is to add an `external point' $\bigtriangleup$ and let $(\varphi, \tau)$ be a real-valued field on $\Lambda_N \cup  \bigtriangleup $. In other words,  $(\varphi, \tau) \in \R^{\Lambda_N} \oplus \R^{ \bigtriangleup }$.  
Also,  since we have set $\htau = \ctau (\log L)^{-3/2} $, so if $\R^{\bigtriangleup}$ is endowed with the norm $|x|_{\bigtriangleup} = h \ctau^{-1} (\log L)^{3/2} |x|$,  then 
\begin{align}
	\norm{K_j (X, \varphi ; \cdot)}_{\vec{h}, T_j (X, \varphi)} = \norm{ \bar{K}_j (X,  (\varphi, \tau) ) }_{h, T_j (X,  (\varphi, \tau) )} \big|_{\tau=0}
\end{align}
where $\bar{K}_j$ is a polymer activity on $\cP_j \times ( \R^{\Lambda_N} \times \R^{\bigtriangleup} )$ defined by
$
\bar{K}_j (X,  (\varphi, \tau)) = K_j (X, \varphi ; \tau)
$. 
Using this observation, we can re-derive most of the results on $\norm{\cdot}_{h, T_j (X, \varphi)}$, obtained in \cite{dgauss1},
for $\norm{\cdot}_{\vec{h}, T_j (X, \varphi)}$.
In this concern, we will mention that certain results of \cite{dgauss1} can be translated into results here, 
often without justifications in full detail.  
Vice versa,  most bounds proved for $\norm{\cdot}_{\vec{h}, T_j (X, \varphi)}$ should also be true for $\norm{\cdot}_{h, T_j(X, \varphi)}$
if there is no specific dependence on $\tau$. 
\end{remark} 

For $U_j$ in a restricted form, we can define and use a stronger norm. 

\begin{definition}
\label{def:U_j_norm}
Let $c_f = \frac{\gamma}{4}$. 
Given $( s_j ,  (z_j^{(q)})_{q \geq 1} ) \in \R \times \R^{\N}$ and $U_j$ given by \eqref{eq:U_j_form}--\eqref{eq:W_j_form}, define 
\begin{align}
	\norm{U_j}_{\Omega_j^U} =
	{ A }	 
	\max\Big\{ |s_j|,   \sup_{q \geq 1} e^{c_f \beta q} |z_j^{(q)}|  \Big\}
	.
\end{align}
\end{definition}

It is known from \cite[Lemma~7.4, Lemma~7.13]{dgauss1} with the choice of $\beta$ and $c_h$ as in Section~\ref{sec:choice_of_parameters}, that
\begin{align}
	\Big\| \frac{1}{2} s_j |\nabla \varphi|^2_B \Big\|_{{2 h}, T_j (B, \varphi)} 
	\leq C |s_j| w_j (B, \varphi)^2 , 
	\qquad 
	\norm{W_j (B, \varphi)}_{{2 h} , T_j (B, \varphi)} 
	\leq C { A^{-1} } \norm{W_j}_{\Omega_j^U}.
\label{eq:components_of_U_j_bound}
\end{align}
($h$ is replaced by $2h$ here, which is okay if we choose $C$ sufficiently large and $\beta \geq 8 c_f^{-1}$.)

\subsection{Key inequalities}

We highlight the most important properties of the norms defined in the previous section.
Whenever we are talking about $j$-scale polymer activities,  $\E$ means we are taking expectation over $\zeta \sim \cN (0, \Gamma_{j+1})$ if $j+1 < N$ and $\zeta \sim \cN (0, \Gamma_{N}^{\Lambda_N})$ if $j+1 = N$. The non-random part of the field is often denoted $\varphip$.

\begin{lemma}[Submultiplicativity of the seminorm]
\label{lemma:submultiplicativity_semi-norm}
For $\varphi \in \R^{\Lambda_N}$ and $Y_1,Y_2,X \in \cP_j$ such that $Y_1, Y_2 \subset X$, 
suppose $F_1 \in \cN_{j,\kt} (Y_1)$, $F_2 \in \cN_{j, \kt} (Y_2)$ with $\norm{F_{\alpha} (Y_{\alpha}, \varphi)}_{\vec{h}, T_j (Y_{\alpha}, \varphi)} < \infty$,  $\alpha \in \{1,2\}$.  Then
\begin{align}
	\norm{F_{1} (Y_1, \varphi ; \cdot ) F_{2} (Y_2, \varphi ; \cdot )}_{\vec{h}, T_j (X, \varphi)} \leq \norm{F_{1} (Y_1, \varphi ; \cdot ) }_{\vec{h}, T_j (Y_1, \varphi )} \norm{ F_{2} (Y_2, \varphi ; \cdot )}_{\vec{h}, T_j (Y_2, \varphi)} .
\label{eq:submultiplicativity}
\end{align}
Also for each $k\geq -1$,
\begin{align}
	\Big\| e^{F_1 (X, \varphi; \cdot)} - \sum_{m=0}^k \frac{1}{m!} (F_1 (X, \varphi ;\cdot))^m  \Big\|_{\vec{h}, T_j (X,\varphi)} \leq \sum_{m=k+1}^{\infty} \frac{1}{m!} \norm{F_{1} (X, \varphi ; \cdot) }_{\vec{h}, T_j (X, \varphi)}^m
\label{eq:exp_taylor_bound}
\end{align}
with convention $\sum_{m=0}^{-1} (\cdots) \equiv 0$.
\end{lemma}
\begin{proof}
The first inequality is a result of the submultiplicativity of the $\norm{\cdot}_{h, T_j (X, \varphi)}$-seminorm,  \cite[(5.22)]{dgauss1}
and Remark~\ref{remark:new_norm}.
To see the second,  consider the sequence of polymer activities
\begin{align}
	H_k (X, \varphi; \tau) = \sum_{m=0}^k \frac{1}{m!} (F_1 (X, \varphi))^m .
\end{align}
Then by the submultiplicativity, \eqref{eq:submultiplicativity},  we have
\begin{align}
	\norm{ H_k (X, \varphi ; \cdot)}_{\vec{h}, T_j (X, \varphi) }
	\leq \exp\big( \norm{F_1 (X, \varphi ; \cdot)}_{\vec{h}, T_j (X, \varphi)} \big)
	.
\end{align}
Since $e^{F_1}$ is a pointwise limit of $H_k$ and $D^n \partial_{\tau}^m H_k (X, \varphi; \tau) \rightarrow D^n \partial_{\tau}^m e^{F_1} (X, \varphi ; \tau)$ as $k\rightarrow \infty$ for each $n, m \geq 0$, we have
\begin{align}
	\sum_{n, m \geq 0}^{n+m \leq N'} \frac{\htau^m h^n}{m! n!} \big\| D^n \partial_{\tau}^m e^{F_1 (X, \varphi ; 0) } \big\|_{n, T_j (X, \varphi)} \leq \limsup_{k\rightarrow \infty} \norm{ H_k (X, \varphi ; \cdot)}_{\vec{h}, T_j (X, \varphi) }
\end{align}
for any $N' >0 $
so we see in fact
\begin{align}
	\big\| e^{F_1 (X, \varphi; \cdot)}  \big\|_{\vec{h}, T_j (X,\varphi)} \leq \exp\big( \norm{F_1 (X, \varphi ; \cdot )}_{\vec{h}, T_j (X, \varphi)} \big) < \infty. 
\end{align}
Then \eqref{eq:exp_taylor_bound} is obtained from the Taylor's theorem applied on $e^{F_1}$.
\end{proof}

The following inequalities control the norms of the polymer activities under renormalisation group maps. 

\begin{lemma}
\label{lemma:E_G^r_j}
For any $X\in \cP_j$, $\varphi \in \R^{\Lambda_N}$ and $0\leq j \leq N-1$,
\begin{align}
	\E[ G_j (X, \varphip + \zeta) ] 
	\leq 2^{|X|_j} G_{j+1} (\bar{X}, \varphip ).
\end{align}
If $\f$ is as in \eqref{quote:assumpf} and $u_j$ is defined by Definition~\ref{def:extfield_def}, then for some $C \equiv C(M, \rho, \rr)$, 
\begin{align}
	\E[ G_j (X, \varphip + \zeta + \rr u_{j+1}) ] 
	\leq C 2^{|X|_j} G_{j+1} (\bar{X}, \varphip ).
\end{align}
\end{lemma}

\begin{lemma} 
\label{lemma:Et_F_bound}

Let $F\in \cN_{j,\kt} (X)$ be such that $\norm{F (X ; \tau)}_{\vec{h}, T_{j} (X)} < \infty$ and $\htau \leq (C_1 \log L)^{-3/2}$ for sufficiently large $C_1 \equiv C_1 (M, \rho)$. 
Then $\D_{\htau} \ni \tau \mapsto \E_{(\tau)} [D^nF (X , \varphip + \zeta ; \tau)]$ is analytic in $\tau \in \D_{\htau}$ for each $n\geq 0$ and satisfies
\begin{align}
\big\| \E_{(\rr + \tau)} [F (X , \varphip + \zeta  ; \cdot) ] \big\|_{\vec{h}, T_j (X, \varphip)} \leq C_2 2^{|X|_j} \norm{F(X; \cdot)}_{\vec{h}, T_j (X)} G_{j+1} (\bar{X}, \varphip) .
\label{eq:Et_F_bound}
\end{align}
for some $C_2 \equiv C_2 (M, \rho,\rr) >0$.
\end{lemma}

Both lemmas are proved in Appendix~\ref{sec:norm_inequalities}.

\subsection{Neutralisation}
\label{sec:neutralisation}

Suppose $F$ is a ($2\pi \beta^{-1/2}$-)periodic polymer activity, 
i.e., for any $n \in 2\pi \beta^{-1/2} \Z$, $F$ satisfies $F(X, \varphi + n \one ) = F(X, \varphi)$.
As the DG model only has a gradient Gibbs measure that is translation invariant in the infinite volume and in the delocalised phase,  effective observables are only written in terms of $\nabla \varphi$.
Thus we use an operation that isolates the part of the polymer activity that only depends on $\nabla \varphi$.  This can be achieved by simply taking the $0^{th}$ component of the Fourier decomposition:
\begin{align}
	\hat{F}_0 (X, \varphi) := \int_0^{1} F(X, \varphi + 2\pi \beta^{-1/2} y \one) dy,
\end{align}
also called the \emph{neutral part} or the \emph{charge $0$ part} of $F$.
It can be observed from the definition that
\begin{align}
	\norm{\hat{F}_0 (X, \varphi)}_{h, T_j (X, \varphi)} 
		\leq \norm{F (X)}_{h,  T_j (X)} G_j (X, \varphi)
	\label{eq:neutralisation_bound}
	.
\end{align}
Indeed, 
\begin{align}
	\norm{\hat{F}_0 (X, \varphi)}_{h, T_j (X, \varphi)} 
		& \leq \int_0^1 \norm{ F(X, \varphi + 2\pi \beta^{-1/2} y \one) }_{h,T_j (X,\varphi)} dy
		\nnb
		& \leq \int_0^1 \norm{ F(X) }_{h,T_j (X)} G_j (X, \varphi + 2\pi \beta^{-1/2} y \one ) dy
		\nnb
		& = \int_0^1 \norm{ F(X) }_{h,T_j (X)} G_j (X, \varphi ) dy
		\nnb
		& = \norm{ F(X) }_{h,T_j (X)} G_j (X, \varphi )
\end{align}
where the equality in the third line uses that the the regulator is invariant under addition of a constant field (see Definition~\ref{def:regulator}).

\subsection{Analyticity of polymer activities}
\label{sec:analyticity-of-polymer-activities}

We start with an observation made in \cite{MR1101688}, where boundedness of a polymer activity implies the (complex-)analyticity of the polymer activity on a neighbourhood of the real space. 

\begin{proposition} 
\label{prop:analytic_on_strip}

Let $h, \htau>0$, $X\in \cP_j$ and $\norm{F (X)}_{\vec{h},T_j (X)} < +\infty$. 
Then $F(X, \cdot ; \tau)$ can be extended to the domain $S_h (X)= \{ \varphi + i\psi  \in \C^{\Lambda_N}  : \varphi(x), \psi(x) \in \R, \; \norm{\psi}_{C_j^2 (X^*)} < h \}$
where each $D^n F(X, \cdot ; \tau)$ is complex analytic 
and satisfies
\begin{align}
|F(X, \varphi + \phi; \tau)| \leq \norm{F (X, \varphi ; \tau)}_{h, T_j (X, \varphi ; \tau)}
\label{eq:bound_in_the_analytic_strip}
\end{align}
whenever $\varphi \in \R^{\Lambda_N}$, $\phi \in \C^{\Lambda_N}$, $\norm{\phi}_{C_j^2 (X^*)} < h$. 
\end{proposition}
\begin{proof}
The proof is the same as that of \cite[Proposition~5.6]{dgauss1},
just replacing the norm $\norm{\cdot}_{h, T_j}$ by $\norm{\cdot}_{\vec{h}, T_j}$.
The inequality \eqref{eq:bound_in_the_analytic_strip} follows from expanding out $F(X, \varphi + \cdot ; \tau)$ in Taylor series (see \cite[(5.23)]{dgauss1}).
\end{proof}

By the proposition,  we can make complex shift of variables in each Gaussian integrals using the Cauchy's integral theorem as long as the shift is not too large. This result is summarised in the next lemma,
whose proof is be presented in Appendix~\ref{sec:complex-valued-shift-of-variables}.
Again, when working with $j$-scale polymer activities,  
$\E$ is an expectation over $\zeta \sim \cN (0, \Gamma_{j+1})$ if $j+1 < N$ and $\zeta \sim \cN (0, \Gamma_{N}^{\Lambda_N})$ if $j+1 = N$.

\begin{lemma}[Gaussian complex shift of variable]
\label{lemma:Et_by_complex_shift}
 
Let $\f$ be as in \eqref{quote:assumpf}, 
$u_{j+1}$ be as in Definition~\ref{def:extfield_def}
and $F \in \cN_{j} (X)$ with $\norm{F (X)}_{h, T_j (X)} < \infty$.  
Then for $\htau < (C \log L)^{-1} h$ with $C >0$ sufficiently large and $\tau \in \D_{\htau}$, 
\begin{align}
\Et \big[ F (X,\varphip + \zeta)   \big] =  \E \big[  F \big( X, \varphip + \zeta + \tau u_{j+1}  \big)  \big]
.
\end{align}
\end{lemma}

As the norm $\norm{F}_{\vec{h}, T_j}$ exploits the analyticity of $F$ even further, we would have to study this a bit further.

\begin{lemma}
\label{lemma:analyticity_preserved_under_E}
Let $\vec{h} = (h, \htau) >0$ and $F \in \cN_{j,\kt} (X)$ be such that $\norm{F(X; \cdot )}_{\vec{h},T_j (X)} < +\infty$ where $X\in \Conn_j$.  
Then for $\varphip \in \R^{\Lambda_N}$ and $n,m\geq 0$,
$\D_{\htau} \ni \tau \mapsto \E [D^n F(X, \varphip + \zeta ; \tau)]$ 
is a complex analytic function
and $\E[ \partial_\tau^m D^n F (X, \varphip + \zeta ; \tau) ] = \partial_\tau^m \E[  D^n F (X, \varphip + \zeta ; \tau) ]$.
\end{lemma}

\begin{proof}
Let $\tau \in \D_{(1-2 \delta) \htau}$ for some $\delta > 0$.
Then by the Cauchy's integral formula, for each $n,m \geq 0$,
\begin{align}
	\norm{ \partial_{\tau}^m D^{n} F (X, \varphi ; \tau) }_{n, T_j (X, \varphi)} 
	& = 
	\Big\| \frac{m !}{2\pi i} \int_{|z| = (1-\delta') \htau } \frac{ D^{n} F (X, \varphi ; z)}{(z- \tau)^{m+1}} dz  \Big\|_{n, T_j (X, \varphi)}
	\nnb
	& \leq
	\frac{(1-\delta') n!}{(\delta - \delta')^2 \htau h^n} \norm{F(X)}_{\vec{h}, T_j (X)} G_j (X, \varphi)
\end{align}
where we have used, for $z \in \D_{\htau}$,
\begin{align}
	\norm{ D^{n} F (X, \varphi ; z) }_{n, T_j (X, \varphi)} 
	& \leq  
	\sum_{k=0}^{\infty} \frac{|z|^k}{k!} \norm{\partial^k_{\tau} D^{n} F (X, \varphi ; \tau) |_{\tau =0} }_{n,  T_j (X, \varphi)}  
	\nnb
	& \leq
	\frac{n!}{h^n} \norm{F (X, \varphi ; \cdot) }_{\vec{h},  T_j (X, \varphi)}  
	\nnb
	& \leq \frac{n!}{h^n} \norm{F (X,  \cdot; \cdot) }_{\vec{h},  T_j (X)}  G_j (X, \varphi)
	\label{eq:D^n_F_bound_from_norm}
\end{align}
and the final inequality is due to \eqref{eq:Xnorm_defi}.
By Lemma~\ref{lemma:E_G^r_j}, $\E [G_j (X, \varphip + \zeta) ] \leq 2^{|X|_j} G_{j+1} (\bar{X} , \varphip)$ for each $\varphip \in \R^{\Lambda}$,
so the Dominated convergence theorem guarantees that $\tau \mapsto\E [D^n F(X, \varphip + \zeta ; \tau)]$ is complex analytic for $\tau \in \D_{\htau}$ and $\E[ \partial_\tau^m D^n F (X, \varphip + \zeta ; \tau) ] = \partial_\tau^m \E[  D^n F (X, \varphip + \zeta ; \tau) ]$.
\end{proof}

\begin{lemma}
\label{lemma:F'_is_analytic}
Let $h, \htau>0$ be such that $\htau < (C \log L)^{-1} h$ for sufficiently large $C$. 
Let $F \in \cN_{j,\kt} (X)$ be such that $\norm{F(X; \tau)}_{(h'', \htau),T_j (X)} < +\infty$ where 
$h'' = h + \htau \norm{u_{j+1}}_{C_{j}^2}$ and
$X\in \Conn_j$. 
If we define 
\begin{align}
F' (X, \varphi  ; \, \cdot \,) : \D_{\htau} \rightarrow \C, \quad \tau \mapsto F(X, \varphi + \tau u_{j+1} ; \tau),
\end{align}
then $D^n F'(X, \varphi ; \cdot)$ and $\E [D^n F' (X,\varphip + \zeta ; \cdot) ]$ are complex analytic functions of $\tau \in \D_{\htau}$ for each $n\geq 0$,
and satisfies
\begin{align}
	&
	\norm{F' (X, \varphi  ; \cdot)}_{\vec{h}, T_j (X, \varphi)} 
	\leq \norm{F(X, \varphi ; \cdot)}_{ (h'' , \htau ) , T_j (X, \varphi)}
	\label{eq:F'_bound}
	\\
	& 
	\norm{\E_{(\cdot)} [F(X, \varphip + \zeta ; \cdot )] }_{\vec{h}, T_j (X, \varphip)} 
	\leq 	\norm{\E [ F(X, \varphip + \zeta ; \cdot) ]}_{(h'', \htau), T_j (X, \varphip)}
	.
	\label{eq:weaker_Et_bound}
\end{align}
\end{lemma}

\begin{proof} 
By Proposition~\ref{prop:analytic_on_strip}, $z \mapsto D^n F(X, \varphi + z u_{j+1} ; \tau)$ is analytic whenever $\norm{z u_{j+1}}_{C_{j}^2 (X^*) } < h$. 
However, since $\norm{u_{j+1}}_{C_j^2 (X^*)} \leq 2 C(M, \rho) \kappa^{-1} $, this condition is satisfied whenever $|z| \leq \htau < (2 C)^{-1 } \kappa h$, making $F' (X, \varphi; \tau)$ analytic in $\tau \in \D_{\htau}$. 
Now by the Chain rule, 
\begin{align}
\frac{d^m}{d\tau^m} D^n F' (X,\varphi ; \tau) = \sum_{k=0}^m
\begin{pmatrix}
m \\
k
\end{pmatrix}
D^{n+k} \partial_{\tau}^{m-k} F(X, \varphi + \tau u_{j+1} ; \tau) (u_{j+1}^{\otimes k})
\end{align}
($D$ and $\partial_{\tau}$ are partial derivatives)
so
\begin{align}
\norm{F' (X, \varphi ; \cdot)}_{\vec{h}, T_j (X, \varphi)}
& \leq \sum_{n,m=0}^{\infty} \sum_{k=0}^{m} \frac{h^n \htau^m}{n! m^!} 
\begin{pmatrix}
m \\
k
\end{pmatrix}
\norm{D^{n+k} \partial_{\tau}^{m-k} F(X, \varphi ; 0) }_{n+k, T_j (X, \varphi)} \norm{u_{j+1}}^k \nnb
& = \sum_{n,k=0}^{\infty} \sum_{m'=0}^{\infty} \frac{h^n \htau^{k+m'}}{n! k! m' !} \norm{u_{j+1}}^k  \norm{D^{n+k} \partial_{\tau}^{m'} F(X, \varphi ; 0) }_{n+k, T_j (X, \varphi)} \nnb
& = \sum_{m' =0}^{\infty} \frac{\htau^{m'}}{m' !} \norm{\partial_{\tau} ^{m'} F(X, \varphi ;0)}_{h + \htau \norm{u_{j+1}}, T_j (X, \varphi) }
,
\end{align}
where the second line follows from change of variable $m' = m-k$ and $\norm{u_{j+1}} = \norm{u_{j+1}}_{C_j^2 (X^*)}$.
This yields \eqref{eq:F'_bound}, 
hence by Lemma~\ref{lemma:analyticity_preserved_under_E}, we also have that $\E [D^n F' (X, \varphip + \zeta ; \tau)]$ complex analytic in $\tau \in \D_{\htau}$ for each $n\geq 0$.
Now \eqref{eq:weaker_Et_bound} follows from Lemma~\ref{lemma:Et_by_complex_shift}, saying that
$
\Et [F(X, \varphip + \zeta ; \tau )] = \E [ F' (X, \varphip + \zeta ; \tau ) ] 
$
and applying the same type of argument on $\E [ F' (X, \varphip + \zeta ; \tau ) ]$.

\end{proof}

This lemma is usually used after setting $C_1 > 0$ sufficiently large so that $h'' \leq 2h$.
Then if we bound the right-hand side of \eqref{eq:weaker_Et_bound} using the definition of $\norm{\cdot}_{\vec{h}, T_j (X)}$-norm, 
we have
\begin{align}
	\norm{\E_{(\rr + \tau)} [F(X,  \varphip + \zeta ; \tau )] }_{\vec{h}, T_j (X, \varphip)} 
	\leq C 2^{|X|_j} \norm{F(X, \cdot ; \tau) ]}_{(2h, \htau), T_j (X)} G_{j+1} (\bar{X} , \varphip)
	,
\end{align}
which is similar to the conclusion of Lemma~\ref{lemma:Et_F_bound}, but weaker.
However, \eqref{eq:weaker_Et_bound} has use of its own as it makes is easy to import inequalities from \cite{dgauss1}.

\ifx\newpageonoff\undefined
{\red command undefined!!}
\else
  \if\newpageonoff1
  \newpage
  \fi
\fi

\section{Localisation and reblocking}
\label{sec:loc_and_reblocking}

In this section, we treat important operations used to define the renormalisation group map and state their key algebraic and analytic properties. 

\subsection{Localisation operator}
\label{sec:loc_operator}

Localising a polymer activity $F(X, \varphi)$ isolates the local relevant terms.
In the case of our interest, these terms firstly take the neutral part of the polymer activity and secondly take
 the constant terms and the second degree terms when expanded out in $\nabla \varphi$. 
Hence we require the localisation operator to approximate the Taylor polynomial of order 2 (see Lemma~\ref{lemma:Loc_minus_Tay}), 
but it is designed to satisfy a specific identity(see \eqref{eq:cU_definition}).
Explicit definition of the localisation is skipped and the proofs of the analytic properties are deferred to the appendix, except for some key lemmas, because they are essentially repetitions of discussions from \cite{dgauss1, dgauss2}, with appendage only to do with analyticity in $\tau$. 
However, these can not be skipped because there are still some technical subtleties coming from complex shift of variables combined with the localisation operators.

\begin{definition} \label{def:Loc}
Let $F_1 \in \cN_j$ be periodic,  
and let $\hat F_0$ be its neutral part.
For $X \in \cS_j$, $B \in \cB_j(X)$, 
define
$\Loc_{X, B} F_1(X)$ and $\Loc_{X} F_1(X)$ according to 
\cite[Definition~6.4]{dgauss1}.
For periodic $F_2 \in \cN_{j, \kt}$, $\hat{F}_{2,0}$ its neutral part and $X\in \cS_j$, define
\begin{align}
	\Loco_{X} F_2 (X) = \hat{F}_{2,0} (X,0) 
	.
\end{align}
\end{definition}

The following shows that $\Loc \Et$ and $\Loco \Et$ are bounded operators, and proved in Section~\ref{sec:Loc_Et_K_j_bound_proof}.

\begin{lemma} \label{lemma:Loc_Et_K_j_bound}
Let $F_1 \in \cN_{j, \kt}$ be periodic and let $X\in \cS_j$, $B\in \cB_j (X)$.
If $\norm{F_1 (X) }_{\vec{h}, T_j (X)} < \infty$, 
then $\Loco_X \E_{(\rr + \tau)} F_1(X, \varphip + \zeta ; \tau)$ is an analytic function of $\tau \in \D_{\htau}$
and there is $C \equiv C (M, \rho, \rr)$ such that
\begin{align}
	\norm{ \Loco_X \E_{(\rr + \cdot)} F_1 (X, \varphip + \zeta ; \cdot) }_{\htau, W} \leq C \norm{F_1 (X; \cdot)}_{\vec{h}, T_j (X)}
	\label{eq:Loco_Et_K_j_bound}
	.
\end{align} 
If $F_2 \in \cN_j$ is periodic with $\norm{F_2}_{h, T_j} < \infty$, then there is $C' \equiv C' (M, \rho, L, \rr)$ such that
\begin{align}
	\norm{\Loc_{X, B} \E_{(\rr + \cdot)} [  F_2 ( X, \varphip + \zeta ) ]}_{\vec{h}, T_j (B, \varphip)} \leq C' \norm{F_2 (X )}_{h,  T_j (X)} e^{c_w \kappa  w_j (B, \varphip)^2} 
	.
\label{eq:Loc_Et_K_j_bound}
\end{align}
\end{lemma}

\subsubsection{Irrelevance of non-local terms.}

The following proposition claims that the terms that are not included in $\operatorname{Loc}^{(k)} \Et$ ($k \in \{0,2\}$) contracts upon the fluctuation integral $\E_{(\rr + \tau)}$ with factor
\begin{equation} \label{e:Loc-contract-kappa}
    \alpha_{\operatorname{Loc}}^{(k)}
    = C (L^{-2} \log L)^{(k+1)/2} + C\min \ha{1,\sum_{q\geq 1}
    	e^{2 \sqrt{\beta}qh}
    	e^{-(q-1/2)\beta\Gamma_{j+1}(0)}}
    ,
 \end{equation}
and thus we see  that the localisations serve as an even better approximations of the polymer activity upon fluctuation integral and change of scale.

\begin{proposition} \label{prop:Loc-contract_v2}
  There exists a constant $c_h>0$ such that the following holds for all
periodic $F \in \cN_{j, \kt}$ and $X \in \cS_j$.
Assume $h \geq \max \{c_h  \sqrt{\beta},  1\}$, $\htau \leq (C \log L)^{-1} h$,
  $L \geq L_0 (R)$ for $C$ and $L_0 (R)$ sufficiently large and $A' \geq 1$.
If $\norm{F}_{\vec{h}, T_j,A'} < \infty$, then for all $\varphip \in \R^{\Lambda_N}$,
  \begin{align} \label{eq:Loco-contract-full}
   &   \| \Loco_X \E_{(\rr + \cdot)} F(X, \varphip+\zeta ; \cdot) - \E_{(\rr + \cdot)} F(X, \varphip+\zeta ; \cdot)\|_{\vec{h},T_{j+1}(\bar X, \varphip)}
    \nnb
   & \qquad \qquad\qquad\qquad\qquad \qquad\qquad\qquad    \leq  
    \alphaLoco (A')^{-|X|_j}
    \|F\|_{\vec{h},T_j, A'} G_{j+1}(\bar X,\varphip)
    .
  \end{align}
If we assume instead instead $\norm{F(X, \cdot ; 0)}_{h, T_j, A'} < \infty$ and $F(X, \varphi) = F(X, -\varphi)$,
then
  \begin{align} \label{eq:Loc-contract-full}
   &   \|  \Loc_X \E_{(\rr + \cdot)} F(X, \varphip+\zeta ; 0) - \E_{(\rr + \cdot)} F(X, \varphip+\zeta ; 0)\|_{\vec{h},T_{j+1}(\bar X, \varphip)}
    \nnb
   & \qquad \qquad\qquad\qquad\qquad \qquad\qquad\qquad    \leq  
    \alphaLoc (A')^{-|X|_j}
    \norm{F (\cdot ; 0)}_{h, T_j, A'} G_{j+1}(\bar X,\varphip)
    .
  \end{align}
\end{proposition}

This proposition is proved in Section~\ref{sec:proof_of_Loc-contract}.

\subsection{Rescale-reblocking}

When we are given with a polymer expansion at scale $j$, the easiest way to rewrite the expansion in scale $j+1$ is to introduce the rescale-reblocking operator
\begin{align}
	\mathbb{S} F(X) = \sum_{Y\in \Conn_j}^{\bar{Y} = X}  F(Y), \qquad X\in \cP_{j+1}^c.
\end{align}
For general $X\in \cP_{j+1}$, let $\mathbb{S} F(X) = \prod_{Z\in \Comp_{j+1} (X)} \mathbb{S} F(Z)$.
The following proposition claims that the polymer activities on small polymers have dominant contribution on $\mathbb{S} F$.

\begin{proposition} \label{prop:largeset_contraction_external_field}
  There exists a constant $\eta >0$ such that the following holds. 
Let $L\geq 5$, $A' \geq A_0 (L)$ where $A_0 (L)$ is some function only polynomially large in $L$.
If $F \in \cN_{j, \kt}$ is supported on large sets (i.e., non-small sets) and 
$\norm{F}_{h, T_j,A'} \leq \epsilon_{rb} (A')$, 
then for $X\in \cP_{j+1}^c$,
\begin{equation}
	\big\| \mathbb{S} \big[ \E_{(\rr + \cdot)} [ F(X, \cdot + \zeta)  ] \big] \big\|_{\vec{h}, T_{j+1} (X), A'} 
	\leq (L^{-1} (A')^{-1} )^{|X|_{j+1}} \norm{F}_{\vec{h},  T_j}.
\end{equation}
\end{proposition}

This proposition is proved in Section~\ref{sec:prop_largeset_contraction_external_field_proof}.

\subsection{Irrelevance of the perturbations}

When the $j$-scale polymer expansion \eqref{eq:Z_j_generic_expression} is perturbed by a field, then we may use field-reblocking operator to rewrite them in terms of a polymer expansion free of the perturbation. This operator is defined as the following. 

\begin{definition} \label{def:K_j^dagger_definition}
Given $v_j \in \C^{\Lambda_N}$ and $K_j \in \cN_{j, \kt}$,  $U_j \in \cN_{j}$, 
define for $X \in \Conn_j$,
\begin{align}
	& \cR_j [v_j, U_j, K_j ] (X, \varphi) 
	=  \sum_{Y\in \cP_j (X)} (e^{U_j(\cdot, \varphi + v_j)  } - e^{U_j (\cdot, \varphi)} )^{X \backslash Y} K_j (Y, \varphi)
\end{align}
where
$(\cdot)^{X\backslash Y}$ is defined according to \eqref{eq:polymer_power_convention}. 
This can be extended to general $Z\in \cP_j$ by $\cR_j (Z) = \cR_j ^{[Z]}$ (recall \eqref{eq:polymer_power_convention}).

Given $u_j$ as in \eqref{quote:assumpu}, $\rr \in \R$, $\tau \in \C$, define
\begin{align}
	K_j^{\dagger} = \cR_j [ (\tau + \rr) u_j, U_j, K_j].
\label{eq:K_j^dagger_definition}
\end{align}
\end{definition}

The rewriting of the polymer expansion is a purely algebraic result, as stated in the next proposition. Also, Lemma~\ref{lemma:reblocking_estimate} indicates that field-reblocking can be done without much cost on the norm, as long as the size of the perturbed field is not too large.

\begin{proposition} 
\label{prop:reblocking_Z_with_Psi}
Let $K_j \in \cN_{j, \kt}$ and let $U_j \in \cN_j$ be additive over blocks,  i.e.,~$U_j(X \cup Y)=U_j(X)+ U_j(Y)$ for all $X\cap Y =\emptyset$, $X,Y \in \cP_j$.
Let $Z_j^{\rr}$ be given by \eqref{eq:Z_j_generic_expression}
and define $K_j^{\rr,\dagger}$ according to Definition~\ref{def:K_j^dagger_definition} (with $K_j^{\rr}$ in place of $K_j$).
Then for $c = -E_{j} |\Lambda_N| + g_{j}^{\rr} (\Lambda_N ;\tau)$,
\begin{equation}
Z^{\rr}_j (\varphi) = e^{c} \sum_{X\in \cP_j} e^{U_j (\Lambda \backslash X, \varphi)}
K_j^{\rr,\dagger}(X, \varphi)
.
\label{eq:reblocking_Z_w_ext_field} 
\end{equation}
If $K^{\rr}_j, U_j$ are ($2\pi \beta^{-1/2}$-)periodic, then so is $K_j^{\rr, \dagger}$.
If $u_j$ is as in \eqref{quote:assumpu} and $P_y^j$ is given by \eqref{eq:union_of_B^j_k's}, 
then 
$K^{\rr}_j (X) = K_j^{\rr, \dagger} (X)$ whenever $X \cap P_y^j = \emptyset$.
\end{proposition}

\begin{proof}
Let $\cz = \rr + \tau$ so $K_{j}^{\rr,\dagger} = \cR_j [\cz u_j, U_j, K_j^{\rr}]$.
Using polymer expansion
\begin{align}
	e^{U_j (\Lambda\backslash X, \varphi + \cz u_j)} = \sum_{Y \in \cP_j (\Lambda \backslash X)} \big( e^{U_j (\cdot, \varphi + \cz u_j)} - e^{U_j (\cdot, \varphi)} \big)^{Y} e^{U_j (\Lambda \backslash (X\cup Y),  \varphi)},
\end{align}
and denoting $Z = X \cup Y$, we can rewrite \eqref{eq:Z_j_generic_expression} as
\begin{align}
	Z^{\rr}_j (\varphi) = e^{c} \sum_{Z\in \cP_j} e^{U_j (\Lambda \backslash Z, \varphi)} \sum_{Y \in \cP_j (Z)} \big( e^{U_j (\cdot, \varphi + \cz u_j)} - e^{U_j (\cdot, \varphi)} \big)^{Y} K^{\rr}_j (Z\backslash Y, \varphi) .
\end{align}
But because of factorisation property of $K_j$,  \eqref{eq:factorisation_of_polymer_activities}, we actually have
\begin{align}
	& \sum_{Y \in \cP_j (Z)} \big( e^{U_j (\cdot, \varphi + \cz u_j)} - e^{U_j (\cdot, \varphi)} \big)^{Y} K_j (Z\backslash Y, \varphi) \nnb
	& \qquad \qquad = \prod_{Z' \in \Comp_j (Z)} \sum_{Y \in \cP_j (Z')} \big( e^{U_j (\cdot, \varphi + \cz u_j)} - e^{U_j (\cdot, \varphi)} \big)^{Y} K^{\rr}_j (Z' \backslash Y, \varphi) ,
\end{align}
giving \eqref{eq:reblocking_Z_w_ext_field}. 
Final remarks are also clear from the definition of the operation $\cR_j$'s.
\end{proof}

\begin{lemma}
\label{lemma:reblocking_estimate}
Suppose $u_j$ satisfies \eqref{quote:assumpu},
 $U_j$ given in form \eqref{eq:U_j_form},  $K_j \in \cN_{j,\kt}$ and $\htau < (C \log L)^{-1} h$ for $C$ sufficiently large.
Then there exists $\epsilon_{r} \equiv \epsilon_r (M, \rho, A,L, \rr) > 0$  
such that whenever $\tau \in \D_{\htau}$ and $\max\{ \norm{U_j}_{\Omega_j^U} , \norm{K_j}_{\vec{h}, T_j} \} \leq \epsilon_{r}$,
\begin{align}
	& \norm{K_j^{\dagger} }_{\vec{h}, T_j, \frac{1}{2} A }, 
	\leq \max\big\{ \norm{U_j}_{\Omega_j^U} , \norm{K_j}_{\vec{h}, T_j} \big\}
	.
	\label{eq:reblocking_estimate1}
\end{align}
\end{lemma}
\begin{proof}
By Lemma~\ref{lemma:F'_is_analytic} and \eqref{eq:components_of_U_j_bound} with the assumption $h + \htau < \kappa (2M)^{-1} h + h < 2h$,  one may write
\begin{align}
	\norm{ U_j (B, \varphi + (\rr + \tau) u_j)  - U_j (B, \varphi + \rr u_j) }_{\vec{h}, T_j (B, \varphi)} & \leq  \norm{U_j (B, \varphi + \rr u_j )}_{h + \htau \norm{u_j} , T_j (B, \varphi)}
	\nnb
	& \leq C {  A^{-1}} \norm{U_j}_{\Omega_j^U} w_j (B, \varphi + \rr u_j)^2.
\end{align}
Also by \eqref{eq:components_of_U_j_bound}, 
\begin{align}
	\norm{ U_j (B, \varphi + \rr u_j)  - U_j (B, \varphi + u_j) }_{\vec{h}, T_j (B, \varphi)} \leq C A^{-1} \norm{U_j}_{\Omega_j^U} ( 1+ C' (\rr) \norm{ u_j }_{C_j^2}^2 )
\end{align}
Then Lemma~\ref{lemma:submultiplicativity_semi-norm} gives
\begin{align}
	\norm{ e^{U_j (B, \varphi + (\rr + \tau ) u_j) }  - e^{U_j (B, \varphi)}}_{\vec{h}, T_j (B, \varphi)}
	& \leq
	\norm{ e^{U_j (B, \varphi)}}_{\vec{h}, T_j (B, \varphi)} \norm{ e^{U_j (B, \varphi + (\rr +\tau) u_j) - U_j (B, \varphi)}  - 1}_{\vec{h}, T_j (B, \varphi)}
	\nnb
	& \leq C {  A^{-1}}  \norm{U_j }_{\Omega_j^U} e^{ C {  A^{-1}}  \norm{U_j }_{\Omega_j^U}   (w_j (B, \varphi)^2 + C' (\rr) \norm{ u_{j}}_{C^2_j}^2 ) }
	.
\end{align}
Hence, 
if $x := \max\{ \norm{U_j }_{\Omega_j^U} ,  \norm{K_j}_{\vec{h}, T^{0}_j, A} \}$ is sufficiently smaller than $\min\{A, c_w \kappa\}$ and $(C' (\rr) \norm{u_{j}}_{C^2_j}^2)^{-1}$, 
then by \eqref{eq:submultiplicativity} and \eqref{eq:strong_regulator_key_bound},
\begin{align}
	& \Big\| \Big( e^{U_j (\cdot, \varphi + ( \rr + \tau) u_j) }  - e^{U_j (\cdot, \varphi)} \Big)^{X \backslash Y} K_j (Y, \varphi)  \Big\|_{\vec{h}, T_j (X, \varphi)} \nnb
	& \qquad\qquad\qquad\qquad\qquad\qquad \leq
	A^{  -|X|_j} ( C x)^{^{|X \backslash Y|_j}} e^{c_w \kappa w_j (X\backslash Y, \varphi)^2} x^{|\operatorname{Comp}_j (Y)|} G_j (Y, \varphi)
	\nnb
	& \qquad\qquad\qquad\qquad\qquad\qquad \leq
	A^{  -|X|_j} ( C x)^{^{|X \backslash Y|_j}} x^{|\operatorname{Comp}_j (Y)|} G_j (X, \varphi)
\end{align}
for $X, Y \in \cP_j$, $Y\subset X$.
By the definition of $K_j^{\dagger}$,  now with $x$ sufficiently smaller than $A$,
\begin{align}
	\norm{K_j^{\dagger} (X ; \tau)}_{\vec{h}, T_j (X)} 
	\leq 
	A^{-|X|_j}
	\sum_{Y \in \cP_j (X)} (C x)^{|X \backslash Y|_j} x^{|\operatorname{Comp}_j (Y)|}  \leq (A/2)^{-|X|_j} x
	.
\end{align}
Together with \eqref{eq:components_of_U_j_bound},  we obtain the bound on $K_j^{\dagger}$.
\end{proof}

\ifx\newpageonoff\undefined
{\red command undefined!!}
\else
  \if\newpageonoff1
  \newpage
  \fi
\fi

\section{Renormalisation group map}
\label{sec:generic_rg_step}

The objective of this section is to define $(E_{j+1},  g^{\rr}_{j+1}, U_{j+1}, K_{j+1}^{\rr})$ as a function of given $(E_j,  g^{\rr}_j, U_j, K_j^{\rr,\dagger})$ using
\begin{align}
	\Phi^{\rr, \Lambda_N}_{j+1} 
	: (U_j,  \vec{K}_j) \mapsto \big( E_{j+1} - E_j,  g_{j+1}^\rr - g_j^\rr,  U_{j+1}, \vec{K}_{j+1} \big)
	.
	\label{eq:rg_map_definition}
\end{align}
If we consider each coordinate-map as a function of $(U_j,  \vec{K}_j)$, 
we can also denote $\Phi^{\rr, \Lambda_N}_{j+1} = (\cE_{j+1} , \kg_{j+1}^\rr, \cU_{j+1},  \vec{\cK}_{j+1})$.
The inductive relation of Section~\ref{sec:effpt} is a crucial restriction on them.

\begin{definition}
\label{def:rgmap}
The map $\Phi^{\rr,\Lambda}_{j+1}$ as in \eqref{eq:rg_map_definition} is called an \emph{RG map} if 
\begin{align}
	(E_{j+1}, g_{j+1}^\rr, U_{j+1}, \vec{K}_{j+1}) = (E_j + \cE_{j+1} ,  g^{\rr}_j + \kg_{j+1}^\rr, \cU_{j+1},  \vec{\cK}_{j+1})
\end{align}
satisfies the following.
If $Z_j^{\rr}$ and $Z_{j+1}^{\rr}$ are given in form \eqref{eq:Z_j_generic_expression},
then they satisfy the inductive relation satisfy \eqref{eq:bulk_RG_flow}. 
\end{definition}

The domain of $\Phi_{j+1}^{\rr,\Lambda_N}$ is given as normed spaces as the following.

\begin{definition}
We denote the RG coordinates and their spaces as the following.
\begin{itemize}
\item 
Denote $U_j$ and $W_j$ to be implicit functions of the coupling constants $(s_j,  (z_j^{(q)})_{q\geq 1} ) \subset \R \times \R^{\N}$
defined by
\eqref{eq:U_j_form} and \eqref{eq:W_j_form}.
The space of $U_j$ of such form with $\norm{U_j}_{\Omega_j^U} < \infty$ (see Definition~\ref{def:U_j_norm}) is called $\Omega_j^U$.

\item Let $\Omega_{j}^{K}$ be the space of $F_1 \in \cN_j$ 
that is invariant under lattice symmetries, ($2\pi \beta^{-1/2}$-)periodic and $F_1 (\cdot , \varphi) = F_1 (\cdot, -\varphi)$,
and let $\norm{F_1}_{\Omega_j^K} := \norm{F_1}_{h, T_j} < \infty$.

\item Let $\Omega_{j, \kt}^K$ be the space of periodic $F_2^{\rr}\in \cN_{j, \kt}$ such that $\norm{F_2^{\rr}}_{\vec{h}, T_j}  < \infty$
and ${\Omega}_{j, \kt, \dagger}^{K}$ is the space of periodic $F_3^{\rr} \in \cN_{j, \kt}$ such that $\norm{F_3^{\rr}}_{\vec{h}, T_j, A/2} < \infty$. 
Then
$\bar{\Omega}_{j, \kt}^K \subset \Omega_j^{K} \times {\Omega}_{j, \kt}^K$ 
(respectively $\bar{\Omega}_{j, \kt, \dagger}^K \subset \Omega_j^{K} \times {\Omega}_{j, \kt, \dagger}^K$) 
is the space of pairs $(F_2^0 (\cdot ; 0), F_2^{\rr} )$ 
(respectively $(F_3^0 (\cdot ; 0), F_3^{\rr} )$) such that 
$F_2^{\rr} (X, \varphi ; \tau) =  F_2^0 (X, \varphi ; 0)$ (respectively $F_3^{\rr} (X, \varphi ; \tau) =  F_3^0 (X, \varphi ; 0)$) whenever $X\in \cP_j$ has $X \cap (P_y^j)^* = \emptyset$.
Then denote
\begin{align}
	\norm{ (F_2^0, F_2^{\rr})}_{\bar{\Omega}_{j, \kt}^K} 
	& 
	= \max\{ \norm{F_2^0 (\cdot ; 0)}_{h, T_j} ,  \norm{F_2^{\rr}}_{\vec{h},  T_j}   \}  
	\\
	\norm{ (F_3^0, F_3^{\rr})}_{\bar{\Omega}_{j, \kt, \dagger}^K} 
	& 
	= \max\{ \norm{F_3^0 (\cdot ; 0)}_{h, T_j} ,  \norm{F_3^{\rr}}_{\vec{h},  T_j, A/2}   \}
.
\end{align}

\item Define $\Omega_j = \Omega_j^{U} \times \Omega_j^K$,  
$\bar{\Omega}_{j,\kt} = \Omega_j^U \times \bar{\Omega}_{j, \kt}^K$
and
$\bar{\Omega}_{j,\kt, \dagger} = \Omega_j^U \times \bar{\Omega}_{j, \kt, \dagger}^K$
endowed with norms
\begin{align}
	\norm{(U_j, F_1)}_{\Omega_j} 
	&= \max \{ \norm{U_j}_{\Omega_j^U} , \norm{F_1}_{\Omega_j^K} \} 
	\\
	\norm{(U_j, F_2^0, F_2^{\rr})}_{\bar{\Omega}_{j, \kt}} 
	&= \max \{ \norm{U_j}_{\Omega_j^U} ,  \norm{(F_2^0, F_2^{\rr})}_{\bar{\Omega}^K_{j, \kt}} \}
	\\
	\norm{(U_j, F_3^0, F_3^{\rr})}_{\bar{\Omega}_{j, \kt, \dagger}} 
	&= \max \{ \norm{U_j}_{\Omega_j^U} ,  \norm{(F_3^0, F_3^{\rr})}_{\bar{\Omega}^K_{j, \kt, \dagger}} \}
\end{align}
\end{itemize}
\end{definition}

An element of the space $\Omega_j$ or $\bar{\Omega}_{j, \kt}$ or $\bar{\Omega}_{j, \kt, \dagger}$ is typically denoted $\omega_j^0$ or $\omega_j^{\rr}$ or $\omega_j^{\rr, \dagger}$, respectively. 
Given an element $\omega_j^{\rr} = (U_j, K_j^0, K_j^{\rr}) \in \bar{\Omega}_{j, \kt}$,  we have $(K_j^{\rr})^{\dagger}$ defined by Definition~\ref{def:K_j^dagger_definition}. 
Then by Proposition~\ref{prop:reblocking_Z_with_Psi} and Lemma~\ref{lemma:reblocking_estimate}, we see that
$\omega_j^{\rr, \dagger} := (U_j, K_j^0, (K_j^{\rr})^{\dagger} ) \in \bar{\Omega}_{j, \kt, \dagger}$
and satisfy
\begin{align}
\norm{\omega_{j}^{\rr, \dagger}}_{\bar{\Omega}_{j, \kt, \dagger}} \leq C \norm{\omega_{j}^{\rr}}_{\bar{\Omega}_{j, \kt}}
.
\label{eq:omega_ddag_omega_bound}
\end{align}
Thus a bounded function $\omega_j^{\rr,\dagger} \mapsto (\cE_{j+1}, \kg_{j+1}^\rr,  \omega_{j+1}^{\rr})$ can be used to define a bounded map $\omega_j^{\rr} \mapsto (\cE_{j+1}, \kg_{j+1}^\rr, \omega_{j+1}^{\rr} )$.

We state the existence and estimates on an RG map in the main theorems of this section.

\begin{theorem}
\label{thm:rg_map_definition}

Let $L, A , \beta > 0$ be sufficiently large. 
Then there exists $\epsilon_{nl} \equiv \epsilon_{nl} ( \beta, A, L) > 0$  (only polynomially small in $\beta$) and an RG map (as in Definition~\ref{def:rgmap})
\begin{align}
	\Phi^{\rr, \Lambda_N}_{j+1} : \big\{ \omega_j^{\rr} \in \bar{\Omega}_{j,\kt} : \norm{\omega_j^{\rr}}_{\bar{\Omega}_{j, \kt}} \leq \epsilon_{nl}  \big\} 
	& \rightarrow 
	\R \times W^+ (\D_{\htau}) \times \Omega_{j+1}^{U} \times \bar{\Omega}^{K}_{j+1, \kt} 
\end{align}
with coordinate maps $(\cE_{j+1}, \kg_{j+1}^\rr, \cU_{j+1}, \vec{\cK}_{j+1})$.
\end{theorem}

\begin{theorem}
\label{thm:rg_map_definition_estimates}

With $\Phi_{j+1}^{\rr,\Lambda_N} = (\cE_{j+1}, \kg_{j+1}^\rr, \cU_{j+1}, \vec{\cK}_{j+1})$ as in Theorem~\ref{thm:rg_map_definition},
let $(\cK_{j+1}^0, \cK_{j+1}^\rr)$ be the two components of $\vec{\cK}_{j+1}$.
We can decompose $\cK_{j+1}^{\rr} = \cL_{j+1}^{\rr} + \cM_{j+1}^{\rr}$ with a linear map $\cL_{j+1}^{\rr}$ of $\vec{K}^{\rr}_{j}$ and a differentiable map $\cM_{j+1}$ of $(U_j, \vec{K}_j^{\rr})$ satisfying
\begin{align}
	& \norm{\cL_{j+1}}_{\Omega^K_{j+1, \rr}} \leq C_1 (L^2 \alphaLoc \norm{K_j^0 (\cdot ; 0)}_{\Omega_j^K} + \alphaLoco \norm{\vec{K}^{\rr, \dagger}_j}_{\bar{\Omega}^K_{j, \rr, \dagger}} ) \label{eq:rg_map_main_estimate1} \\
	& \norm{D \cM_{j+1}}_{\Omega^K_{j+1, \rr}} \leq C_2 \norm{\omega_j^{\rr, \dagger} }_{\bar{\Omega}_{j, \kt, \dagger}}
	\label{eq:rg_map_main_estimate2}
\end{align}
for $(\alphaLoc, \alphaLoco)$ as in \eqref{e:Loc-contract-kappa}
and some $C_1 \equiv C_1 (\rr) >0$ independent of $\beta, A,L$ and $C_2 \equiv C_2 (\rr, \beta, A,L)$.
\end{theorem}

We note that $\Omega_j^K$ and $\Omega_{j, \kt}^K$ are vector spaces of collections $(F(X, \cdot))_{X\in \Conn_j}$, but not of $(F(X, \cdot))_{X\in \cP_j}$. 
To say that a function on $\Omega_j^K$ or $\Omega_{j,\kt}^K$ is linear means that it is linear in $(F(X, \cdot))_{X\in \Conn_j}$.

As always, at scale $j \in \{0, \cdots, N-1\}$,  $\Eplus$ always mean the expectation taken over $\zeta \sim \cN (0, \Gamma_{j+1})$ if $j+1 <N$ and $\zeta \sim \cN (0, \Gamma_N^{\Lambda_N})$ if $j+1 = N$.
Usually, the remaining non-random coordinate is denoted $\varphip$ and $\varphi = \varphip + \zeta$.

\subsection{Choice of counterterm and vacuum energy}
\label{sec:vacuum_energy}

When $\cz = \rr = \tau = 0$,  
$Z^{0}_j (\varphi ; 0)$ is described by $(E_j, s_j,  (z_j^{(q)})_{q\geq 1})$ and $K_j^0 (\cdot ; 0)$ with $\kg^\rr_j \equiv 0$.
These are studied in \cite{dgauss1} in detail.
In the RG language,
$\frac{1}{2} s_0 |\nabla \varphi|^2_{\Lambda_N}$ in \eqref{eq:Z_0_definition} is the counterterm,
and $\lim_{j \rightarrow \infty} E_j$ is the vacuum energy.
For the RG steps to proceed,
counterterm is chosen so that coordinates $(s_j, (z_j^{(q)})_{q\geq 1}, K_j^0 (\cdot ; 0)) \rightarrow 0$ in the norms of $\Omega_j$'s. 
This is an important reference point for our analysis, as our analysis is a perturbation of the model with vanishing external field. 
Thus we recall the choice of the counterterm and the vacuum energy. 

In the following, 
functions $(\cE_{j+1}, \cU_{j+1} ) : \Omega_j \rightarrow \R \times \Omega_{j+1}^U$, $(U_j, K^0_j) \mapsto (E_{j+1}, U_{j+1})$ are as in \cite[Definition~7.8]{dgauss1}, so that they solve
\begin{align}
\begin{split}
	&-\cE_{j+1} (U_j, K^0_j) |B| + \cU_{j+1} (U_j, K_j, B, \varphip) \\
	&\qquad \qquad \qquad \qquad \qquad = \E U_j (B, \varphip + \zeta) + \sum_{X\in \cS_j : X \supset B} \Loc_{X, B} \E K^0_j (X, \varphip + \zeta)
	\label{eq:cU_definition}
	.
\end{split}	
\end{align}

\begin{definition}
We say that bulk RG flow of length $N$ exists with initial condition $(s_0,  (z_0^{(q)})_{q \ge 1} )$ if the following hold.
Take $s=s_0$ (for $s$ in \eqref{eq:C_m^2_definition}) and define $Z_0$ using parameters $(s_0,z_0$) (as in Lemma~\ref{lemma:first_reformulation_Z2}) and $K^0_0 (X)= 1_{X=\emptyset}$.
Then there exists $(K_j^0 )_{j \in \{1, \cdots, N\}} \in \otimes_{j=1}^N \Omega_j^K$ such that 
\eqref{eq:bulk_RG_flow} and \eqref{eq:Z_j_generic_expression} hold 
with $\cz = \rr = \tau = 0$,
$E_{j+1} = E_j + \cE_{j+1}$,  $g_{j+1} \equiv 0$, 
and
$U_{j+1} = \cU_{j+1} (U_j, K^0_j (\cdot ; 0) )$,
i.e.,
\begin{align}
	e^{-(E_{j+1} - E_j) |\Lambda_N | }\sum_{Z \in \cP_{j+1}} e^{U_{j+1} (\Lambda_N \backslash Z,  \,  \cdot)} K_{j+1}^0 (Z, \, \cdot) = \Eplus^{\zeta}_{\Gamma_{j+1}} \sum_{X \in \cP_{j}} e^{U_j (\Lambda \backslash X, \cdot + \zeta)} K_j^0 (X, \cdot + \zeta)
\end{align}
\end{definition}

The following theorem makes a specific choice of the counterterm by 
putting together \cite[Theorem~7.5, Corollary~8.7]{dgauss1}, 
and also guarantees the existence of the RG flow of any length,
where we recall that $\alphaLoc$ is defined in \eqref{e:Loc-contract-kappa}.

\begin{theorem}
\label{thm:tuning_s-v2}

Let $\beta \geq \beta_0$, $L \geq  L_0$ and $A \geq A_0 (L)$ be sufficiently large.
Then there exists $s_0^c (\beta) = O(e^{-c_f \beta})$ such that
the bulk RG flow of length $N$ exists with initial condition $(s_0,  (z_0^{(q)})_{q \ge 1}) = (s_0^c (\beta),  (\tilde{z}^{(q)})_{q \ge 1} )$ (as in Lemma~\ref{lemma:first_reformulation_Z2}) for any $N\geq 1$.
Moreover, the RG coordinates satisfy
\begin{equation}\label{eq:finalbounds}
    \|U_j\|_{\Omega_j^U} \leq O(e^{-c_f \beta}L^{-\alpha j}), \qquad \| {K}^0_j \|_{\Omega_j^{{K}}} \leq O(e^{- c_f \beta}L^{-\alpha j})
  \end{equation}
for $\alpha > 0$ such that $L^{\alpha} = O\big( L^{-2} (\alphaLoc)^{-1} \big)$, 
and the bounds are uniform in $\Lambda_N$. 
 
\end{theorem}

Existence of such $s_0^c (\beta)$ can be seen from linear approximations \cite[(8.9)--(8.11)]{dgauss1} of the map $(s_j, z_j, K_j) \mapsto (s_{j+1}, z_{j+1}, K_{j+1})$:
\begin{align}
	s_{j+1} & = s_j + O(K_j^0), \quad \\
	z_{j+1}^{(q)} & = L^2 e^{-\frac{1}{2} \beta \Gamma_{j+1} (0,0) } z_{j}^{(q)},  \quad \\
	K^0_{j+1} & = \cL^0_{j+1} (K^0_j) + O \big(\norm{K_j^0}_{\Omega_j^K} + \norm{U_j}_{\Omega_j^U}  \big)^2
\end{align}
where $\cL_{j+1}^0$ is a linear map with its (operator) norm $<1$ for sufficiently large $\beta$.
Then $s_0^c (\beta)$ is constructed as a stable manifold of the dynamical system $(s_j,z_j, K_j)_{j\geq 1}$ with given $z_0$.

Coefficient of the counterterm, $s_0^c (\beta)$, is a fundamental quantity related to the DG model in the delocalised phase,
so we emphasise the initial condition again.

\begin{equation} \stepcounter{equation}
	\tag{\theequation $\rginitial$} \label{quote:rginit}
	\parbox{\dimexpr\linewidth-4em}{%
	Given $\beta \geq \beta_0$, 
the parameter $s$ is set to be $s_0^c (\beta)$, 
the initial coupling constants $U_0 (X, \varphi) = \frac{1}{2} s_0 |\nabla \varphi|^2_X + \sum_{x\in X} \sum_{q\geq 1} z_0^{(q)} \cos(q \beta^{1/2} \varphi(x))$ are given by $s_0=s_0^c(\beta)$, $z_0^{(q)}$ as in Lemma~\ref{lemma:first_reformulation_Z2} and $K_0 (X) = 1_{X = \emptyset}$.
	}
\end{equation}

\subsection{RG map}

Suppose we are given $(E_j,  g^{\rr}_j, U_j, K_j^{\rr})$ and $Z^{\rr}_j$ is in form \eqref{eq:Z_j_generic_expression}.
Then by Proposition~\ref{prop:reblocking_Z_with_Psi}, 
\begin{align}
	Z^{\rr}_j (\varphi  ; \tau) = e^{-E_j |\Lambda_N | + g^{\rr}_j (\Lambda_N ; \tau) } \sum_{X\in \cP_j (\Lambda_N)} e^{U_j (\Lambda_N \backslash X , \varphi ) } (K^{\rr}_j)^{\dagger} (X, \varphi ; \tau)
	.
\end{align}
So we construct our RG map starting from $(U_j,  (K^{\rr}_j)^{\dagger})$.

Given coordinates $U_j$ and $\vec{K}_j^{\rr,\dagger} = (K_j^0, (K_j^{\rr})^{\dagger} ) \in \bar{\Omega}_{j, \kt, \dagger}^K$, we use the map $(U_j, K^0_j(\cdot ; 0)) \mapsto (\cE_{j+1}, \cU_{j+1}, K^0_{j+1} (\cdot ; 0))$ defined according to \eqref{eq:cU_definition} and Theorem~\ref{thm:tuning_s-v2}. 
In the following,  we recall $P_y^j = \cup_{k=0}^4 B_k^j$ (see \eqref{eq:union_of_B^j_k's}).

\begin{definition} \label{def:evolution_of_g_j}

For $0 \leq j  < N$ and given $\vec{K}_j^{\rr, \dagger} \in \bar{\Omega}_{j, \kt, \dagger}^K$,  define 
\begin{align}
\mathfrak{g}^{\rr}_{j+1} (B,  \vec{K}^{\rr}_j ; \tau) = 1_{B \subset (P_y^j)^*}
 \sum_{Y \in \cS_j}^{Y \supset B} \frac{1}{|Y \cap (P_y^j)^* |_j}  \Loco_Y \E_{(\rr + \tau)} \big[ 
 (K^{\rr}_j )^{\dagger} (Y, \zeta ; \tau) - K_j^0 (Y, \zeta ; 0) \big] 
\label{eq:e_j+1_definition} 
\end{align}
for $B \in \cB_j$, and for $X\in \cP_{j+1}$, let
\begin{align}
\mathfrak{g}^{\rr}_{j+1} (X,  \vec{K}^{\rr,\dagger}_j ; \tau) 
= \sum_{B\in \cB_j (X)}  \mathfrak{g}_{j+1}^{\rr} (B,  \vec{K}^{\rr,\dagger}_j ; \tau) .
\end{align}
\end{definition}

It is an immediate consequence of Lemma~\ref{lemma:Loc_Et_K_j_bound} that, if 
$\norm{\vec{K}_j^{\rr,\dagger}}_{\bar{\Omega}^K_{j, \kt,\dagger}} < \infty$,
then $\kg^{\rr}_{j+1} (B, \vec{K}^{\rr}_j ; \tau)$ is an analytic function of $\tau \in \D_{\htau}$ with bound
\begin{align}
 \norm{  \mathfrak{g}_{j+1}^{\rr} (B,  \vec{K}^{\rr,\dagger}_j ; \cdot) }_{\htau,  W} 
 \leq C(M, \rho,\rr) A^{-1} \norm{\vec{K}_j^{\rr,\dagger}}_{\bar{\Omega}^K_{j, \kt,\dagger}}
 \label{eq:kg_are_analytic}
\end{align}
for each $B\in \cB_j$. 
The renormalisation group map for $K$-coordinate is defined as the following. 

\begin{definition} \label{def:expression_for_K_j+1_external_field}

Under the same assumptions, 
the map $\omega^{\rr, \dagger}_j = (U_j, \vec{K}^{\rr, \dagger}_j ) \mapsto K^{\rr}_{j+1}$ is defined by
\begin{align}
	\label{eq:expression_for_K_j+1_external_field}
	& \mathcal{K}^{\rr}_{j+1}  (U_j,  \vec{K}_j^{\rr,\dagger},   X,  \, \cdot \,  ; \tau )
	= \sum_{X_0, X_1, Z, (B_{Z''})}^{*} e^{ \cE_{j+1} |T| - \kg^{\rr}_{j+1} (T ; \tau) } e^{\cU_{j+1} (X \backslash T, \cdot + (\rr + \tau) u_{j+1})} 
	\nnb
	& \quad
	\times \E_{(\rr + \tau)} \Big[ (e^{U_j (B, \cdot)} - e^{ -\cE_{j+1} |B| + \kg^{\rr}_{j+1} (B ; \tau) + \cU_{j+1} (B, \cdot + (\rr + \tau) u_{j+1}) } )^{X_0} (\bar{K}^{\rr}_j ( \cdot + \zeta ; \tau) - \mathcal{E} K^{\rr}_j (\cdot ; \tau))^{[X_1]}  \Big] 
	\nnb
	& \quad \times \prod_{Z'' \in \operatorname{Comp}_{j+1} (Z)} J^{\rr}_j (B_{Z''}, Z'' ,  \, \cdot \,  ; \tau)
	,
\end{align}
whenever the integral $\E_{(\rr + \tau)}$ converges, 
where $\sum^{*}$
is running over disjoint $(j+1)$-polymers $X_0, X_1, Z$ such that
$X_1 \not\sim Z$,  $B_{Z''} \in \cB_{j+1} (Z'')$ for each $Z'' \in \operatorname{Comp}_{j+1} (Z)$,
$T = X_0 \cup X_1 \cup Z$ and $X = \cup_{Z''}B^*_{Z''} \cup X_0 \cup X_1$,  and
\begin{align}
Q^{\rr}_j ( D, Y, \varphip ; \tau) &= 1_{Y\in \cS_j} \Big( \Loc_{Y, D} \E_{(\rr + \tau)} [  K^0_{j} ( Y, \varphip + \zeta ; 0 ) ]  \\
& \qquad \quad \; +  \frac{1_{D \in ( Y \cap ( P_y^j )^* ) }}{|Y \cap (P_y^j)^* |_j} \Loco_Y \E_{(\rr + \tau)} \big[ 
 (K^{\rr}_{j} ) ^{\dagger} (Y, \varphip + \zeta  ; \tau) - K^0_{j} (Y,  \varphip + \zeta ; 0) \big]   \Big)
\nnb
J^{\rr}_j (B, X, \varphip ; \tau)  &= 1_{B\in \cB_{j+1} (X)} \sum_{D\in \cB_j (B)} \sum_{Y\in \cS_j}^{D\in \cB_j (Y)} Q^{\rr}_j (D, Y, \varphip ; \tau) ( 1_{\bar{Y} = X} - 1_{B=X} ) . \label{eq:J_j^Psi_definition} 
\\
\mathcal{E} K^{\rr}_j ( X, \varphip ; \tau ) &= \sum_{B\in \mathcal{B}_{j+1} (X)} J^{\rr}_j (B, X, \varphip ; \tau) \label{eq:cEK^Psi_definition} \\
\bar{K}^{\rr}_j (X, \varphip + \zeta ; \tau) &= \sum_{Y\in \mathcal{P}_j}^{\bar{Y}=X} e^{U_j (X \backslash Y, \varphip + \zeta )}  (K^{\rr}_j )^{\dagger} (Y, \varphip + \zeta ; \tau  ) 
\label{eq:K_bar^Psi_definition}
\end{align}
for $D\in \cB_j$, $B\in \cB_{j+1}$, $Y\in \cP_j$ and $X\in \cP_{j+1}$.
\end{definition}

\begin{remark} \label{remark:algebraic_part_is_trivial}
As is apparent in the notation, if we set $\rr = \tau = 0$ and
$K_{j+1}^0 = \cK_{j+1}^0 (U_j, \vec{K}_j^{0, \dagger})$, then this exactly corresponds to the choice of $K_{j+1}^0$ in Theorem~\ref{thm:tuning_s-v2}.
See \cite[Definition~7.9]{dgauss1}.

Also, by comparing the expressions for $\cK_{j+1}^0$ and $\cK_{j+1}^{\rr}$ and the proof of \cite[Theorem~7.5]{dgauss1}, \eqref{eq:bulk_RG_flow} follows with the RG coordinates at scale $j+1$ given by \eqref{eq:rg_map_definition}, only with minor adaptations.
\end{remark}

\ifx\newpageonoff\undefined
{\red command undefined!!}
\else
  \if\newpageonoff1
  \newpage
  \fi
\fi

\section{Proof of Theorem~\ref{thm:rg_map_definition} and \ref{thm:rg_map_definition_estimates}}
\label{sec:proof-of-theorem-thm-local-part-of-K}

As was mentioned in Remark~\ref{remark:algebraic_part_is_trivial}, the algebraic part of Theorem~\ref{thm:rg_map_definition} will not be discussed in detail.  We focus on the proof of the bounds \eqref{eq:rg_map_main_estimate1} and \eqref{eq:rg_map_main_estimate2}.  The next theorem is essentially just a restatement of the estimates.

\begin{theorem}[Estimate for remainder coordinate] 
  \label{thm:local_part_of_K_j+1}
  
Under the setting of Theorem~\ref{thm:rg_map_definition} and on each $X\in \cP_{j+1}^c$,  the map $\cK_{j+1}^{\rr}$ admits a decomposition
\begin{align}
\mathcal{K}^{\rr}_{j+1} ( \omega_j^{\rr,\dagger} , X  ; \tau)
= \mathcal{L}^{\rr}_{j+1} ( \omega_j^{\rr,\dagger} , X ; \tau) + \mathcal{M}^{\rr}_{j+1} ( \omega_j^{\rr,\dagger} , X; \tau)
\end{align}
into polymer activities at scale $j+1$ such that the following holds provided $L \geq L_0 (\rr)$,  $A \geq A_0(L)$:
\begin{itemize}
\item[(i)] When $X\in \Conn_j$, the map $\omega_j^{\rr, \dagger} \mapsto \mathcal{L}^{\rr}_{j+1} ( \omega_j^{\rr} , X)$ is independent of $U_j$, linear in $\vec{K}_j^{\rr, \dagger}$
such that
\begin{align}
	\norm{\mathcal{L}^{\rr}_{j+1} ( \omega_j^{\rr, \dagger}  )}_{\vec{h}, T_{j+1}^0} 
	\leq C_1
	\big( L^2\alphaLoc 
	\norm{K^0_j (\cdot ; 0)}_{\Omega_j^K} 
	+ \alphaLoco \norm{\vec{K}_j^{\rr, \dagger}}_{\bar{\Omega}^K_{j, \kt, \dagger}} \big)
	.
   \label{eq:bound_for_L_j_K_j}
\end{align}
for some constant $C_1 \equiv C_1 (\rr) >0$ independent of $\beta$, $A$, $L$,
and $(\alphaLoc, \alphaLoco)$ are as in \eqref{e:Loc-contract-kappa}.

\item[(ii)] The remainder map $\mathcal{M}_{j+1}$ satisfies $\mathcal{M}_{j+1} = O( \norm{\omega_j^{\rr,\dagger}}_{\bar{\Omega}_{j, \kt, \dagger}}^2 )$ 
in the sense that there exists $C_2 =C_2 (\rr, \beta,A, L) >0$ such that $\mathcal{M}_{j+1} (\omega_j^{\rr,\dagger})$ is continuously Fr\'{e}chet-differentiable and 
\begin{align}
	& \norm{D \mathcal{M}_{j+1}  (\omega_j^{\rr, \dagger}) }_{\vec{h}, T_{j+1}^0}  
	\leq C_2  \norm{ \omega_j^{\rr, \dagger} }_{\bar{\Omega}_{j, \kt, \dagger}}
\label{eq:bound_for_derivative_of_Nj}
\end{align}
and $\cM_{j+1} (0,0) = 0$.
\end{itemize}
\end{theorem}

\begin{proof}[Proof of Theorem~\ref{thm:rg_map_definition}]

Let $(\cE_{j+1}, \cU_{j+1})$ be given by \eqref{eq:cU_definition}.
Given $K^{\rr, \dagger}_{j}$ be given by \eqref{eq:K_j^dagger_definition},
$\kg^{\rr}_{j+1} \equiv \kg^{\rr}_{j+1} (\vec{K}^{\rr, \dagger}_{j})$ be given by Definition~\ref{def:evolution_of_g_j} and $\vec{\cK}^{\rr}_{j+1} \equiv \vec{\cK}^{\rr}_{j+1} (U_j, \vec{K}^{\rr, \dagger}_{j})$ be given by Definition~\ref{def:expression_for_K_j+1_external_field}.
Then $\kg^{\rr}_{j+1} \in W^+ (\D_{\htau})$ by the remark after Definition~\ref{def:evolution_of_g_j}, 
and $\norm{K_{j+1}^{\rr}}_{\vec{h}, T_{j+1}} < \infty$ by Theorem~\ref{thm:local_part_of_K_j+1}, 
as $\norm{\omega_j^{\rr}}_{\bar{\Omega}_{j, \kt}} \leq \epsilon_{nl}$ would imply $\norm{\omega_j^{\rr, \dagger}}_{\bar{\Omega}_{j, \kt, \dagger}} \leq C \epsilon_{nl}$ for some $C>0$.
Also, since any $X \in \cP_{j+1}$ such that $X \cap (P_y^{j+1})^* = \emptyset$ also satisfies $X \cap (P_y^j)^* = \emptyset$, we see that each polymers appearing in \eqref{eq:expression_for_K_j+1_external_field} is disjoint from $(\supp (u_{j+1}) P_y^j)^*$ ($*$ is taken in scale $j$) for such $X$, which means that $K_{j+1}^{\rr} (X ; \tau) = K_{j+1}^0 (X ; 0)$ for any $\tau$. Therefore we have
$\vec{\cK}_{j+1}^{\rr} \in \bar{\Omega}^K_{j+1, \kt}$.

Finally,  \eqref{eq:bulk_RG_flow} follows from the proof of \cite[Theorem~7.5]{dgauss1}, 
only with minor adaptations to the current setting.
\end{proof}

\begin{proof}[Proof of Theorem~\ref{thm:rg_map_definition_estimates}]

The decomposition and the estimates are direct from Theorem~\ref{thm:local_part_of_K_j+1}.
\end{proof}

\subsection{Proof of Theorem~\ref{thm:local_part_of_K_j+1}~(i)}

Given $\cK_{j+1}^{\rr} ( \omega_j^{\rr, \dagger} )$ defined according to \eqref{eq:expression_for_K_j+1_external_field},
consider its formal expansion in linear order of $(U_j, K_j^0 (\cdot ; 0),  (K^{\rr}_j)^{\dagger})$ by 
(1) replacing $e^{ \cE_{j+1} |T| - \kg^{\rr}_{j+1} (T) + \cU_{j+1} (X \backslash T, \cdot + (\rr + \tau) u_{j+1})}$ by 1,  $e^{U_j} - e^{ -\cE_{j+1} |B| + \kg^{\rr}_{j+1} (B) + \cU_{j+1} (B, \cdot + (\rr + \tau) u_{j+1}) }$ by $U_j - \cE_{j+1} |B| + \kg^{\rr}_{j+1} (B) + \cU_{j+1} (B, \cdot + (\rr + \tau) u_{j+1})$,
(2) extracting out terms with 
\begin{align}
\# (X_0, X_1, Z) = |X_0|_{j+1} + |\Comp_{j+1} (X_1)| + |\Comp_{j+1} (Z)| 
\leq 1
,
\end{align}
and (3) approximating $\bar{K}_j^{\rr} (X)$ by $\mathbb{S}[ (K_j^{\rr})^{\dagger} ] (X)$.
These give us the expression 
\begin{align}    
	\mathcal{L}^{\rr}_{j+1} (\omega_j^{\rr, \dagger} ) ( X, \varphip) 
	& := \sum_{Y : \bar{Y} = X} \Big(  1_{Y\in \Conn_j}  \E_{(\rr + \tau)}  (K^{\rr}_j)^{\dagger} (Y, \varphip+\zeta ) 
	- 1_{Y\in \cS_j} \sum_{D\in \cB_j (Y)} Q^{\rr}_j (D, Y, \varphip) \Big) 
	 \nnb
	& \qquad + \sum_{D \in \cB_{j}}^{\bar{D}=X}  \Big( \E_{(\rr + \tau)} [ U_j (D, \varphip+\zeta)] + \cE_{j+1} |D|  - \mathfrak{g}^{\rr}_{j+1} (D) \nnb
	& \qquad \qquad \qquad -  \cU_{j+1} (D, \varphip + (\rr + \tau) u_{j+1}) 
	- \sum_{Y\in \cS_j}^{D\in \cB_j (Y)}  Q^{\rr}_j (D, Y, \varphip) \Big).
\label{eq:L_j_K_j_0}
\end{align}
(See \cite[(7.41)]{dgauss1} for a detailed treatment for a similar computation.)
Also by the definition of $Q_j$ and change of variable 
(justified by Lemma~\ref{lemma:Et_by_complex_shift}),  
\begin{align}
	& \sum_{Y \in \cS_j}^{Y \supset D} Q^{\rr}_j (D, Y, \varphip)  \nnb
	& = \sum_{Y \in \cS_j}^{Y \ni D} \Loc_{Y, D} \E_{(\cz)} [ K^0_j (Y, \varphip + \zeta ; 0)] 
	+ \sum_{Y \in \cS_j}^{Y \supset D} \frac{1_{D\in (P_y^j)^*}}{|Y \cap (P_y^j)^*|_j} \Loco_Y \E_{(\cz)} D_j^{\rr} (Y,  \varphip + \zeta ; \tau) \nnb
	& = \sum_{Y \in \cS_j}^{Y \supset D} \Loc_{Y, D} \E [K^0_j (Y, \varphip + \zeta + \cz u_{j+1} ; 0)] 
	+ \sum_{Y \in \cS_j}^{Y \supset D} \frac{1_{D\in (P_y^j)^*}}{|Y \cap (P_y^j)^*|_j} \Loco_Y \E_{(\cz)} D_j^{\rr} (Y,  \varphip + \zeta ; \tau) 
\end{align}
where $\cz = \rr + \tau$ and
\begin{align}
	D_j^{\rr} (Y, \varphi; \tau) = (K^{\rr}_{j} )^{\dagger} (Y, \varphi ; \tau) - K^0_{j} (Y, \varphi ; 0).
\end{align}
But by \eqref{eq:cU_definition} and Definition~\ref{def:evolution_of_g_j}, we see that the second line of \eqref{eq:L_j_K_j_0} vanishes. 
Hence
\begin{align}
\cL^{\rr}_{j+1} (\omega_j^{\rr,\dagger})  = \sum_{b=1,2,3}  \cL_{j+1}^{\rr, (b)} (\omega_j^{\rr,\dagger}) 
\end{align}
where
\begin{align}
	\cL_{j+1}^{\rr,  (1)} &= \sum_{Y : \bar{Y} = X} 1_{Y\in \cS_j}  \big( 1- \Loc_Y \big) \E_{(\cz)} [K^0_j ( Y, \varphip+\zeta ; 0) ]  \\
	\cL_{j+1}^{\rr,  (2)} &= \sum_{Y : \bar{Y} =X} 1_{Y\in \cS_j} \big(1- \Loco_Y\big) \E_{(\cz)} [ D^{\rr}_j (Y, \varphip + \zeta ; \tau) ]  \\
	\cL_{j+1}^{\rr,  (3)} &= \sum_{Y : \bar{Y} =X} 1_{Y\in \Conn_j \backslash \cS_j} \E_{(\cz)} [ (K^{\rr}_j)^{\dagger} (Y, \varphip + \zeta ; \tau) ] \nnb
	&= \mathbb{S}\big[ 1_{Y\in \Conn_j \backslash \cS_j} \E_{(\cz)} [ (K^{\rr}_j)^{\dagger} (Y, \varphip + \zeta ; \tau) ]  \big] (X)
	\label{eq:L_j_K_j^(k)}
\end{align}
For $b=1$, by Proposition~\ref{prop:Loc-contract_v2}, we have
\begin{align}
\norm{ \cL_{j+1}^{\rr,(1)} (X) }_{\vec{h}, T_{j+1} (X)} \leq C L^2 \alphaLoc A^{-|X|_j} \norm{K^0_j (\cdot ; 0)}_{h, T_j} .
\end{align}
where $L^2$ factor originates from the fact that there are at most $O(L^2)$ small polymers $Y$ such that $\bar{Y} = X$.
For $b=2$,  we see from Proposition~\ref{prop:reblocking_Z_with_Psi} that, $(K_j^{\rr} )^{\dagger} (Y) = K^0_j (Y; 0)$, equivalently $D^{\rr} (Y) = 0$,  unless $Y \cap (P_y^j)^* \neq \emptyset$. 
So 
\begin{align}
\cL_{j+1}^{\rr,(2)} &  = \sum_{Y : \bar{Y} =X}^{Y\cap (P_y^j)^* \neq \emptyset} 1_{Y\in \cS_j} \big(1- \Loco_Y\big) \E_{(\cz)} [ D^{\rr}_j (Y, \varphip + \zeta ; \tau) ]
\end{align}
but then by Proposition~\ref{prop:Loc-contract_v2}, 
we have
\begin{align}
\big\| \big(1- \Loco_Y\big) \E_{(\cz)} [ D^{\rr}_j  (Y, \varphip + \zeta ; \tau) ] \big\|_{\vec{h}, T_j (Y)} \leq C \alphaLoco (A/2)^{-|X|_j} \norm{\omega^{\rr}_j}_{ \bar{\Omega}_{j, \kt}}.
\end{align}
Finally, for $b=3$,  Proposition~\ref{prop:largeset_contraction_external_field} gives
\begin{align}
	\norm{ \cL_{j+1}^{\rr, (3)} }_{\vec{h}, T_{j+1} (X)} \leq (2L^{-1} A^{-1})^{|X|_{j+1}} \norm{( K^{\rr}_j )^{\dagger} }_{\vec{h}, T_{j}} 
	.
\end{align}

\subsection{Bound on the non-linear part}
\label{subsec:M_j+1_decomposition}

The terms with order $\geq 2$ can be identified by following the linearisation process of $\cK^{\rr}_{j+1}$ backwards. 
If we denote
\begin{align}
\bar{U}_{j+1} (X, \varphip) = - \cE_{j+1} |X| + \kg_{j+1}^{\rr} (X) + U_{j+1} (X, \varphip + (\rr + \tau) u_{j+1})
\end{align}
and
\begin{align}
\bar{\kK}_j (U_j,  \vec{K}^{\rr, \dagger}_j) \equiv \bar{\kK}_j (\omega^{\rr, \dagger}_j) 
= (\cE_{j+1} |X| - \kg_{j+1}^{\rr} (X) , U_j, \bar{U}_{j+1},  (K^{\rr}_j)^{\dagger} , \bar{K}^{\rr}_j, \cE K^{\rr}_j,  J^{\rr}_j) (\omega_j^{\rr, \dagger})
\label{eq:bar_kK_j_definition}
\end{align}
then $\cM^{\rr}_{j+1} := \cK^{\rr}_{j+1} (\omega_{j}^{\rr,\dagger}) - \cL^{\rr}_{j+1} (\omega_{j}^{\rr,\dagger})$ can be considered as a function
\begin{align}
\cM_{j+1}^{\rr} (\omega_j^{\rr,\dagger}) =\kM_{j+1}^{\rr} (\bar{\kK}_j (\omega^{\rr,\dagger}_j), X, \varphip )
\end{align}
This can be compared to \cite[(7.52)--(7.55)]{dgauss1}. 
The differentiability of $\kM^{\rr}_{j+1}$ can be checked from regularity estimates of their components, as summarised following.

\begin{align} 
	& \parbox{\dimexpr\linewidth-4em}{%
	(1) For $B \in \cB_{j+1}$, $Z \in \cP_{j+1}$, $\varphi \in \R^{\Lambda_N}$ and $k\in \{0,1,2\}$,
	}
	\nnb
	& \qquad\quad \norm{\mathfrak{U} (B, \varphi ; \cdot)}_{\vec{h}, T_j (B, \varphi)} \leq 
	C(\delta, L) ( 1+  \delta c_w \kappa  w_j (B, \varphi)^2 ) |x| \label{eq:Ujbound_1}  \\  
	& \qquad\quad \Big\| e^{\mathfrak{U} (B, \varphi ; \cdot)} - \sum_{m=0}^{k} \frac{1}{m!} (\mathfrak{U}(B, \varphi ; \cdot ))^m \Big\|_{\vec{h}, T_{j} (B, \varphi)} \leq 
	C(\delta, L)  e^{\delta c_w \kappa  w_j (B, \varphi)^2} |x|^{k+1} \label{eq:Ujbound_2}, \\	
	& \parbox{\dimexpr\linewidth-4em}{%
	for $\mathfrak{U}\in \{ U_j, \bar{U}_{j+1}\}$ and some $C(L)$, and the same inequalities hold with $\mathfrak{U} (B)$ and 
	$C(\delta, L)$ replaced by 
	$\cE_{j+1} |B|$ or $- \kg^{\rr}_{j+1} (B)$ 
	and $C(L)$, respectively, but $\delta$ set to $0$. 
	}	
	\stepcounter{equation}
	\tag{\theequation $\regularityone$} \label{quote:regularity_one}
\end{align}
\begin{align} 
	& \parbox{\dimexpr\linewidth-4em}{%
	(2) With $D$ the Fr\'echet derivative in $x$,
	}
	\nnb
	& \qquad\quad \norm{D e^{\mathfrak{U}' (B, \varphi ; \cdot)} }_{\vec{h}, T_j (B, \varphi)} \leq C(L) e^{ c_w \kappa  w_j (B, \varphi)^2}
	,
	\label{eq:derivatives1}
	\\
	& \qquad\quad \norm{D^2 e^{\mathfrak{U}' (B, \varphi ; \cdot)} }_{\vec{h}, T_j (B, \varphi)} \leq C(L) e^{ c_w \kappa w_j (B, \varphi)^2}
	, 
	\label{eq:derivatives2}
	\\
	& \qquad\quad \norm{D J^{\rr}_j (B, Z, \varphi ; \cdot )}_{\vec{h}, T_j (B, \varphi)} \leq C(L) A^{-1} e^{ c_w \kappa  w_{j} (B, \varphi)^2}
	,
	\label{eq:derivatives3} 
	\\
	& \qquad\quad \norm{D \mathcal{E} K^{\rr}_j (Z, \varphi ; \cdot)  }_{\vec{h}, T_j (Z, \varphi)} \leq C(A, L) A^{- (1+ \eta)|Z|_{j+1}} e^{ c_w \kappa  w_j (Z, \varphi)^2}
	,
	\label{eq:derivatives4} 
	\\
	& \qquad\quad \norm{D  \bar{K}^{\rr}_j (Z, \varphi ; \cdot)  }_{\vec{h}, T_j (Z, \varphi)} \leq C(A, L) A^{- (1+ \eta)|Z|_{j+1}} G_j (Z, \varphi)  
	,
	\label{eq:derivatives5} \\
	& \parbox{\dimexpr\linewidth-4em}{%
	$\mathfrak{U}' \in \{ U_j, \bar{U}_{j+1}, \cE_{j+1} |B| - \kg_{j+1}^{\rr} (B) \}$, and in the case of $\mathfrak{U}' = \cE_{j+1} |B| { - \kg^{\rr}_{j+1} (B) }$, the factor $e^{c_w \kappa  w_j (B, \varphi)^2}$ can be omitted.
The derivatives exist in the space of polymer activities with finite $\norm{\cdot}_{\vec{h}, T_j (B)}$-norm for $e^{\mathfrak{U}'}$, $D e^{\mathfrak{U}'}$, $J^{\rr}_j$, 
finite $\norm{\cdot}_{\vec{h}, T_j (Z)}$-norm for $\cE K^{\rr}_j$
and finite $\norm{\cdot}_{\vec{h}, T_j (Z)}$-norm for $\bar{K}^{\rr}_j$.
	}	
	\stepcounter{equation}
	\tag{\theequation $\regularitytwo$} \label{quote:regularity_two}
\end{align}

\begin{lemma} \label{lemma:Ujbound-summary} 
  Under the assumptions of Theorem~\ref{thm:local_part_of_K_j+1},
  for any $\delta>0$ and $\beta$ sufficiently large,  
  $\epsilon (\delta, L) \equiv \epsilon (\delta, \beta, L)$ (only polynomially small in $L$ and $\beta$) and $\eta >0$
  such that \eqref{eq:Ujbound_1}--\eqref{eq:derivatives5} (with constants possibly depending on $\beta$) hold for $\bar{\kK}_j$ defined by \eqref{eq:bar_kK_j_definition} with $\norm{\omega_j^{\rr, \dagger}}_{\bar{\Omega}_{j, \kt, \dagger}} \leq \epsilon (\delta, \beta,L)$ 
\end{lemma}

The proof of this lemma will be given shortly. 
But before proving it, we emphasize that this lemma is enough to obtain the differentiability of $\kM_{j+1}^{\rr}$, due to the analysis of \cite{dgauss1}.

\begin{proof}[Proof of Theorem~\ref{thm:local_part_of_K_j+1},(ii)] 
We apply \cite[Lemma~7.12]{dgauss1} with its assumptions verified by Lemma~\ref{lemma:Ujbound-summary}--the details of the definitions of the normed spaces differ, but they do not play any crucial role. 
Thus we obtain
\begin{align}
	\norm{D \kM_{j+1} (\bar{\kK}_j (\omega^{\rr, \dagger}_{j} )) }_{\vec{h}, T_{j+1}} 
	\leq C_2 (\rr, \beta,A,L) \norm{\omega^{\rr, \dagger}_j}_{\bar{\Omega}_{j, \kt, \dagger}} 
	.
\end{align}
\end{proof}

\subsection{Proof of Lemma~\ref{lemma:Ujbound-summary}}
\label{sec:proof_of_Lemma_Ujbound-summary}

We are now only left to prove Lemma~\ref{lemma:Ujbound-summary}. The proof of this lemma is divided into two parts, 
the first part discussing \eqref{quote:regularity_one} and the second part discussing \eqref{quote:regularity_two}.

\begin{lemma}
\label{lemma:Ujbound-summary-proof_stage_1}

Suppose we are in the setting of Lemma~\ref{lemma:Ujbound-summary}, 
so $\norm{\omega_j^{\rr, \dagger} }_{\bar{\Omega}_{j, \kt, \dagger}} \leq \epsilon (\delta, \beta, L)$
with $\epsilon (\delta, \beta, L) > 0$ sufficiently small. 
Then \eqref{quote:regularity_one} holds.
\end{lemma}

\begin{proof}
For the case $\mathfrak{U} = U_j$ and $\cE_{j+1} |B|$, since $U_j$ does not have any dependence on $\tau$,  the norm $\norm{\cdot}_{\vec{h}, T_j (B, \varphi)}$ can actually be replaced by $\norm{\cdot}_{h, T_j (B, \varphi)}$. 
However, after the replacement, these bounds are simply implications of \cite[Lemma~7.11]{dgauss1}.
For the case $\mathfrak{U} = \bar{U}_{j+1}$ and $-\kg_{j+1}^{\rr}(B)$, again by \cite[Lemma~7.11]{dgauss1}, for any $\delta >0$, there exists $C(\delta, \beta, L) > 0$ such that
\begin{align}
	\norm{ - \cE_{j+1} |B| + U_{j+1} (B, \varphi) }_{h, T_j (B, \varphip)}  
	\leq C(\delta, \beta, L) \Big(1+ \frac{1}{3} \delta c_w \kappa (L) w_j (B, \varphi)^2 \Big) \norm{\omega_j^{\rr, \dagger} }_{\bar{\Omega}_{j, \kt, \dagger}} .
\end{align}
for $B\in \cB_{j+1}$,
and by \eqref{eq:kg_are_analytic},
\begin{align}
	\norm{\kg_{j+1}^{\rr} (B, K_j^{\rr} ; \cdot) }_{\htau, W}  
	\leq C A^{-1} \norm{\omega_j^{\rr, \dagger} }_{\bar{\Omega}_{j, \kt, \dagger}}
	.
\end{align}
Letting $\bar{U}'_{j+1} (B) = U_{j+1} (B) - \cE_{j+1} |B| + \kg^{\rr}_{j+1} (B)$, we have
\begin{align}
	\norm{ \bar{U}'_{j+1} (B, \varphi) }_{h, T_j (B, \varphip)}  
	\leq C' (\delta, \beta, L) \Big( 1+ \frac{1}{3} \delta c_w \kappa (L) w_j (B, \varphi)^2 \Big) \norm{\omega_j^{\rr, \dagger} }_{\bar{\Omega}_{j, \kt, \dagger}}
\end{align}
and by Lemma~\ref{lemma:F'_is_analytic} 
\begin{align}
	\norm{ \bar{U}'_{j+1} (B, \varphi + \tau u_{j+1}) }_{\vec{h}, T_j (B, \varphi)} 
	\leq 
	\norm{ \bar{U}'_{j+1} (B, \varphi) }_{h'' , T_j (B, \varphi)} ,
\end{align}
where $h'' = h+ \htau \norm{u_{j+1}}_{C_{j}^2}$, and if we set $\htau \leq (C \log L )^{-1} h$ for sufficiently large $C$, then $h'' \leq 2h$. 
But since $\bar{U}_{j+1} (B, \varphi)$ is a polynomial in $\varphi$ with degree 2 plus trigonometric terms, 
we actually have
\begin{align}
	\norm{ \bar{U}'_{j+1} (B, \varphi) }_{h'' , T_j (B, \varphip)} \leq 4 \norm{ \bar{U}'_{j+1} (B, \varphi) }_{h, T_j (B, \varphip)} .
\end{align}
Thus by the definition of $\bar{U}_{j+1}$, 
\begin{align}
	\norm{ \bar{U}_{j+1} (B, \varphi )}_{\vec{h}, T_j (B, \varphi)}  \leq 4C' \Big( 1+ \frac{\delta}{3} c_w \kappa_L w_j (B, \varphi + \rr u_{j+1} )^2 \Big) \leq C'' \Big( 1+ \frac{\delta}{3} c_w \kappa_L w_j (B, \varphi)^2 \Big)
\end{align}
which proves \eqref{eq:Ujbound_1} for $\bar{U}_{j+1}$.
Now \eqref{eq:Ujbound_2} is a pure implication of \eqref{eq:Ujbound_1}. Indeed, by \eqref{eq:exp_taylor_bound}, 
\begin{align}
\Big\|  e^{\kU (B, \varphi)} - \sum_{m=0}^k \frac{1}{m!} (\kU (B, \varphi) )^m  \Big\|_{\vec{h}, T_j (B, \varphi)} \leq \frac{1}{k!} \norm{\kU (B, \varphi)}_{\vec{h}, T_j (B, \varphi)}^{k} e^{\norm{\kU (B, \varphi)}_{\vec{h}, T_j (B, \varphi)}}
\label{eq:exp_taylor_bound_restated}
\end{align}
for $k \in \{ 0,1,2\}$. 
Also again by \eqref{eq:exp_taylor_bound} and \eqref{eq:Ujbound_1}, for $\norm{\omega_j^{\rr, \dagger} }_{\bar{\Omega}_{j, \kt, \dagger}} \leq \frac{1}{C(\delta, \beta, L)}$, 
\begin{align}
& e^{\norm{\kU (B, \varphi)}_{\vec{h}, T_j (B, \varphi)}} \leq e^{1+ \frac{1}{3} \delta c_w \kappa w_j (B, \varphi)^2} 
\\
& \norm{\kU (B, \varphi)}_{\vec{h}, T_j (B, \varphi)}^k \leq C''(\delta, \beta, L) e^{\frac{1}{3} k \delta c_w \kappa w_j (B, \varphi)^2} \norm{\omega_j^{\rr, \dagger} }_{\bar{\Omega}_{j, \kt, \dagger}}^k ,
\end{align}
which concludes the proof after having inserted into \eqref{eq:exp_taylor_bound_restated}.
\end{proof}

\begin{lemma}
Under the setting of Lemma~\ref{lemma:Ujbound-summary},
\eqref{quote:regularity_two} holds.
\end{lemma}
\begin{proof}
\eqref{eq:derivatives1} and \eqref{eq:derivatives2} for the  cases $\kU = U_j$ or $\cE_{j+1}|B|$ are already proved by \cite[Lemma~7.15]{dgauss1}. 
Also for the case $\kU = \bar{U}_{j+1}$, the proof of \cite[Lemma~7.15]{dgauss1} reveals that these bounds are derived purely from \eqref{eq:Ujbound_1} and \eqref{eq:Ujbound_2}. 
Their differentiabilities are also consequences of the differentiability argument of the same reference.

To derive \eqref{eq:derivatives3} and \eqref{eq:derivatives4}, we first look for a bound on $Q^{\rr}_j$.
But Lemma~\ref{lemma:Loc_Et_K_j_bound} bounds both terms of $Q^{\rr}_j$ by
\begin{align}
	\norm{Q^{\rr}_j (B,Z, \varphip ; \cdot)}_{\vec{h}, T_j (B, \varphip)} 
	\leq 
	C(L) \norm{\omega_j^{\rr, \dagger} }_{\bar{\Omega}_{j, \kt, \dagger}} e^{c_w \kappa  w_j (B, \varphip)^2} .
\end{align}
Then by the definition of $J^{\rr}_j$ and $\cE K^{\rr}_j$,  \eqref{eq:cEK^Psi_definition}--\eqref{eq:J_j^Psi_definition},  and since $Q^{\rr}_j$ is a linear function of $(K^0_j,  K_j^{\rr, \dagger})$, we have the bounds \eqref{eq:derivatives3} and \eqref{eq:derivatives4}.
The indicated differentiabilities of $J_j^{\rr}, \cE K^{\rr}_j$ follow because they are linear functions of $\omega_j^{\rr, \dagger}$

Finally,  to see the bound on $\bar{K}^{\rr}_j$, define the function $\bar{\cF}$ on polymer activities by
\begin{align}
\bar{\cF} (U_j, K) (X, \varphi ; \tau) = \sum_{Y\in \cP_j}^{\bar{Y}=X} e^{U_j (X \backslash Y, \varphip + \zeta )} K (Y, \varphi ; \tau)
\end{align}
so that $\bar{\cF} (U_j,K_j^{\rr,\dagger}) = \bar{K}^{\rr}_j$. 
We already have a good control of $\bar{\cF}$,
since \cite[Lemma~7.16]{dgauss1} gives that
\begin{align}
\norm{D \bar{\cF} (U_j,K^0_j) (Z, \varphi ) }_{h, T_j (Z, \varphi)} \leq C(A,L) A^{-(1+\eta) |Z|_{j+1}} G_j (Z, \varphi)
\end{align}
(where $D$ is a derivative in $(U_j, K^0_j)$)
for some $\eta >0$, $C(A,L) >0$ and sufficiently small $\norm{(U_j, K_j^0)}_{\Omega_j}$.
By Remark~\ref{remark:new_norm}, we can in fact translate this into a result written in terms of $\norm{\cdot}_{\vec{h}, T_j (Z, \varphi)}$,  giving
\begin{align}
\norm{D \bar{\cF} (U_j,K^{\rr,\dagger}_j) (Z, \varphi ; \tau ) }_{\vec{h}, T_j (Z, \varphi)} \leq C(A,L) (A/2)^{-(1+\eta) |Z|_{j+1}} G_j (Z, \varphi)
\end{align}
(with $A/2$ because we only have $\norm{K_j^{\rr,\dagger}}_{\vec{h}, T_j, A/2} < \infty$)
for sufficiently small $\norm{\omega_j^{\rr, \dagger}}_{\bar{\Omega}_{j, \kt, \dagger}}$. This is exactly \eqref{eq:derivatives5}.
The differentiability of $\bar{K}^{\rr}_j$ also follows from the argument of \cite[Lemma~7.16]{dgauss1}.

\end{proof}

\ifx\newpageonoff\undefined
{\red command undefined!!}
\else
  \if\newpageonoff1
  \newpage
  \fi
\fi

\section{Renormalisation group computation of the two-point energy}
\label{sec:renormalisation-group-computation}

In this section, we always assume that $\f_1$ and $\f_2$ satisfy \eqref{quote:assumpf} and $u_1$ and $u_2$ are defined according to \eqref{eq:extfield_def} using $\f_1$ and $\f_2$, respectively. 
However, $\f$ will be either $f_1 + T_y \f_2$ or $\f_1$ or $\f_2$ or $T_y \f_2$ depending on the situation. 
Hence for $\kg_j^{\rr}$ defined according to Definition~\ref{def:evolution_of_g_j}, 
the dependence on $\f_1, \f_2$ will be denoted explicitly by putting $\f$-arguments in square brackets whenever necessary, 
e.g., $\kg^{\rr}_j \equiv \kg^{\rr}_j [f_1 + T_y \f_2]$.

If $j$ is sufficiently small compared to the separation between $u_{j,1}$ and $T_y u_{j,2}$, 
then one may claim that $u_{j,1}$ and $T_y u_{j,2}$ do not feel each other, 
since $\varphi |_{\supp (u_{j_1})}$ and $\varphi|_{\supp (T_y u_{j_2})}$ are independent when integrating with respect to $\E_{\Gamma_j}$.
On the other hand, if $j$ becomes sufficiently large so that the range of $\Gamma_{j}$ is large compared to the separation, then the fluctuation integrals pick up extra terms due to this lack of independence.
This can also be implemented in the RG computation by treating the case $j < j_{0y}$ and the case $j\geq j_{0y}$ differently for $\kg_{j+1}^{\rr}$, recalling that $j_{0y}$ is the coalescence scale
\begin{align}
j_{0y} = \min \Big\{ j \geq 0 \; : \; (B^j_0)^{***} \cap Q_y^j  \neq \emptyset \Big\} 
.
\end{align}

\subsection{The infinite volume limit}

To understand the properties of the DG measure on $\Z^2$,
it is necessary to understand the RG flow under taking limit $N\rightarrow \infty$. 
Since the infinite volume limit was constructed for the case $(\tau, \rr) =(0,0)$ by \cite[Proposition~8.3]{dgauss1},
this result can be extended to $(\tau , \rr) \in \D_{\htau} \times \R$ taking the $(\tau, \rr) = (0,0)$ as the reference point. 
However, this extension is not needed in full generality. 
Rather, we only need the information that 
the coupling constants and the free energy, $(\cU_{j}^{\Lambda_N}, \cE_j^{\Lambda_N}, \kg_j^{\rr, \Lambda_N})$,
are independent of $\Lambda_N$.
This can be seen by a short induction argument, 
taking it guaranteed the existence of the renormalisation group flow for any $j < N$.

\begin{proposition}
\label{prop:spatial_locality_of_rg_map}

Let $L$ be sufficiently large and $N \geq j_{0y} + 2$.
Suppose $(\cE_j^{\Lambda_N}, \kg_j^{\rr, \Lambda_N}, \cU_{j}^{\Lambda_N},  \vec{K}_j^{\rr, \Lambda_N})_{j\leq N}$ can be constructed on $\Lambda_N$ via the maps $(\Phi^{\rr, \Lambda_N}_{j+1} )_{0\leq j \leq N-1}$ of \eqref{eq:rg_map_definition} and initial condition \eqref{quote:rginit}.
Then $(\cE_j^{\Lambda_N}, \kg_j^{\rr, \Lambda_N},  \cU_{j}^{\Lambda_N})$ are independent of $N$ when $j < N$.
\end{proposition} 

\begin{proof}
This result is implied by a stronger result on the (spatially) local dependence of the coupling constants on the renormalisation group flow. 
To see this, let $N' > N$ be sufficiently large and
$\iota_{N, N'} : \Lambda_{N} \rightarrow \Lambda_{N'}$ be the embedding with $\iota_{N, N'} (0) = 0$. 
Also construct $(\cU_j^{\Lambda_N},  \cE_j^{\Lambda_N}, \kg_j^{\rr,\Lambda_N}, \vec{K}_j^{\rr,\Lambda_N})$ for $j\leq N-1$ using Definition~\ref{def:evolution_of_g_j} on $\Lambda_N$, 
and also construct the corresponding objects on $\Lambda_{N'}$.
Then assume, as an induction hypothesis, that whenever $j< N-1$,
\begin{align}
\vec{K}^{\rr, \Lambda_N}_j (Y,   \varphi \circ \iota_{N,N'} ) = \vec{K}^{\rr, \Lambda_{N'}}_j ( \iota_{N, N'} (Y)  , \varphi), \qquad \varphi \in S_h (\Lambda_{N'})
\label{eq:K_independent_of_N}
\end{align}
for any $Y \in \cP_j (\Lambda_{N})$ such that $Y \subset B_0^{N-1}$.
This holds for $j=0$ by the assumptions on the initial condition.
As an immediate consequence of \eqref{eq:K_independent_of_N}, we have the same for $(K^{\rr, \Lambda_N}_j)^{\dagger}$ and $(K^{\rr, \Lambda_{N'}}_j)^{\dagger}$.
But since $\kg^{\rr}_{j+1}$, $\cE_{j+1}$ and $\cU_{j+1}$ only depend on $(\vec{K}^{\rr,  \cdot }_j )^{\dagger} (Y)$ for $Y\in \cS_j (\Lambda_{N'})$ such that $Y \cap (P_y^{j})^* \neq \emptyset$ 
(to be precise, $\cE_{j+1}$ and $\cU_{j+1}$ are functions of $K_j^0 (X ; 0)$ for $X \in \cS_j$, but since $K_j^0 (\cdot ; 0)$ is translation invariant,  we only need information of $K_j^0 (Y ; 0)$ for $Y\in \cS_j$ and $Y \cap (P_y^j)^* \neq \emptyset$), 
and any such $Y$ is contained in $[-R,R]^2$ where 
\begin{align}
R = 5 L^j + L^{j_{0y}} \leq 6 L^{N-2} \leq \frac{(L-2)L^{N-2}}{2},
\end{align}
we also have
\begin{align}
\cU^{\Lambda_N}_{j+1} = \cU^{\Lambda_{N'}}_{j+1}, 
\quad 
\cE^{\Lambda_N}_{j+1} = \cE^{\Lambda_{N'}}_{j+1}, 
\quad
\kg^{\rr,\Lambda_N}_{j+1} = \kg^{\rr,\Lambda_{N'}}_{j+1}
\label{eq:U,E,kg_independent_of_N}
.
\end{align}
Also for $X \in \cP_{j+1} (\Lambda_{N'})$ such that $X\subset B_0^{N-1}$, 
\eqref{eq:expression_for_K_j+1_external_field} and \eqref{eq:U,E,kg_independent_of_N} indicate that $\vec{K}^{\rr, \Lambda_N}_{j+1} (X)$ is independent of $N$, completing the induction. 
The conclusion follows from \eqref{eq:U,E,kg_independent_of_N}. 
\end{proof}

By the proposition, \eqref{eq:K_independent_of_N} and \eqref{eq:U,E,kg_independent_of_N}, we can almost pretend as if the dependence of $\Lambda_N$ is not present in the RG flow $(\cE^{\Lambda_N}_j, \kg^{\rr,\Lambda_N}_j, U^{\Lambda_N}_j, K^{\Lambda_N}_j)_{j \leq N-1}$. 
Hence we may drop the superscript $\Lambda_N$ when $j \leq N-1$ (but keep it for $j = \Lambda_N$, 
e.g. , 
use $\kg_N^{\rr,\Lambda_N}$).
Even more, we can now talk about the infinite sequence $(\kg^{\rr}_j)_{j \in \Z_{\geq 1}}$ without making reference to $\Lambda_N$
and we may also state the decay of $\kg^{\rr}_j$ as $j\rightarrow \infty$. 
We start this by justifying the existence of the RG flow of infinite length for sufficiently large $\beta >0$. 

\begin{proposition}
\label{prop:stability_of_dynamics}

Let $\tau \in \D_{\htau}$. 
There exists $\beta_0 (\rr), L_0 (\rr) >0$ such that the renormalisation group flow, \eqref{eq:rg_map_definition}, 
is defined for any $N >0$ and all $j \in \{0,1, \cdots, N-1\}$ at $\beta \geq \beta_0 (\rr)$ and $L \geq L_0(\rr)$
when the initial condition is given by \eqref{quote:rginit}.
\end{proposition}

\begin{proof}
By Theorem~\ref{thm:tuning_s-v2},  existence of the flow $(U_j,E_j,  \vec{K}^{0}_j (\cdot ; 0))_{j \leq N-1}$ is guaranteed when $\beta \geq \beta_0$ is taken sufficiently large and the initial condition is given by \eqref{quote:rginit}. They also satisfy
\begin{align}
	\norm{U_j}_{\Omega_j^U},  \; 
	\norm{K^0_j (\cdot ;  0)}_{\Omega_j^K}
 	\leq C e^{-c_f \beta} L^{-\alpha j}
	\label{eq:U_j_K_j_flow_bound}
	.
\end{align}
for some $C>0$.
So we only have to prove bounds on $K^{\rr}_j (\cdot ; \tau)$ and use them as an input to Theorem~\ref{thm:rg_map_definition_estimates}.
For this,  assume as an induction hypothesis, that
\begin{align}
\norm{K^{\rr}_j }_{\vec{h}, T_j} \leq C' e^{-c_f \beta} L^{-\alpha j},  \qquad j < N-1
\label{eq:K_j_inductive_bound}
\end{align}
for some $C ' \equiv C' (\rr) > 0$. 
When $\beta \geq \beta_0 (\rr)$ is taken sufficiently large, then $K^{\rr}_{j}$ sits in the domain where Theorem~\ref{thm:local_part_of_K_j+1} holds.
Then by \eqref{eq:rg_map_main_estimate1} and \eqref{eq:rg_map_main_estimate2}, 
\begin{align}
	\norm{K_{j+1}^{\rr} (\cdot ; \cdot) - K^{\rr}_{j+1} (\cdot ; 0) }_{\vec{h}, T_{j+1}} 
	& = 
	\big\| \cK_{j+1}^{\rr} (\omega_j^{\rr,\dagger}) - \cK^{\rr}_{j+1} \big( U_j, K_j^0 (\cdot ; 0) ,  (K_j^{\rr})^{\dagger} (\cdot ; 0)  \big) \big\|_{\vec{h}, T_{j+1}} \nnb
	& 
	\leq C_1  \alphaLoco \norm{ (K^{\rr}_j)^{\dagger} (\cdot ; \cdot) - (K^{\rr}_j)^{\dagger} (\cdot ; 0)}_{\vec{h}, T_j, A/2} 
\end{align}
and
\begin{align}
	& \norm{K_{j+1}^{\rr} (\cdot ; 0) - K_{j+1}^0 (\cdot ; 0) }_{\vec{h}, T_{j+1}} \nnb
	& \qquad = 
	\norm{\cK_{j+1}^{\rr} ( U_j, K_j^0 (\cdot ; 0) ,  (K_j^{\rr})^{\dagger} (\cdot ; 0) ) - \cK_{j+1}^{\rr} ( U_j, K_j^0 (\cdot ; 0) ,  (K_j^{0})^{\dagger} (\cdot ; 0) ) }_{h, T_{j+1}} 
	\nnb
	& \qquad  \leq 
	C_1  \alphaLoco \norm{ (K^{\rr}_j)^{\dagger} (\cdot ; 0) - (K^{0}_j)^{\dagger} (\cdot ; 0)}_{h, T_j, A/2} 
\end{align}
whenever $\norm{\omega^{\rr, \dagger}_j}_{\bar{\Omega}_{j, \kt, \dagger}} = \norm{(U_j, \vec{K}^{\rr, \dagger}_j)}_{\bar{\Omega}_{j, \kt, \dagger}}$ is sufficiently small compared to $\alphaLoco$--this can be achieved by taking $\beta$ sufficiently large in \eqref{eq:U_j_K_j_flow_bound}, and \eqref{eq:K_j_inductive_bound} and using Lemma~\ref{lemma:reblocking_estimate}.
Also by Lemma~\ref{lemma:reblocking_estimate},
\begin{align}
	& 
	\norm{ (K^{\rr}_j)^{\dagger} (\cdot ; \cdot) }_{\vec{h}, T_j, A/2}, \;\;  
	\norm{ (K^{0}_j)^{\dagger} (\cdot ; 0)}_{h, T_j,A/2} 
	\leq  \norm{\omega^{\rr}_j}_{\bar{\Omega}_{j, \kt}} 
	.
\end{align}
Hence combining these bounds with \eqref{eq:U_j_K_j_flow_bound},
\begin{align}
	\norm{K^{\rr}_{j+1} (\omega_j) (\cdot ; \cdot) }_{\vec{h}, T_{j+1}}
	\leq \big( C + 2 C_1 (C + C') L^{\alpha}  \alphaLoco \big)  e^{-c_f \beta} L^{-\alpha (j+1)} .
\end{align}
But if we use the fact that $L^{\alpha}$ is chosen in such a way that $L^2 \alphaLoc \leq O( L^{-\alpha})$ (see Theorem~\ref{thm:tuning_s-v2}) and $\alphaLoco \leq L^2 (\log L)^{-1} \alphaLoc$,  
we may take $L$ and $C' / C$ sufficiently large to obtain
\begin{align}
C + 2 C_1 (C + C') L^{\alpha}  \alphaLoco \leq C' 
,
\end{align}
completing the induction.
\end{proof}

In fact, the proof also gives a bound on the RG coordinates. 

\begin{corollary}
\label{cor:K_j_e_j_bounds}
Under the assumptions of Proposition~\ref{prop:stability_of_dynamics},  for each $B\in \cB_j$ and some $C >0$,
\begin{align}
\norm{K^{\rr}_j}_{\vec{h}, T_j},  \;\; \norm{\kg^{\rr}_j (B; \tau)}_{\htau, W} \leq C e^{-c_f \beta} L^{-\alpha j}
.
\end{align}
\end{corollary}
\begin{proof}
The bound on $K^{\rr}_j$ is from \eqref{eq:K_j_inductive_bound}. 
Then the bound on $\kg^{\rr}_j$ also follows from \eqref{eq:kg_are_analytic}.
\end{proof}

\subsection{Computation of the free energy}

Based on the RG analysis, we can now show that
\begin{align}
	F_{N,m^2} [\f] (\rr, \tau)
	= \frac{ \Et \big[ Z_0 (\phi^{(m^2)} + \gamma \tau \f + \rr {\textstyle \sum_{j=0}^N u_j } ) \big] }
	{\E \big[ Z_0 (\phi^{(m^2)}) \big]}
\end{align}
exhibits a well-defined limit as $m^2 \downarrow 0$ and $N \rightarrow \infty$ described in terms of the renormalisation group coordinates, 
where $\phi^{(m^2)} \sim \cN(0,  \tilde{C} (s_0^c (\beta)) + t_N (m^2) Q_N)$ (see Proposition~\ref{prop:final_alternative_form_of_mgf}).

\begin{proposition} \label{prop:inf_vol_energy}
Let $\f = \f_1 + T_y \f_2$ be as in \eqref{quote:assumpf}.
Under the assumptions of Proposition~\ref{prop:stability_of_dynamics}, 
if $\cz \in \C$, $\tau \in \D_{\htau}$ and $\rr \in \R$ are such that $\cz = \tau + \rr$, then
\begin{align}
	F_{N,m^2} [\f] (\rr,\tau)
\xrightarrow{N\rightarrow \infty} 
\exp \Big( \sum_{j =1}^{\infty} \sum_{B \in \cB_j ((P_y^j)^*)} \kg^{\rr}_j (B ; \tau) [\f] \Big) & \quad \text{uniformly in} \;\; m^2 \in (0,1]
\label{eq:moment_convergence_unif_in_m^2}
\end{align}
where $\tilde{\f} = \f+ s_0^c (\beta) \gamma \Delta \f$.
\end{proposition}
\begin{proof}

Recall 
$\tilde{C}(s_0^c) = \sum_{j=1}^{N-1}\Gamma_j + \Gamma_N^{\Lambda_N}$ and
\begin{align}
	Z^{\rr}_N (\varphip  ; \tau) = \Et \big[ Z_0 (\varphip + \zeta \gamma \tau \f + \rr {\textstyle \sum_{j=0}^N u_j }  ) \big], \qquad \zeta \sim \cN \big( 0, \textstyle \tilde{C} (s_0^c) \big)
	.
	\label{eq:Z_N_recall}
\end{align}
With initial condition $(\Phi_j^{\rr})_{j \leq N}$ is given by \eqref{quote:rginit}, 
Proposition~\ref{prop:stability_of_dynamics} guarantees that we may iterate the RG map $\Phi_j^{\rr}$ any number of times,
and hence successive application of Theorem~\ref{thm:rg_map_definition} gives
\begin{align}
Z^{\rr}_N (\varphip ; \tau) &= e^{- E_N^{\Lambda_N} |\Lambda_N|} \exp \Big( \sum_{j=1}^N \sum_{B\in \cB_j ((P_y^j)^*)} \kg_{j}^{\rr,\Lambda_N} [\f] (B ; \tau)  \Big) \nnb
& \qquad\qquad   \times
 \big( e^{U_N^{\Lambda_N} (\Lambda_N, \varphip + (\tau + \rr ) u_N) } + K^{\rr,\Lambda_N}_N (\Lambda_N, \varphip  ; \tau) \big)
	\label{eq:Z_N_expression}
\end{align}
with $u_N = \Gamma_{N}^{\Lambda_N} \f$, and satisfy estimates
\begin{align}
	\norm{\kg_N^{\rr,\Lambda_N} [\f]}_{\htau, W}, \;\;
	\norm{U_N^{\Lambda_N}}_{\Omega_N^U}, \;\;
	\norm{K_N^{\rr, \Lambda_N}}_{\vec{h}, T_N}
	\leq C e^{- c_f \beta} L^{-\alpha N}
	.
	\label{eq:prop_inf_vol_energy_inter1}
\end{align}
But also by Proposition~\ref{prop:reblocking_Z_with_Psi} (note that the sum is just over $X \in \cP_N = \{ \emptyset , \Lambda_N \}$)
\begin{align}
e^{U_N^{\Lambda_N} (\Lambda_N, \varphip + (\tau + \rr) u_N) } + K^{\rr, \Lambda_N}_N (\Lambda_N, \varphip  ; \tau) = e^{U_N^{\Lambda_N} (\Lambda_N, \varphip) } + (K^{\rr,\Lambda_N}_N)^{\dagger} (\Lambda_N, \varphip ; \tau)
\label{eq:reblocking_final_step}
\end{align}
and by Lemma~\ref{lemma:reblocking_estimate}, 
\begin{align}
\big\| (K^{\rr,\Lambda_N}_N)^{\dagger} (\Lambda_N, \varphip ; \cdot ) \big\|_{\vec{h}, T_N (\Lambda_N, \varphip)} \leq C' e^{-\frac{1}{4} \gamma \beta} L^{-\alpha N} G_N (\Lambda_N, \varphip) .
\end{align}
Finally,  we commence the integral
\begin{align}
\E \big[ Z_0 (\phi^{(m^2)} + \rr u_N) \big] 
= \E^{\varphip} \big[  Z^{\rr}_N (\varphip + \rr u_N ; \tau) \big], \qquad  \varphip \sim \cN (0, t_N Q_N) .
\end{align}
Observe that,  if we take $Y\sim \cN (0, t_N L^{-N})$, then $Y \one$ has the same distribution as $\varphip$ (where $\one$ is the constant field taking value 1, also see \eqref{quote:Gamma_four}) so $G_N (\Lambda_N, \varphip) \equiv \text{(constant)}$ almost surely.  Also $|\nabla \varphip|^2 = 0$ almost surely, so $|U_N^{\Lambda_N} (\Lambda_N, \varphip)| \leq \norm{U_N^{\Lambda_N}}_{\Omega_N^{U}}$.  Therefore
\begin{align}
	\E^{\varphip} \big[ Z^{\rr}_N (\varphip  ; \tau) \big]  
	& = e^{-E_N^{\Lambda_N} |\Lambda_N|} \exp \Big( \sum_{j=1}^{N} \sum_{B\in \cB_j ((P_y^j)^*)} \kg^{\rr, \Lambda_N}_{j} [\f] (B ; \tau)  \Big) 
	\big( 1+  O(L^{-\alpha N} ) \big)
	\label{eq:prop_inf_vol_energy2}
	.
\end{align}
Since $g_j^{\rr, \Lambda_N}$ is independent of $\Lambda_N$ for $N > \max\{ j,j_{0y} +2 \}$,
the convergence \eqref{eq:moment_convergence_unif_in_m^2} holds by \eqref{eq:prop_inf_vol_energy_inter1} and Corollary~\ref{cor:K_j_e_j_bounds}.
The convergence is uniform in $m^2$ because the convergence rate only depends on $O(L^{-\alpha N} )$ and bounds on $\kg^{\rr, \Lambda_N}_j [\f] (B; \tau)$'s, which are independent of $m^2 >0$.
\end{proof}

In this proposition, the finite volume function was computed only to justify that the infinite volume limit exists. However, finite volume case is actually the case of our interest, so we record one by-product in the next corollary.
In what follows,  we use the notation
\begin{align}
\kg^{\rr}_{j} [\f] (\Z^2 ; \tau) := \sum_{B \in \cB_j ((P_y^j)^*)} \kg^{\rr}_j [\f]  (B ; \tau) 
\end{align}
although $\kg^{\rr}_j (B ; \tau) [\f]$ is actually defined on $\Lambda_N$.

\begin{corollary} \label{cor:finite_volume_energy}
Fix $R >0$ and let $\cz = \rr + \tau$ with $\tau \in \D_{\htau}$ and $\rr \in [-R,R]$. 
Then for sufficiently large $L$ (depending on $R$) and $N$,
and under the same assumptions as in Proposition~\ref{prop:inf_vol_energy} but $\f$ satisfying \eqref{quote:assumpfa},
\begin{align}
	\lim_{m^2 \downarrow 0} F_{N,m^2} [\f]  (\rr,\tau) 
= e^{ \sum_{j =1}^{\infty}  \kg^{\rr}_j [\f] (\Z^2 ; \tau)  } 
\big(1 + \tilde{\psi}^{\rr}_N [\f] (\tau,y)  \big)
\label{eq:finite_volume_energy}
\end{align}
where $\tilde{\psi}^{\rr}_N [\f] (\tau,y)$ is an analytic function of $\tau \in \D_{\htau}$  and $\norm{ \tilde{\psi}^{\rr}_N [\f] (\cdot,y) }_{\htau, W} = O( L^{-\alpha N})$ with the rate uniform on $\rr \in [-R, R]$. 
\end{corollary}
\begin{proof}
By Proposition~\ref{prop:final_alternative_form_of_mgf}, $F_{N,m^2}$ admits a limit as $m^2 \downarrow 0$ when $\f$ satisfies \eqref{quote:assumpfa} and the limit is an analytic function.  
It follows from \eqref{eq:Z_N_recall},  uniformity of \eqref{eq:prop_inf_vol_energy2} in $m^2$ and definition of $F_{N,m^2}$ that
\begin{align}
	 1 + \tilde{\psi}_N^{\rr} [\f ] (\tau,y)  := e^{- \sum_{j =1}^{\infty}  \kg^{\rr}_j (\Z^2 ; \tau) }  \lim_{m^2 \downarrow 0} F_{N,m^2} [\f]
\end{align}
satisfies
\begin{align}
	1 + \tilde{\psi}_N^{\rr} [\f] (\tau,y) 
	= e^{-\sum_{j > N}  \kg^{\rr}_j (\Z^2 ; \tau)} \lim_{m^2 \downarrow 0} \E \Big[  e^{U_N^{\Lambda_N} (\Lambda_N,  Y \one) }  + (K_{N}^{\rr, \Lambda_N} )^{\dagger} (\Lambda_N,  Y \one ; \tau) \Big] = 1 + O(L^{-\alpha N})
\end{align}
with $Y \sim \cN(0, t_N (m^2) L^{-N})$. 
It also satisfies $\norm{\tilde{\psi}_N^{\rr} [\f] (\cdot, y)}_{\htau, W} = O(L^{-\alpha N})$ as
\begin{align}
\norm{ \E[ \partial^n_{\tau} (K_{N}^{\rr, \Lambda_N} )^{\dagger} (\Lambda_N,  Y \one ; \tau) ] }_{h, T_j (\Lambda_N, 0)} \leq C \norm{ \partial^n_{\tau} (K_{N}^{\rr, \Lambda_N} )^{\dagger} (\Lambda_N) }_{h, T_j (\Lambda_N)} 
\end{align}
for any $n\geq 0$.
Thus we have \eqref{eq:finite_volume_energy}.

The final remark follows from the observation that the estimates of the previous sections would still stay uniformly true even if we were looking at the regime $|\rr| \leq R$. 
\end{proof}

\subsection{One and two-point energies}

Due to Corollary~\ref{cor:finite_volume_energy}, 
the proof of Proposition~\ref{prop:overview_of_the_proof} is complete once we show that the sum $\sum_{j} g_j^{\rr} [T_y \f_2 ] (\Z^2 ; \tau)$ does not depend on $y$.
In this section, we will see that this is the case, using translation invariance of $F_{N,m^2}[\f]$.

\subsubsection{One-point energy.}

For any $j\geq 1$, we have
\begin{align}
	\sum_{B\in \cB_j ( (P_y^j)^* )} \kg^{\rr}_j  [T_y \f_2] (B ; \tau)  
		= \sum_{B\in \cB_j ( (Q_y^j)^* )} \kg^{\rr}_j [T_y \f_2] (B ; \tau)  
	.
\end{align}
Hence, by \eqref{eq:moment_convergence_unif_in_m^2}, this implies
\begin{align}
	\exp \Big(  \sum_{j =1}^{\infty} 
	\sum_{B\in \cB_j ( (Q_y^j)^* )} \kg^{\rr}_j [T_y \f_2]  (B ; \tau)  \Big) 
	= \lim_{N\rightarrow \infty}   F_{N, m^2} [T_y \f_2]
	\quad \text{uniformly in} \; m^2 > 0 .
	\label{eq:prop_one_pt_energy_1}
\end{align}
Also, if $\f = \f_2$,  the same principles give
\begin{align}
	\exp \Big( \sum_{j\geq 1} \sum_{B\in \cB_j( (B_0^j)^* )} \kg^{\rr}_{j} [ \f_2] (B ; \tau) \Big) 
	= \lim_{N\rightarrow \infty} F_{N, m^2} [\f_2]
	\quad \text{uniformly in} \; m^2 > 0
	.
	\label{eq:prop_one_pt_energy_2}
\end{align}
But since $F_{N,m^2} [T_y \f_2]$ is independent of the choice of $y$, 
we see that the expression on the left side of \eqref{eq:prop_one_pt_energy_1},
\begin{align}
	\exp \Big( \sum_{j =1}^{\infty} \sum_{B\in \cB_j ((Q_y^j)^* )} \kg^{\rr}_j  [T_y \f_2]  (B ; \tau) \Big) 
		= \exp \Big( \sum_{j\geq 1}  \kg^{\rr}_j [T_y \f_2] (\Z^2 ; \tau)  \Big)
\end{align}
should also not depend on $y$,
i.e.,
\begin{align}
\sum_{j \geq 1}  \kg^{\rr}_j [T_y \f_2] (\Z^2 ; \tau)  = \sum_{j\geq 1}  \sum_{B\in \cB_j( (B_0^j)^* )} \kg^{\rr}_{j} [ \f_2]  (B ; \tau) . 
\end{align}
In summary, we have the following. 

\begin{proposition} \label{prop:one_point_energy}
Let $\tau \in \D_{\htau}$ and $\beta > 0$ be sufficiently large and
$\f_1$ be as in \eqref{quote:assumpf}. 
Define the infinite volume one-point energy of $\f_1$ by
\begin{align}
	\tilde{g}^{\rr}_{\infty} [\f_1 ] (\tau)  = \sum_{j\geq 1} \kg^{\rr}_{j} [ \f_1] (\Z^2 ; \tau)
	\label{eq:e_infty_definition}
	.
\end{align}
Then $\tilde{g}^{\rr}_{\infty} [\f_1 ] (\tau)  = \sum_{j \geq 1}  \kg^{\rr}_j [T_y \f_1] (\Z^2 ; \tau) $ for any $y$.
This series converges absolutely with rate
\begin{align}
	\Big\| \tilde{g}^{\rr}_{\infty} [\f_1] (\tau) - \sum_{j = 1}^{m} \kg^{\rr}_{j} [T_y \f_1] (\Z^2 ; \tau) \Big\|_{\htau, W} 
	\leq C(M, \rho,  \rr)  L^{-\alpha m}
\end{align}
with rate uniform on $\rr \in [-R,R]$ and $y\in \Z^2$, 
and $g^{\rr}_{\infty} [\f_1] (\tau)  \in W^+ (\D_{\htau})$.
\end{proposition}

\begin{proof}
The first part follows from the discussion above. 
Also the second part is a consequence of the estimate of Corollary~\ref{cor:K_j_e_j_bounds} and recalling that any uniform limit of analytic functions is also an analytic function. 
\end{proof}

\subsubsection{Two-point energy.}

The two-point energy is defined to be the free energy in scales after $j_{0y}$. But for doing so, we will have to make sure that the free energy before the scale $j_{0y}$ can be expressed as sum of the one-point energies.

\begin{lemma}
\label{lemma:before_coalsc_scale}
Let $j \leq j_{0y} - 1$ and $\f = \f_1 + T_y \f_2$ as in \eqref{quote:assumpf}.
Then
\begin{alignat}{2}
	\kg_j^{\rr}  [\f_1 + T_y \f_2] (B) 
		&= \kg_j^{\rr} [\f_1](B)	& \quad & \text { for } B \in \cB_{j-1} ( (B_0^{j-1})^{*} ) \text{ and} \\
	\kg_j^{\rr} [\f_1 + T_y \f_2](B) 
		& = \kg_j^{\rr} [T_y \f_2](B) & \quad &  \text{ for } B\in \cB_{j-1} ( (Q_y^{j-1})^{*} )
	.
\end{alignat}
\end{lemma}

\begin{proof}

These follow from the `local dependence' of the RG flow as in the proof of Proposition~\ref{prop:spatial_locality_of_rg_map}. 
For the first statement,
we assume, as an induction hypothesis, that
\begin{align}
	K_{j'}^{\rr} [\f_1 + T_y \f_2] (X)  = K_{j'}^{\rr}  [\f_1]  (X)
	\quad 
	\text{for any} \; X\in \cP_{j'},  \;  X \cap (Q_y^{j_{0y} - 1} )^{*} = \emptyset ,  
	\; j' < j_{0y}-1. 
\end{align}
We have $X^* \cap \operatorname{supp} (T_y u_{j,2}) = \emptyset$ for any such $X$, 
so this would imply $K_{j'}^{\rr, \dagger} [\f_1 + T_y \f_2]  (X) = K_{j'}^{\rr, \dagger} [\f_1] (X)$.
Also, by definition of $j_{0y}$, saying $j' < j_{0y} -1$ would mean
$(B_0^{j' + 1})^{**} \cap (Q_y^{j_{0y}-1})^* = \emptyset$,
so $\kg_{j' +1} [\f_1 + T_y \f_2] (B) = \kg_{j' +1} [\f_1 ](B)$  for any $B\in \cB_{j'} ( (B_0^j)^{*})$ by Definition~\ref{def:evolution_of_g_j}. 
If we also had $j' < j_{0y} -1$,  $Y \in \cP_{j' +1}$,  and $Y \cap (Q_y^{j'+1})^* = \emptyset$, since $K_{j' +1}^{\rr} (Y)$ only depends on $\f_1 + T_y \f_2$ via $u_{j+1} |_{Y^*}$, $(\kg_{j' + 1}^{\rr} (B))_{B\in (B_0^j)^*}$ and $(K_{j'}^{\rr, \dagger} (X') : X'\in \cP_{j'} ( X^* ), X\in \cP_{j'} (Y) )$,  the induction proceeds.  The conclusion was also obtained during the induction process.

The second statement follows from an identical argument. 
\end{proof}

As we have claimed, this lemma implies that the free energy before scale $j_{0y}$ is just the sum of two one-point energies.

\begin{corollary} \label{cor:before_coalsc_scale}
If $j < j_{0y}$, then
\begin{align}
	\kg_j^{\rr}  [\f_1 +  T_y \f_2] (\Z^2 ; \tau) = \kg_j^{\rr}  [\f_1] (\Z^2 ; \tau) + \kg_j^{\rr}  [T_{y} \f_2] (\Z^2 ; \tau)
\end{align}
\end{corollary}

Thus we can see that the two-point energy, defined in the following lemma, has diminishing contribution in the limit $\norm{y}_2 \rightarrow \infty$.

\begin{lemma} 
\label{lemma:two_point_energy}
Let $\tau \in \D_{\htau}$ and $\beta >0$ be sufficiently large. 
Define
\begin{align}
\tilde{g}^{(2),\rr}_{\infty}  [\f_1,  T_y \f_2] (\tau) = \sum_{j=1}^{\infty} \kg^{\rr}_j [\f_1 + T_y \f_2] (\Z^2 ; \tau)  -  \tilde{g}^{\rr}_{\infty} [\f_1] (\tau)  - \tilde{g}^{\rr}_{\infty} [\f_2] (\tau) 
.
\end{align}
Then $\tilde{g}^{\rr}_{\infty}  [\f_1,  T_y \f_2] \in W^+ (\D_{\htau})$ and satisfies the bound
\begin{align}
	\big\| \tilde{g}^{(2),\rr}_{\infty} [\f_1,  T_y \f_2]  \big\|_{\htau, W} 
	= O\big( \norm{y}_2^{-\alpha} \big)
	.
\end{align}
\end{lemma}

\begin{proof}
By Corollary~\ref{cor:before_coalsc_scale} and Proposition~\ref{prop:one_point_energy} we have
\begin{align}
	\tilde{g}^{\rr}_{\infty} [\f_1,  T_y \f_2]  (\tau) 
=  \sum_{j=j_{0y} }^{\infty} \Big( \kg^{\rr}_j [\f_1 + T_y \f_2] (\Z^2 ; \tau)  - \kg^{\rr}_j [\f_1 ] (\Z^2 ; \tau)  - \kg^{\rr}_j [T_y \f_2] (\Z^2 ; \tau)  \Big) .
\end{align}
Hence by Corollary~\ref{cor:K_j_e_j_bounds},  the norm on $\tilde{g}^{\rr}_{\infty} (\tau) [\f_1,  T_y \f_2]$ is bounded by $O(L^{-\alpha j_{0y}})$. But since $L^{-j_{0y}} = O(\norm{y}_2^{-1})$, we have the desired conclusion.
\end{proof}

Thus we are equipped with all components required to prove Proposition~\ref{prop:overview_of_the_proof}.

\begin{proof}[Proof of Proposition~\ref{prop:overview_of_the_proof}]

Our aim is to rewrite the limit $\lim_{m^2\downarrow 0}F_{N,m^2} [\f]$.
By Corollary~\ref{cor:finite_volume_energy}, Proposition~\ref{prop:one_point_energy} and Lemma~\ref{lemma:two_point_energy}, 
\begin{align}
	\lim_{m^2 \downarrow 0} F_{N,m^2} [\f] = e^{\tilde{g}_{\infty}^{\rr} [\f_1] + \tilde{g}_{\infty}^{\rr} [\f_2] +  \sum_{j\geq j_{0y}} \tilde{\kg}_j^{(2),\rr} [\f_1, T_y \f_2] ) } ( 1 + \tilde{\psi}^{\rr}_N [\f] (\cdot,  y) )
\end{align}
where
\begin{align}
	& \tilde{\kg}_j^{(2),\rr} [\f_1 , T_y \f_2] = \kg^{\rr}_j [\f_1 + T_y \f_2] (\Z^2)  -  {\kg}_j^{\rr}[\f_1] (\Z^2)  - \kg_j^{\rr} [\f_2] (\Z^2)
\end{align}
and the required properties also follow from the same references.
\end{proof}

\ifx\newpageonoff\undefined
{\red command undefined!!}
\else
  \if\newpageonoff1
  \newpage
  \fi
\fi

\appendix

\section{Action of $\Et$ on polymer activities}
\label{sec:norm_inequalities}

\subsection{Subdecomposition of the the field}

We import some lemmas from \cite{dgauss1} and \cite{dgauss2},
that use the idea of subdecomposing the field and the regulator. 
When $\zeta \sim \cN (0, \Gamma_{j+1})$, it is observed in \cite[Section~4.3]{dgauss1} that,
for $N' \in \N$ and $I_{N'} = \{ s = 0, \frac{1}{N'}, \cdots, 1-\frac{1}{N'}  \}$ such that $\ell := L^{1/N'} \in \N$,
there exist covariances $(\Gamma_{j+s, j+s+(N')^{-1}})_{s \in I_{N'}}$ such that
\begin{align}
\zeta = \sum_{s\in I_{N'}} \xi_{j+s, j+s'} ,  \qquad \Gamma_{j+1} = \sum_{s \in I_{N'}} \Gamma_{j+s, j+s'}
\label{eq:covariance_subdecomposition}
\end{align} 
(where $s' = s + (N')^{-1}$) with each $\xi_{j+s, j+s'} \sim \cN(0, \Gamma_{j+s, j+s'})$ independent
and $(\Gamma_{j+s, j+s+(N')^{-1}})_{s \in I_{N'}}$'s satisfy properties similar to \eqref{quote:Gamma_one}--\eqref{quote:Gamma_four}.
All subsequent properties we need are summarised in the following lemmas, 
where we use
\begin{align}
g_{j+s} (X, \varphi) = \exp\Big( \kappa c_4 \sum_{a=0}^2 W_{j+s} (X, \nabla_j^a \varphi ) \Big),  \quad X \in \cP_{j+s}
\end{align}
where $W_{j+s}$ (resp. $\cP_{j+s}$) is a subscale analogue of \eqref{eq:W^r_j_definition} (resp. $\cP_j$)
and $c_4 >0$ is chosen in the next lemma. 
Given $Y\in \cP_j$, denote $Y_{s}$ the smallest polymer in $\cP_{j+s}$ that contains $Y$,  
and $\norm{\cdot}_{C^2_{j+s} (X)}$, $\norm{\cdot}_{L^2_{j+s} (X)}$ and $\norm{\cdot}_{L^2_{j+s} (\partial X)}$ are as in \cite[Section~5.3]{dgauss1}.

\begin{lemma} \label{lemma:extfield_bound_fine}
For $\alpha = 1,2$, let $\f_{\alpha}$ be as in \eqref{quote:assumpf} and let
\begin{align}
	u_{j+s,j+\bar{s}, \alpha} = \Gamma_{j+s, j+s'} \tilde{\f}_{\alpha}, \quad \bar{s} > s, \; s \in I_{N'}, \; \bar{s} \in (N')^{-1} + I_{N'}.
\end{align}
Then $u_{j+s, j+\bar{s} , \alpha}$ is supported on $B_{j+\bar{s}}^{0}$, the unique $j+\bar{s}$-block containing 0 and for each $n\geq 0$ and
\begin{align}
	& \norm{\nabla^n u_{j+s,j+\bar{s}, \alpha}}_{L^{\infty}} \leq 
	\begin{array}{ll}
	\begin{cases}
	C_0 M \rho^2  \log L & \text{if} \;\; n =0 \\
	C_n M \rho^2 L^{-n (j+s) }  & \text{if} \;\; n \geq 1
	\end{cases}
	\end{array}
\end{align}
Also, $u_{j+1, \alpha}$ defined by Definition~\ref{def:extfield_def} admits decomposition
\begin{align}
	u_{j+1, \alpha} = \sum_{s\in I_{N'}} u_{j+s, j+s + (N')^{-1}, \alpha} .
\end{align}
\end{lemma}
\begin{proof}
The proof is identical to Lemma~\ref{lemma:extfield_bound}.
\end{proof}

\begin{lemma}[\!\!{\cite[Lemma~3.1]{dgauss2}}] \label{lemma:G_change_of_scale_external_field}

Assume $0\leq j < N$, $L=\ell^{N'}$.
For $X\in \cP_{j+s}$ and $\varphi$, $\xi_o$, $\xi_B \in \C^{\Lambda_N}$ for each $B\in \cB_{j+s} (X)$, define
\begin{align}
& \log G_{j+s} (X, \varphi, \xi_o, (\xi)_{B\in \cB_{j+s} (X)})  \\
& = \kappa \norm{\nabla_{j+s} (\varphi + \xi_o)}_{L_{j+s}^2 (X)}^2 
+ \kappa \,  c_2 \norm{\nabla_{j+s} (\varphi + \xi_o)}_{L_{j+s}^2 (\partial X)}^2 
+ \kappa  \sum_{B\in \cB_{j+s} (X)} \norm{\nabla_{j+s}^2 (\varphi + \xi_B)  }^2_{L^{\infty} (B^*)} . \nonumber
\end{align}
For any choice of $c_2$ small enough compared to $1$, 
there exist $c_4 \equiv c_4(c_2)$ and an integer $\ell_0=\ell_0(c_1,c_2)$, 
such that for all $\ell \geq \ell_0$, $N' \geq 1$, $s \in \{0,\frac{1}{N'},\dots,1-\frac1{N'} \}$, $s' = s+ (N')^{-1}$ and $\kappa > 0$,
for $X\in \cP_{j+s}^c$,
\begin{equation}\label{eq:G_change_of_scale_variant_form}
G_{j+s} (X, \varphi , \xi_o, (\xi_B)_{B\in \cB_{j+s} (X)}) 
\leq 
\max_{\mathfrak{a} \in \{o\} \cup \cB_{j +s} (X)} g_{j+s}( X_{s'}, \xi_{\mathfrak{a}} ) G_{j+s'} ( X_{s'}, \varphi)
\end{equation}
where $X_{s'}$ is the smallest $j+s'$-polymer containing $X$.
\end{lemma}

By the lemmas, one may decompose
\begin{align}
	\frac{ G_j (X, \varphip + \zeta + \rr u_{j+1} )  }{G_{j+1} (\bar{X}, \varphip )}
	& \leq
	\prod_{s \in I_{N'} } g_{j+s} \big( X_{s'} ,   \xi_{j+s, j + s'}  + \rr u_{j+s, j+s'} \big) \nnb
	& \leq e^{C(M, \rho, \rr, \ell) N' / \log L  }  \prod_{s \in I_{N'} } g_{j+s} \big( X_{s'} ,   \xi_{j+s, j + s'}  \big)^2
	,
\end{align}
where $u_{j+s,j+s'} = u_{j+s,j+s',1} + T_y u_{j+s,j+s', 2}$ and the second inequality follows from Lemma~\ref{lemma:extfield_bound_fine}.  But since $N' = \frac{\log L}{\log \ell}$, 
\begin{align}
	\frac{ G_j (X, \varphip + \zeta + \rr u_{j+1} )  }{G_{j+1} (\bar{X}, \varphip )}
	& \leq C(M, \rho,\rr, \ell) \prod_{s \in I_{N'} } g_{j+s} \big( X_{s'} ,   \xi_{j+s, j + s'}  \big)^2
	.
	\label{eq:G^r_j_G_j+1_ratio}
\end{align}
The expectation of each $(g_{j+s})^2$ is controlled by the next lemma.

\begin{lemma}[\!\!{\cite[Lemma~5.12]{dgauss1}}] \label{lemma:g_j+s_bound_by_quadratic}
The following hold for $g_{j+s}$.

\begin{itemize}
\item There exists $C>0$ such that
for any $X\in \cP_{j+s}$ and $\zeta \in \R^{\Lambda_N}$, 
\begin{equation}  \label{eq:g_j+s_bound}
  g_{j+s} (X, \zeta) \leq
  \exp \Big( \frac{1}{2} Q_{j+s} (X,\zeta) \Big) 
  := \exp\bigg( C c_4 \kappa \sum_{a=0}^4 
  \sum_{\vec{\mu} \in \hat{e}^a}   \norm{\nabla_{j+s}^{\vec{\mu}} \zeta}^2_{L^2_{j+s} (X^*)} \bigg)
  .
\end{equation}
\item For any $c_4>0$, any integer $\ell$, there is $c_\kappa = c_\kappa(c_4,\ell)>0$ such
  that if $\kappa (L)=c_\kappa (\log L)^{-1}$ then
\begin{equation} \label{eq:g_j+s_expectation}
  \E [ e^{2 Q_{j+s}(X,\zeta)} ] \leq 2^{(N')^{-1}|X|_{j+s}}, \qquad \zeta \sim \cN (0, \Gamma_{j+s, j+s+(N')^{-1}} + m^2)
\end{equation}
for $m^2 \in [0, \delta]$ and sufficiently small $\delta >0$. 
\end{itemize}
\end{lemma}

The statement of this lemma is actually slightly stronger than its original version, but the proof is exactly the same. 
There is an immediate corollary of these computations. 

\begin{proof}[Proof of Lemma~\ref{lemma:E_G^r_j}]
The first inequality is \cite[Proposition~5.9]{dgauss1}. The second inequality follows from combining \eqref{eq:G^r_j_G_j+1_ratio} with Lemma~\ref{lemma:g_j+s_bound_by_quadratic}. 
\end{proof}

\subsection{Complex-valued shift of variables}
\label{sec:complex-valued-shift-of-variables}

We prove Lemma~\ref{lemma:Et_by_complex_shift} here.

\begin{lemma} \label{lemma:complex_shift_of_variables}
Let $X\in \cP_j$, 
let $F$ be a function such that $\norm{F}_{h, T_j (X)} < \infty$.
Let $v_2 \in \R^{\Lambda_N}$ be such that $\norm{\Gamma_{j+s, j+\bar{s}} v_{2}}_{C^2_{j}} < h/2$ for each $s,\bar{s} \in \{0, \cdots, 1-(N')^{-1} \}$ and $s < \bar{s}$.
If $\varphip \in \R^{\Lambda_N}$, $v = v_1 + iv_2$ for some $v_1 \in \R^{\Lambda_N}$,   
and $\zeta \sim \cN(0, \Gamma_{j+1})$ (or $\cN(0, \Gamma_N^{\Lambda_N})$ when $j+1 = N$) then
\begin{align}
\E[ F (X, \varphip  + \zeta  + \Gamma_{j+1} v) ] = e^{-\frac{1}{2} (v,  \Gamma_{j+1} v) } \E \big[ e^{(\zeta,  v)} F(X,  \varphip + \zeta  ) \big]
.
\end{align}
\end{lemma}
\begin{proof}
Since $X$ does not play any role, we will drop it at most places. 
By change of variable $\zeta \rightarrow \zeta - \Gamma_{j+1} v_1$,  we have
\begin{align}
& \E[ F ( \varphi + \zeta + \Gamma_{j+1} v_1 + i \Gamma_{j+1} v_2) ] 
 = e^{-\frac{1}{2} (v_1, \Gamma_{j+1} v_1)  } \E[ e^{ (\zeta,  v_1)}  F (\varphi + \zeta + i \Gamma_{j+1} v_2) ] .
\end{align}
We aim to prove, for any $m^2 >0$ sufficiently small, 
\begin{align}
& \E[ e^{ (\zeta^{(m^2)},  v_1)}  F (\varphi + \zeta^{(m^2)}  + i C v_2) ] 
  = e^{- i (v_1,  C v_2 ) +  \frac{1}{2} (v_2,  C v_2)} \E \big[ e^{( \zeta^{(m^2)} ,   v_1 + iv_2  )} F( \varphi + \zeta^{(m^2)} ) \big] 
  \label{eq:complex_int_by_parts}
\end{align}
where now $C = \Gamma_{j+1} + N' m^2$ and $\zeta^{(m^2)} \sim \cN (0,   C|_{X^*})$.  Once this is obtained, we can take the limit $m^2 \downarrow 0$ to conclude. 

Also having done the subscale decomposition $\zeta^{(m^2)} = \sum_{s=0}^{N'} \xi_{s}$ with $\xi_{s} \sim \cN(0,  \Gamma_{j+s, j+s'} + m^2)$ (where $s' = s+ (N')^{-1}$), 
the proof of \eqref{eq:complex_int_by_parts}, upto an induction,  reduces to proving
\begin{align}
& \E \big[  e^{(\xi_s ,  v_1)} F(\varphip + \xi_s + i  ( C_s + C_{> s}  ) v_2  ) \big] \nnb
&  \qquad \qquad \qquad  = e^{-i (v_1, C_s v_2) + \frac{1}{2} (v_2, C_s v_2)} \E \big[ e^{(\xi_s, v_1 + i v_2  )} F(\varphip + \xi_s + i C_{> s} v_2 ) \big]
\end{align}
where $C_s = \Gamma_{j+s, j+s'} + m^2$, $C_{>s} = \sum_{\bar{s} > s} C_{\bar{s}}$, $\xi_s \sim \cN (0, C_s)$ and $\varphip = \varphi + \sum_{\bar{s} \neq s} \xi_{\bar{s}}$. 
After writing the expectation in integral form, this is equivalent to
\begin{align}
& \int_{\R^{X^*}} d\zeta_s \,  e^{-\frac{1}{2} (\zeta_s,  C_{s}^{-1} \zeta_s) } e^{(\zeta_s ,  v_1 + i v_2 )} F(\varphip + \zeta_s  + i  ( C_s + C_{> s} ) v_2  )  \nnb
& \qquad \qquad  = \int_{\R^{X^*} - i C_s v_2 } d\zeta_s \,   e^{-\frac{1}{2} (\zeta_s,  C_{s}^{-1} \zeta_s) } e^{(\zeta_s ,  v_1 + i v_2 )} F(\varphip + \zeta_s + i  ( C_s + C_{> s} ) v_2  ) 
\end{align}
After having done orthonormal change of basis of $(\delta_{x} : x \in X^*)$ to $( e_y : y \in A)$ (with $|A| = |X^*|$) such that $C_s v_2 = \alpha e_{y_0}$ for some $\alpha \in \R$ and $y_0 \in A$,  
and writing $(\xi_{s,y})_{y \in A}$ for the coordinates for $\xi_s$ in this basis (so $\xi_s = \sum_{y\in A} \xi_{s,y} e_y$), 
this is implied by
\begin{align}
	& \int_{\R } d \xi_{s,y_0} e^{-\frac{1}{2} (\xi_s ,  C_s^{-1} \xi_s) } e^{(\xi_s ,  v_1 + i v_2)} F(\varphip + \xi_s +  i  ( C_s + C_{> s} ) v_2  )  \nnb
	& \qquad  \qquad = \int_{\R - i\alpha} d\xi_{s,y_0} \,   e^{-\frac{1}{2} (\xi_s,  C_{s}^{-1} \xi_s) } e^{(\xi_s ,  v_1 + i v_2 )} F(\varphip + \xi_s +  i  ( C_s + C_{> s} ) v_2  ) .
\end{align}
By Proposition~\ref{prop:analytic_on_strip} and our assumptions on the size of norms of $C_s v_2$ and $C_{>s} v_2$,  
the function $\xi_{s,y_0} \mapsto F( \sum_{y \in A} \xi_{s,y} e_y + \varphip +  i (C_s + C_{>s} ) v_2 )$ is analytic in an open neighbourhood of $\{ \xi_{y_0} = a+ ib : a\in \R, b \in [-\alpha,0] \}$ when $\xi_{s,y} \in \R$ for each $y\neq y_0$,
so the equality is again implied by Cauchy's theorem on complex integrals if we could prove
\begin{align}
	\Big| e^{-\frac{1}{2} (\xi_{s} + it e_{y_0} , C_s^{-1} (\xi_{s} + it e_{y_0})) } e^{(\xi_{s} + it e_{y_0} ,  v_1 + i v_2 )} F(\varphip + \xi_s + it e_{y_0} +  i  ( C_s + C_{> s} ) v_2  )  \Big| = o(1)
\label{eq:F_decay_at_infty}
\end{align}
as $\norm{\xi_{s} + it e_{y_0} }_{L^2 (X^*)} \rightarrow \infty$, when $\xi_{s} \in \R^{X^*}$ and $t \in [-\alpha, 0]$.
To bound this, 
we use \eqref{eq:bound_in_the_analytic_strip} and Lemma~\ref{lemma:G_change_of_scale_external_field},
obtaining
\begin{align}
	\big| F(\varphip + \xi_s + C_s v_1 + it e_{y_0}  + i  ( C_s + C_{> s} ) v_2  ) \big| 
	& \leq \norm{F}_{h,  T_j (X) } G_j (X, \varphi + \xi_s + C_s v_1 ) \nnb
	&  \leq C \norm{F}_{h,  T_j (X) } G_{j+s'}(X, \varphi) g_{j+s} (X_{j+s'},\xi_s )^2
	,
\end{align}
where $s' = s+ (N')^{-1}$.
While by Lemma~\ref{lemma:g_j+s_bound_by_quadratic}, 
$g_{j+s} (X_{j+\bar{s}},\xi)^2 \leq e^{ (\xi,  Q_{j+{s}} \xi) }$ for some quadratic form $Q_{j+s}$ with
$
\E\big[ e^{ 2 (\xi,  Q_{j+{s}} \xi) } \big] < \infty
$.
This implies
\begin{align}
	e^{-\frac{1}{2} (\xi_s , C_s^{-1} \xi_s) } 
	\big| F(\varphip + \xi_s + C_s v_1 + it e_{y_0}  + i  ( C_s + C_{> s} ) v_2  ) \big|
 	= o(1) \;\; \text{as} \;\; \norm{\xi_s}_{L^2 (X^*)} \rightarrow \infty
\end{align}
and proves \eqref{eq:F_decay_at_infty}.
\end{proof}

\begin{proof}[Proof of Lemma~\ref{lemma:Et_by_complex_shift}]
Recall, $u_{j+1} = \Gamma_{j+1} \tilde{\f}$ if $j\geq 0$.
By definition, 
\begin{align}
\Et \big[ F (X, \varphip + \zeta)   \big] = e^{-\frac{1}{2} \tau^2 (\tilde{\f}, \Gamma_{j+1} \tilde{\f}) } \E \Big[ e^{(\zeta,  \tau \tilde{\f})} F(X, \varphip + \zeta) \Big]
\end{align}
Also,  Lemma~\ref{lemma:extfield_bound_fine} implies
\begin{align}
	\norm{\tau \Gamma_{j+s, j+ \bar{s}} \tilde{f}}_{C^2_{j}} < h/2, \quad \bar{s} > s, \; s \in I_{N'}, \; \bar{s} \in (N')^{-1} + I_{N'}
\end{align}	
whenever $\tau < (2 M \rho^2 \log L )^{-1} h /2$. 
This verifies the assumptions of Lemma~\ref{lemma:complex_shift_of_variables}, completing the proof.
\end{proof}

\subsection{Proof of Lemma~\ref{lemma:Et_F_bound}} 

\begin{lemma}[Gaussian integration by parts]
\label{lemma:gauss_int_by_parts}

Given $\tilde{\f}$, let $u_{j+1} = \Gamma_{j+1} \tilde{\f}$. 
Let $F$ be a polymer activity such that $\norm{F}_{\vec{h}, T_j (X)} < \infty$. Then
\begin{align}
	\E\Big[  D^k F(X,  \varphi ' + \zeta) (u_{j+1}^{\otimes k}) \Big] = \sum_{l=0}^k
	\begin{pmatrix}
	k \\
	l
	\end{pmatrix}
	(u_{j+1}, \tilde{\f})^{l/2} \Her_l(0) \E^{\zeta}_{\Gamma_{j+1}} \big[ (\zeta|_{X^*} , \tilde{\f})^{k-l} F(X, \varphip + \zeta)  \big]
\end{align}
when $\Her_{2p +1} (0) = 0$ and $\Her_{2p} (0) = (-1)^p \frac{(2p)!}{2^p p!}$
for $p\in \Z_{\geq 0}$.
\end{lemma}

\begin{proof}
We drop  $X$ so that $F(X, \varphi)$ is written as $F(\varphi)$.
It is sufficient to prove this for $\zeta \sim \cN (0,C|_{X^*})$
and $u_{j+1} = C \tilde{\f}$, 
where $C = \Gamma_{j+1} + m^2$ and $m^2 >0$. 
Also, by integration by parts and upto a normalisation factor, 
$\E_C [  D^k F(X,  \varphi ' + \zeta) (u_{j+1}^{\otimes k}) ]$ can be written in terms of the Lebesgue integral
\begin{align}
\int_{\R^{X^*}} d\zeta \,  e^{-\frac{1}{2} (\zeta, C^{-1} \zeta)} D^k F (\varphip + \zeta ) (u_{j+1}^{\otimes k}) = \int_{\R^{X^*}} d\zeta \,  (-D)^k e^{-\frac{1}{2} (\zeta, C^{-1} \zeta)}  (u_{j+1}^{\otimes k}) F (\varphip + \zeta )  . 
\end{align}
The derivative on the right-hand side can be written as
\begin{align}
(-1)^k \frac{d^k}{d s^k} \Big|_{s=0} e^{-\frac{1}{2}  (\zeta + su_{j+1}, C^{-1} (\zeta + su_{j+1})) }  = (-1)^k \frac{d^k}{d s^k} \Big|_{s=0} e^{-\frac{1}{2} A s^2 - Bs -  E} 
\end{align}
where $A = (u_{j+1}, C^{-1} u_{j+1}) = (u_{j+1}, \tilde{\f})$, $B= (\zeta, C^{-1} u_{j+1}) = (\zeta, \tilde{\f})$ and $E = \frac{1}{2} (\zeta, C^{-1} \zeta )$. 
By the Leibniz formula,
\begin{align}
(-1)^k \frac{d^k}{d s^k} e^{-\frac{1}{2} A s^2 - Bs - E} = \sum_{l=0}^k \begin{pmatrix}
k \\
l
\end{pmatrix}
\Her_l ( \sqrt{A} s) A^{l/2} B^{k-l} e^{-\frac{1}{2} A s^2 - Bs - E} , 
\end{align}
where $\Her_l (x)$ is the Hermite polynomial of degree $l$. 
Therefore we obtain
\begin{align}
& \E_C \Big[  D^k F(X,  \varphi ' + \zeta) (u_{j+1}^{\otimes k}) \Big] \nnb
& \qquad \propto \sum_{l=0}^k
\begin{pmatrix}
k \\
l
\end{pmatrix}
(u_{j+1},  \tilde{\f})^{l /2} \Her_l (0) 
\int_{\R^{X^*}} d\zeta \, e^{-\frac{1}{2} (\zeta, C^{-1} \zeta) } (\zeta,  \tilde{\f})^{k-l} F(X, \varphip + \zeta)  \big] .
\end{align}
To conclude, take limit $m^2 \rightarrow 0$.

\end{proof}

\begin{lemma} \label{lemma:G_j_times_polynomial_bound}
Let $\f$ be as in \eqref{quote:assumpf}, $\tilde{\f} = (1+ s\gamma \Delta) \f$ and $u_{j+1}$ be defined by Definition~\ref{def:extfield_def}. 
If $\norm{D^n F (X)}_{n, T_j (X)} < \infty$, 
there are some constants $C_1 (M, \rho), C_2(M,  \rho, \rr)>0$ such that
\begin{align}
\begin{split}
& \big\|  \E^{\zeta}_{\Gamma_{j+1}} \big[ (\zeta|_{X^*} ,  \tilde{\f})^k D^n F (X, \varphip + \zeta + \rr u_{j+1})  \big] \big\|_{n, T_j (X, \varphi) } \\
& \qquad \qquad \leq C_2  2^{|X|_j} (C_1 \log L)^{\frac{3}{2} k} \Big( \Big\lceil \frac{k}{2} \Big\rceil ! \Big) \, G_{j+1} (\bar{X}, \varphip) \norm{D^n F (X)}_{n, T_j (X)} .
\end{split}
\end{align}
\end{lemma}

\begin{proof}
Again, we drop $X$. 
Since $|D^n F (\varphip + \zeta)| \leq \norm{D^n F}_{n, T_j (X)} G_j (X, \varphip + \zeta)$ 
we just need to bound $\E[ |(\zeta|_{X^*}, \tilde{\f})|^k G_j (X, \varphip + \zeta)  ]$. 
Firstly, after the subdecomposition $\zeta|_{X^*} = \sum_{s=0}^{N' - 1} \xi_s |_{X^*}$ (as in \eqref{eq:covariance_subdecomposition}) and using the Jensen's inequality, we have
\begin{align}
	|(\zeta|_{X^*} , \tilde{\f})|^k 
	= 
	\Big|  {\textstyle \sum_{s=0}^{N'-1} } \big( \xi_s |_{X^*} , \tilde{\f} \big) \Big|^k 
	& \leq 
	(N')^{k-1}  {\textstyle \sum_{s=0}^{N'-1} } | (\xi_s |_{X^*} , \tilde{\f}) |^k  \nnb
	& \leq 
	(N')^{k-1}  {\textstyle \sum_{s=0}^{N'-1} }  \norm{\tilde{\f}}^k_{L^1} \norm{\xi_s }^k_{L^{\infty} (X^*)}
\end{align}
But since
\begin{align}
 \norm{\xi_s}^k_{L^{\infty} (X^*)} \leq (c_4 \kappa )^{-\frac{k}{2}} \Gamma \Big( \frac{k + 2}{2} \Big)  \exp\big( c_4  \kappa \norm{\xi_s}^2_{L^{\infty} (X^*)} \big)
\end{align}
($\Gamma$ is the gamma function)
and $\norm{\tilde{\f}}_{L^1} \leq M \rho^2$, 
we see for some $C \equiv C(M, \rho)$ that
\begin{align}
|(\zeta |_{X^*} , \tilde{\f})|^k 
\leq 
(N')^{k-1} \sum_{s=0}^{N' - 1} C^k \kappa^{-\frac{k}{2}} \Big( \Big\lceil \frac{k}{2} \Big\rceil ! \Big)  e^{ c_4 \kappa  \norm{\xi_s}^2_{L^{\infty}(X^*) } } . 
\label{eq:(zeta,f)^k_bound}
\end{align}
Secondly,   \eqref{eq:G^r_j_G_j+1_ratio} gives decomposition
\begin{align}
G_j (X, \varphip + \zeta +\rr u_{j+1} ) \leq C(M, \rho,\rr,\ell) \prod_{s \in I_{N'} } g_{j+ (N')^{-1} s} (X_{j+(N')^{-1} s} ,  \xi_s)^2 G_{j+1} (\bar{X}, \varphip) . 
\end{align}
But since 
$
e^{ c_4 \kappa \norm{\xi_s}^2_{L^{\infty} (X^*)} } \leq g_{j+(N')^{-1} s} (X, \xi_s)
$,  
combining this with \eqref{eq:(zeta,f)^k_bound}, we have
\begin{align}
\big| (\zeta |_{X^*} , \tilde{\f}) \big|^k G_j (X, \varphip + \zeta) 
\leq (N')^{k} C^k \kappa^{-\frac{k}{2} }  \Big\lceil \frac{k}{2} \Big\rceil !  \, G_{j+1} (\bar{X}, \varphip)   \prod_{s=0}^{N' - 1} g_{j+ (N')^{-1} s} (X_{j+(N')^{-1} s} ,  \xi_s)^3  .
\end{align}
Also Lemma~\ref{lemma:g_j+s_bound_by_quadratic} says
\begin{align}
\E \Big[ \prod_{s=0}^{N' - 1} g_{j+ (N')^{-1} s} (X_{j+(N')^{-1} s} ,  \xi_s)^3 \Big] \leq 2^{|X|_j},
\end{align}
so the conclusion follows from recalling that $N' = \frac{\log L}{\log \ell}$ and $\kappa = c_{\kappa} (\log L)^{-1}$.

\end{proof}

\begin{lemma} \label{lemma:E[F(cdot+tau_u)]_bound}
Let $\f$, $\tilde{\f}$ and $u_{j+1}$ be as in the previous Lemma.
Let $\htau \leq (C \log L)^{-3/2}$ for sufficiently large $C \equiv C(M, \rho)$
and $\norm{ F(X)}_{\vec{h}, T_j (X)} < \infty$. 
Then
\begin{align}
	\norm{ \E [  F (X, \varphip + \zeta + (\cdot + \rr) u_{j+1} ; \cdot) ] }_{\vec{h}, T_j (X, \varphip)} \leq C_2 2^{|X|_j} \norm{ F(X)}_{\vec{h}, T_j (X)} G_{j+1} (\bar{X}, \varphip) 
\end{align}
for some $C_2 \equiv C_2(M, \rho,\rr)$,
when $\zeta \sim \cN (0, \Gamma_{j+1})$.
\end{lemma}
\begin{proof}
Again, we drop $X$.
We prove this by making the bound
\begin{align}
& \sum_{l=0}^{\infty} \frac{\htau^l}{l !} \Big\| \frac{d^l}{d \tau^l} \Big|_{\tau =0} D^n \E [ F(\varphip + \zeta +  (\tau + \rr) u_{j+1} ; \tau ) \Big\|_{n, T_j (X, \varphi)}  \nnb
& \qquad \qquad \leq C 2^{|X|_j} \sum_{l=0}^{\infty} \frac{\htau^l}{l!} \| \partial_{\tau}^l D^n F( \cdot ; \tau ) \|_{n, T_j (X)} G_{j+1} (\bar{X}, \varphip)
\end{align}
for each $n \geq 0$ and $C$ independent of $n$.
Also, since $n$ does not play any role in the proof, we will just prove this for the case $n=0$.
To show this, first make expansion
\begin{align}
& \frac{d^l}{d \tau^l} \Big|_{\tau =0}  F(\varphip + \zeta + (\tau + \rr) u_{j+1} ; \tau ) 
= \sum_{k=0}^l
\begin{pmatrix}
l \\
k
\end{pmatrix}
D^{l-k} \partial_{\tau}^{k } F(\varphip + \zeta + \rr u_{j+1} ; 0) (u_{j+1}^{\otimes l-k}) 
\end{align}
and by Lemma~\ref{lemma:gauss_int_by_parts}, 
\begin{align}
& \E \big[ D^{l-k} \partial_{\tau}^{k} F(\varphip + \zeta +  \rr u_{j+1} ; 0) (u_{j+1}^{\otimes l-k}) \big] \\
& \nonumber = \sum_{m=0}^{l-k}
\begin{pmatrix}
l-k \\
m
\end{pmatrix}
(u_{j+1}, \tilde{\f})^{\frac{m}{2}} \Her_m(0) \E\big[ (\zeta |_{X^*} , \tilde{\f})^{l-k-m} \partial_{\tau}^{k} F( \varphip + \zeta + \rr u_{j+1} ; 0)  \big] ,
\end{align}
while by Lemma~\ref{lemma:G_j_times_polynomial_bound}, 
\begin{align}
&\big| \E\big[ (\zeta |_{X^*} , \tilde{\f})^{l-k-m} \partial_{\tau}^{k} F( \varphip + \zeta + \rr u_{j+1} ; 0)  \big] \big| \\
&\nonumber  \leq C_2 2^{|X|_j} (C_1 \log L)^{\frac{3}{2} (l-k-m)} \Big( \Big\lceil \frac{l-k-m}{2} \Big\rceil ! \Big) \, G_{j+1} (\bar{X}, \varphip) \norm{\partial_{\tau}^{k} F}_{0, T_j (X)}
\end{align}
Combining these bounds, using $(u_{j+1}, \tilde{\f}) \leq C(M, \rho) \log L$ and
$\Her_{2p-1} (0) = 0$, $\Her_{2p} (0) = (-1)^p \frac{(2p) !}{2^p p!}$ for $p \in \Z_{\geq 0}$, we have
\begin{align}
& \sum_{l=0}^{\infty} \frac{\htau^l}{l !} \Big\| \frac{d^l}{d \tau^l} \Big|_{\tau =0} \E [ F(\varphip + \zeta +  (\tau + \rr) u_{j+1} ; \tau ) \Big\|_{0,  T_j (X, \varphi)} 
\\
& \leq 
C_2 2^{|X|_j} G_{j+1} (\bar{X}, \varphip) 
\sum_{k=0}^{\infty} \sum_{l' = 0}^{\infty} \frac{\htau^{l' + k}}{k !}   \sum_{p=0}^{\lfloor \frac{l'}{2} \rfloor} C_3^{2p} \frac{(C_1 \log L)^{\frac{3}{2} l' - 2p }}{2^{p} p! (l'-2p)! } \Big( \Big\lceil \frac{l'-2p}{2} \Big \rceil ! \Big)  \norm{ \partial_{\tau}^{k} F }_{0, T_j(X)}
\nonumber
\end{align} 
after reparametrising $l' = l-k$, $m = 2p$ and $C_3 = (C_1 / C)^{1/2}$.
This is bounded by
\begin{align}
C_2 2^{|X|_j} \sum_{k=0}^{\infty} \frac{\htau^k}{k!} \| \partial_{\tau}^k  F( \cdot ; \tau ) \|_{0, T_j (X)} G_{j+1} (\bar{X}, \varphip)  \times  \sum_{l'=0}^{\infty} \Big(\cdots \Big) 
\label{eq:E[F(cdot+tau_u)]_bound_final_bound}
\end{align}
where
\begin{align}
\sum_{l'=0}^{\infty} \Big(\cdots \Big) =   \sum_{l'=0}^{\infty} \big( (C_1 \log L)^{\frac{3}{2}} \htau \big)^{l'} \sum_{p=0}^{\lfloor \frac{l'}{2} \rfloor}   C_3^p \frac{(C_1 \log L)^{-2p }}{2^p p! (l'-2p) !}  \Big\lceil \frac{l'-2p}{2} \Big\rceil ! . 
\end{align}
But after using the trivial bound $\sum_{p=0}^{\lfloor \frac{l'}{2} \rfloor} \frac{1}{p! (l'-2p) !}  \big\lceil \frac{l' -2p}{2} \big\rceil !  \leq e$ and setting $\htau$ and $L$ so that $ \frac{1}{2} C_3 (C_1 \log L)^{-2} \leq 1$
and $\htau \leq \frac{1}{2} (C_1 \log L)^{-3/2}$,
we see that \eqref{eq:E[F(cdot+tau_u)]_bound_final_bound} is bounded by a constant that is independent of $L$.

\end{proof}

\begin{proof}[Proof of Lemma~\ref{lemma:Et_F_bound}]
By Lemma~\ref{lemma:Et_by_complex_shift},
\begin{align}
\E_{(\rr + \tau)} [F (X , \varphip  ; \tau)] &= \frac{\E \big[ e^{(\zeta, \tilde{\f})} F(X, \varphip + \zeta + \rr u_{j+1} ; \tau)  \big]}{\exp \big( \frac{1}{2}\tau^2 (\tilde{\f}, \Gamma_{j+1} \tilde{\f}) \big)} \nnb
&= \E [ F (X, \varphip + \zeta + (\tau + \rr ) u_{j+1} ; \tau)].
\end{align}
Hence in fact
\begin{align}
	\norm{ \E_{(\rr + \cdot)} [F (X , \varphip + \zeta  ; \cdot)] }_{\vec{h}, T_j (X, \varphip)} = \norm{ \E [ F (X, \varphip + \zeta + (\cdot + \rr) u_{j+1} ; \cdot)] }_{\vec{h}, T_j (X, \varphip)}
\end{align}
and we obtain the bound \eqref{eq:Et_F_bound} by applying Lemma~\ref{lemma:E[F(cdot+tau_u)]_bound}. 
The analyticity of $\tau \mapsto \E_{(\tau)} [D^n F(X, \varphip + \zeta  ; \tau)]$ is a consequence of Lemma~\ref{lemma:F'_is_analytic}.
\end{proof}

\subsection{Proof of Lemma~\ref{lemma:Loc_Et_K_j_bound}}
\label{sec:Loc_Et_K_j_bound_proof}

For the first inequality, since $\Loco_X \Et F_1(X, \varphip + \zeta) = \Et \hat{F}_{1,0} (X, \zeta)$ and by \eqref{eq:neutralisation_bound} and Lemma~\ref{lemma:Et_F_bound}, 
\begin{align}
\norm{ \E_{(\rr + \cdot)} \hat{F}_{1,0} (X, \zeta ; \cdot )}_{\htau, W} \leq \norm{ \E_{(\rr + \cdot)} F_1 (X, \zeta ;\cdot  )}_{\vec{h}, T_j (X)} \leq C' 2^{|X|_j} \norm{F_1 (X ; \cdot  )}_{\vec{h}, T_j (X)} .
\end{align}
The analyticity follows from the analyticity statement in Lemma~\ref{lemma:Et_F_bound}. 

For the second, since $\Loc_{Z,B}$ gives a polynomial of degree 2, we have
\begin{align}
\norm{\Loc_{Z, B} \E_{(\rr + \cdot)} [  F_2 ( Z, \varphip + \zeta ) ]}_{\vec{h}, T_j (B, \varphip)} \leq 4 \norm{\Loc_{Z, B} \E_{(\rr + \cdot)} [  F_2 ( Z, \varphip + \zeta ) ]}_{(h/2, \htau), T_j (B, \varphip)}
\end{align}
while by \eqref{eq:weaker_Et_bound}, 
\begin{align}
\norm{\Loc_{Z, B} \E_{(\rr + \cdot)} [  F_2 ( Z, \varphip + \zeta ) ]}_{(h/2, \htau),  T_j (B, \varphip)} \leq \norm{\Loc_{Z, B} \E [  F_2 ( Z, \varphip + \zeta + \rr u_{j+1} ) ]}_{h,  T_j (B, \varphip)}
\end{align}
whenever $\htau < C (\log L)^{-1} h$.
But \cite[(6.16)]{dgauss1} says
\begin{align}
\norm{\Loc_{Z, B} \E [  F_2 ( Z, \varphip + \zeta + \rr u_{j+1} ) ]}_{h,  T_j (B, \varphip)} & \leq C \log L \norm{ F_2 (Z )}_{h, T_j (Z)} e^{c_w \kappa (L) w_j (B, \varphip + \rr u_{j+1})^2} \nnb
& \leq C' (L) \norm{ F_2 (Z )}_{h, T_j (Z)} e^{c_w \kappa (L) w_j (B, \varphip )^2}
\end{align}
so we obtain \eqref{eq:Loc_Et_K_j_bound}.

\section{Contraction estimates}
\label{sec:contraction-estimates}

\subsubsection{Taylor expansion.}

Given a polymer activity $F(X) \in \cN_j (X)$ on $X\in \cP_j$,
we define the Taylor polymer of degree $n$ by
\begin{align}
& \operatorname{Tay}^{\varphi}_n F (X, \varphi + \psi) = \sum_{k=0}^{n} \frac{1}{k !} \sum_{x_1, \cdots, x_k \in X^*} \frac{\partial^k F( X, \psi)}{\partial \psi (x_1) \cdots \partial \psi (x_k)} \varphi(x_1) \cdots \varphi(x_k) \\
& \operatorname{Rem}^{\varphi}_n F (X, \varphi + \psi) = F( X, \varphi + \psi) - \operatorname{Tay}^{\varphi}_n F ( X, \varphi + \psi) .
\end{align}
Also for $F(X) \in {\cN}_j(X)$ neutral, define
\begin{align}
& \bar{\Tay}_2^{ \varphi} F(X, \varphi + \psi) = \frac{1}{|X|} \sum_{x_0 \in X} \Tay_2^{\delta \varphi} F(X, \delta \varphi + \psi) \\
& \bar{\Rem}_2^{ \varphi} F (X, \varphi + \psi) = F(X,  \varphi + \psi) - \bar{\Tay}_2^{ \varphi} F(X,  \varphi + \psi)
\end{align}
where $\delta \varphi (x) := \varphi(x) - \varphi(x_0)$ is dependent of the choice of $x_0 \in X$.

\subsubsection{Fourier decomposition.}

If $F : \cP_j \times \R^{\Lambda} \rightarrow \C$ is a ($2\pi\beta^{-1/2}$-)periodic function, i.e., $F(X, \varphi + y \one) = F(X, \varphi)$ for any $X\in \cP_j$, $y\in 2\pi \beta^{-1/2} \Z$ and $\one$ is the constant field, we have defined the neutral part of $F$ in Section~\ref{sec:neutralisation}. 
In fact, this definition can be extended to Fourier components of any period, 
\begin{align}
\hat{F}_q( X, \varphi) = \frac{\sqrt{\beta}}{2\pi} \int_0^{\frac{2\pi}{\sqrt{\beta}}} ds \; e^{-i\sqrt{\beta} q s} F(X, \varphi + s), \quad q \in \Z.
\end{align}
This is also called the \emph{charge$-q$ component} of $F$, and 
satisfy
\begin{equation}
\label{eq:charge_qdefin}
F(X,\varphi+ t)= \sum_{q\in \mathbb{Z}} e^{i\sqrt{\beta} q t} \hat{F}_q(X, \varphi ), \quad t \in \R .
\end{equation}
If $F = \hat{F}_q$ for some $q\in \Z$, then we call $F$ is a polymer activity of charge $q$.

As explained in Section~\ref{sec:loc_operator}, 
$\operatorname{Loc}_X^{(k)} F$ is intended to be the approximation of $\bar{\Tay}_k F$ with better algebraic properties, so we can bound $(1- \operatorname{Loc}_X^{(k)} ) \E_{(\rr + \tau)} F(X)$ by bounding each term of the following: 
\begin{align}
	(1- \operatorname{Loc}_X^{(k)} ) \E_{(\rr+\tau)} F(X, \varphip + \zeta) \label{eq:1-Loc_rewriting}
	& = \bar{\Rem}_{k}^{\varphip}  \E_{(\rr + \tau)} \hat{F}_0 (X, \varphip + \zeta) \\
	& \quad + (\bar{\Tay}_{k}^{ \varphip} - \operatorname{Loc}^{(k)}_X ) \E_{(\rr + \tau)} \hat{F}_0 (X, \varphip + \zeta) ) 
	\\
	& \quad + \sum_{q \in \Z \backslash \{ 0 \}} \E_{(\rr + \tau)} \hat{F}_q (X, \varphip + \zeta) 
\end{align} 
Hence the proof of Proposition~\ref{prop:Loc-contract_v2} can now be reduced to the following statements. 

\begin{proposition} \label{prop:contraction_estimates_external_field}
Let $X \in \mathcal{S}_j$,  $\vec{h} = (h, \htau)$
 and let $F$ be a ($2\pi/\sqrt{\beta}$-)periodic polymer activity such that $\norm{F}_{\vec{h}, T_j (X)} < \infty$.
Let $L\geq L_0$, $h\geq c_h \sqrt{\beta}$ and $\htau < (C_1 \log L)^{-1} h$ for $L_0,  c_h$ and $C_1$ sufficiently large.
 Then the following hold for some $C=C (M, \rho, \rr)>0$.
\begin{itemize}
\item[(1)] If $F$ has charge $q$ with $|q|\geq 1$, then for all $\varphip \in \R^{\Lambda_N}$,
\begin{align}
\begin{split}
& \norm{\E_{(\rr +\cdot)} F (X, \varphip + \zeta) }_{ \vec{h}, T_{j+1} ({X}, \varphip)}
\\
& \qquad \qquad \qquad \qquad \leq C e^{2 \sqrt{\beta} |q| h} e^{-(|q|-1/2) \Gamma_{j+1} (0)} \norm{F(X)}_{\vec{h} , T_j (X)} G_{j+1} (\bar{X}, \varphip) . \label{eq:contraction_of_charge_q_term_external_field}
\end{split}
\end{align}

\item[(2)] If $F$ is neutral, then for all $\varphip \in \R^{\Lambda_N}$ and $m \in \{0,2\}$,
\begin{align}
\begin{split}
& \norm{ \bar{\Rem}_{m}^{\varphip} \E_{(\rr +\cdot)}  F(X,\varphip + \zeta)  }_{ \vec{h}, T_{j+1} (X, \varphip)}  \\
& \qquad \qquad \qquad \qquad \qquad \leq C
L^{-(m+1)} (\log L)^{(m+1)/2} \norm{F(X)}_{\vec{h} , T_j (X)} G_{j+1} (\bar{X}, \varphip) . \label{eq:gaussian_contraction_external_field}
\end{split}
\end{align}
\end{itemize}
\end{proposition}

\begin{lemma} \label{lemma:Loc_minus_Tay}
    Let $F$ be a neutral polymer activity such that
    $F(X, \varphi) = F(X,-\varphi)$
    and $\norm{F}_{\vec{h}, T_j} < \infty$. 
    Then for $X\in \cS_j$,
\begin{align} \label{eq:Loc_minus_Tay_2}
\begin{split}
    & \| \Loc_X \E_{(\rr + \cdot)} F(X, \varphip +\zeta ; \cdot) - \bar\Tay_2 \E_{(\rr + \cdot)} F(X, \varphip +\zeta; \cdot)\|_{\vec{h},T_{j+1}(\bar{X}, \varphip)} \\
    & \qquad \qquad \qquad \qquad \qquad \qquad  \qquad \qquad  \leq
	C L^{-3}(\log L) (A')^{-|X|_j} 
	\|F\|_{\vec{h},T_j, A'} G_{j+1} (\bar{X}, \varphip).
\end{split}
\end{align}
  for any $A' \geq 1$.
\end{lemma}

\subsection{Proof of Proposition~\ref{prop:contraction_estimates_external_field}~(1)}

\begin{corollary}
Let $q\in \Z \backslash \{ 0 \}$, $X\in \cS_j$, $\xi(x) = \sqrt{\beta}(\Gamma_{j+1} (x-x_0) - \Gamma_{j+1} (0))$
and $h \geq c_h \sqrt{\beta}$ for $c_h$ sufficiently large.
Let $\f$ satisfy \eqref{quote:assumpf},  $u_{j+1}$ defined by Definition~\ref{def:extfield_def} 
and $|\tau| < \htau \leq (C \log L)^{-1} h$ for sufficiently large $C$.
If $F$ is a $q$-charge polymer activity on $X$ with $\norm{F(X)}_{h, T_j (X)} < \infty$, then
\begin{align}
\E[F (X, \varphi + \zeta + \tau u_{j+1}) ] 
= 
e^{-\frac{1}{2} \beta \Gamma_{j+1} (0) (2|q| - 1)} \E \big[ e^{-i\sqrt{\beta} \sigma_q \zeta_{x_0}} F(X, \varphi + \zeta +  \tau u_{j+1} + i \sigma_q \xi ) \big]
\end{align}
where $\sigma_q = \operatorname{sign} (q)$ and any $\varphi \in \R^{\Lambda_N}$. 
\end{corollary}
\begin{proof}
Again, omit $X$. Without loss of generality, we take $q >0$.
We apply Lemma~\ref{lemma:complex_shift_of_variables}
with $v = i \sqrt{\beta} \delta_{x_0}$ 
(and the assumptions are verified by \cite[Lemma~6.11]{dgauss1}, if $c_h$ is chosen sufficiently large)
to obtain
\begin{align}
\E \big[ e^{-i\sqrt{\beta} \sigma_q \zeta_{x_0}} F(\varphi + \zeta +  \tau u_{j+1} + i \sqrt{\beta}  \Gamma_{j+1} (x-x_0) ) \big] = e^{\frac{1}{2} \beta \Gamma_{j+1} (0) } \E[ F (\varphi + \zeta + \tau u_{j+1}) ]
.
\end{align}
But since $F$ has charge $q$, 
\begin{align}
F \big((\varphi + \zeta +  \tau u_{j+1} + i \sqrt{\beta}  \Gamma_{j+1} (x-x_0) \big) = e^{ \beta \Gamma_{j+1} (0) |q|} F \big( \varphi + \zeta +  \tau u_{j+1} + i  \xi \big) 
.
\end{align}
\end{proof}

\begin{lemma} \label{lemma:contraction_of_charge_q_term}
Let $h \geq c_h  \sqrt{\beta}$ and $L\geq L_0$ for $L_0$ and $c_h$ sufficiently large.
There exists $C \equiv C(M, \rho,\rr) >0$ such that for $X\in \mathcal{S}_j$,
and any
charge-$q$ polymer activity $F$
with $|q| \geq 1$ and $\|F (X )\|_{h,T_j (X)} < \infty$, and all $\varphip \in \R^{\Lambda_N}$,
\begin{equation} \label{eq:contraction_of_charge_q_term-1}
  \norm{\E_{(\rr)} [ F (X,  \varphip + \zeta )]}_{2h, T_{j+1}(X, \varphip)}
  \leq C e^{2 \sqrt{\beta} |q| h} e^{-(|q|-1/2) r \beta \Gamma_{j+1} (0)} \norm{ F(X)}_{h, T_j (X)} G_{j+1} (\bar{X},  \varphip ).
\end{equation}
\end{lemma}
\begin{proof}
This almost follows from \cite[Lemma~6.13]{dgauss1}. To be specific, by \cite[(6.59)]{dgauss1} (applied with $r=1$ there),  for any $h' > 0$,
\begin{align}
	\norm{\E_{(\rr)} [F(X, \varphip + \zeta)]}_{h', T_j (X, \varphip)} \leq e^{-\frac{2 |q| - 1 }{2} \beta \Gamma_{j+1} (0) } \E_{(\rr)}  \big[ \norm{F(X, \varphip + \zeta)}_{h' + \norm{\xi} , T_j (X, \varphip)} \big] 
	\label{eq:6.59_dgauss1}
\end{align}
where $\xi(x) = \sqrt{\beta} ( \Gamma_{j+1} (x-x_0) - \Gamma_{j+1} (0))$,  $x_0 \in X$ and $\norm{\xi} \equiv \norm{\xi}_{C_j^2} \leq C_1 \sqrt{\beta}$ for some $C_1 > 0$ (by \cite[Lemma~6.11]{dgauss1}). Also by \cite[(6.60)]{dgauss1}, 
\begin{align}
	\norm{\E_{(\rr)}  [ F(X, \varphip + \zeta) ]}_{2h, T_{j+1} (X, \varphip)} \leq e^{2\sqrt{\beta} q h} \norm{ \E_{(\rr)}  [ F(X, \varphip + \zeta)] }_{2C_2 L^{-1} h, T_j (X,\varphip )} 
	\label{eq:6.60_dgauss1}
\end{align}
for some $C_2 >0$.
Hence setting $h' = 2 C_2 L^{-1} h$ in \eqref{eq:6.59_dgauss1} and combining with \eqref{eq:6.60_dgauss1} gives
\begin{align}
& \norm{\E_{(\rr)}  [F(X, \varphip + \zeta)]}_{2h, T_{j+1} (X,\varphip)} \nnb
& \qquad  \leq e^{-\frac{2|q| -1}{2} \beta \Gamma_{j+1} (0)} e^{2\sqrt{\beta} |q| h} \E[G_j (X, \varphip + \zeta + \rr u_{j+1})] \norm{F (X)}_{2 C_2 L^{-1} h + \norm{\xi}, T_j (X)} .
\end{align}
By Lemma~\ref{lemma:E_G^r_j} 
$\E[G_j (X, \varphip + \zeta + \rr u_{j+1})] \leq C 2^{|X|_j} G_{j+1} (X, \varphip)$,
and for $L$ and $c_h$ sufficiently large, $2 C L^{-1} h + \norm{\xi} \leq h$, so we have the desired bound. 
\end{proof}

\begin{proof}[Proof of Proposition~\ref{prop:contraction_estimates_external_field}~(1)]

The conclusion follows from combining \eqref{eq:contraction_of_charge_q_term-1} with \eqref{eq:weaker_Et_bound}.
\end{proof}

\subsection{Proof of Proposition~\ref{prop:contraction_estimates_external_field}~(2)}

\begin{lemma} \label{lemma:gaussian_contraction2}
For $X \in \cS_j$, $\varphi \in \R^{\Lambda_N}$ and neutral $F(X)$ such that $\norm{F}_{h, T_j } < \infty$ and $m \in \{0,2\}$,
choose some $x_0 \in X$ and let $\delta \varphip (x) = \varphip (x) - \varphip (x_0)$.
Then
\begin{equation}
\begin{split}
  \norm{ 
    & \textnormal{Rem}^{\delta \varphip}_m \E_{(\rr)}  F 
    ( {X} , \zeta + \delta \varphip)}_{2h, T_{j+1} (\bar{X}, \varphip)}  \\
    & \qquad\qquad\qquad\qquad \leq C 
   L^{-(m+1)} (\log L)^{(m+1)/2}  (A')^{-|X|_j} 
    \|F\|_{h, T_j,A'}
    G_{j+1} (\bar{X}, \varphip)
\end{split}
\end{equation}
for some $C \equiv C(M, \rho,\rr)$.
\end{lemma}
\begin{proof}
The proof is very similar to \cite[Lemma~6.16]{dgauss1}, 
just observing that the norm on the left-hand side can admit $2h$ instead of $h$, 
and using Lemma~\ref{lemma:E_G^r_j}.

\end{proof}

\begin{proof}[Proof of Proposition~\ref{prop:contraction_estimates_external_field}~(2)]
Since $\bar{\textnormal{Rem}}^{\varphip}_m = \frac{1}{|X|} \sum_{x_0 \in X} \textnormal{Rem}^{\delta \varphip}_m$ where $\delta \varphip (x) = \varphip (x) - \varphip (x_0)$,  it suffices to prove the same bound for $\textnormal{Rem}^{\varphip}_m \E_{(\rr + \tau )} F( {X} , \zeta + \varphip ; \tau)$. Also by \eqref{eq:weaker_Et_bound}, 
\begin{align}
	\norm{\textnormal{Rem}^{\varphip}_m \E_{(\rr + \cdot )} F( {X} , \zeta + \varphip ; \cdot) }_{\vec{h}, T_{j+1} (\bar{X}, \varphi)} 
	\leq \norm{\textnormal{Rem}^{\varphip}_m \E_{(\rr)} F( {X} , \zeta + \varphip ; \cdot) }_{\vec{h}', T_{j+1} (\bar{X}, \varphi)}
\end{align}
where $\vec{h}' = (2h, \htau)$, and by definition of $\norm{\cdot}_{\vec{h}', T_{j+1} (\bar{X}, \varphi)}$,
\begin{align}
	\norm{\textnormal{Rem}^{\varphip}_m \E_{(\rr)} F( {X} , \zeta + \varphip ; \cdot) }_{\vec{h}', T_{j+1} (\bar{X}, \varphi)} 
	\leq \sum_{k=0}^{\infty } \frac{\htau^k }{k!} \norm{\textnormal{Rem}^{\varphip}_m \E_{(\rr)} \partial_{\tau}^k F( {X} , \zeta + \varphip ; \tau) }_{2h, T_{j+1} (\bar{X}, \varphi)} .
\end{align}
But Lemma~\ref{lemma:gaussian_contraction2} bounds the right-hand side, giving the desired conclusion. 
\end{proof}

\subsection{Proof of Proposition~\ref{prop:Loc-contract_v2}}
\label{sec:proof_of_Loc-contract}

\begin{proof}[Proof of Lemma~\ref{lemma:Loc_minus_Tay}]
Since $\Loc_X \E_{(\rr+ \tau)} F(X, \varphip + \zeta)$ and $\bar{\Tay_2} \E_{(\rr+ \tau)} F (X, \varphip + \zeta)$ are polynomials in $\varphip$ of degree 2, we have
  \begin{align}
    & \| \Loc_X \E_{(\rr+ \tau)} F (X, \varphip +\zeta) - \bar\Tay_2 \E_{(\rr+ \tau)} F (X, \varphip +\zeta)\|_{h,T_{j+1}(\bar{X}, \varphip)} \nnb
    & \qquad \qquad \leq 4 \| \Loc_X \E_{(\rr+ \tau)} F (X, \varphip +\zeta) - \bar\Tay_2 \E_{(\rr+ \tau)} F (X, \varphip +\zeta)\|_{2h,T_{j+1}(\bar{X}, \varphip)}   .
  \end{align}
Also,  by \eqref{eq:weaker_Et_bound}, it is enough to prove
\begin{align}
    & \| \Loc_X \E_{(\rr)} F (X, \varphip +\zeta) - \bar{\Tay}^{\varphip}_2 \E_{(\rr)} F (X, \varphip +\zeta)\|_{2h,T_{j+1}(\bar{X}, \varphip)}
	\nnb    
    & \qquad \qquad \qquad \qquad \qquad \qquad \qquad  \leq
    C L^{-3}(\log L) (A')^{-|X|_j} 
    \norm{F}_{h,T_j, A'} G_{j+1} (\bar{X}, \varphip).
    \label{eq:Loc_minus_Tay_intermediate}
\end{align}
for any neutral polymer activity $F$. 
And this follows from (the proof of) \cite[Lemma~6.16]{dgauss1}, which says
\begin{align}
    & \| \Loc_X \E_{(\rr)} F(X, \varphip +\zeta) - \bar{\Tay}^{\varphip}_2 \E_{(\rr)} F (X, \varphip +\zeta)\|_{h,T_{j+1}(\bar{X}, \varphip)} \nnb
    & \qquad \qquad \qquad \qquad \leq    C L^{-3}(\log L) (A')^{-|X|_j} 
    \norm{F}_{h,T_j, A'} \E_{(\rr)} [ G_{j} ( X, \varphip) ],  \nnb
       & \qquad \qquad \qquad \qquad \leq   C'  L^{-3}(\log L) (A' / 2)^{-|X|_j} 
    \norm{F}_{h,T_j, A'}  G_{j+1} (\bar{X}, \varphip), 
\end{align}
where the final inequality is due to Lemma~\ref{lemma:E_G^r_j}.

\end{proof}

\begin{proof}[Proof of Proposition~\ref{prop:Loc-contract_v2}]

Take $k =0$ or $2$.
After expanding in the Fourier series $F = \sum_{q \in \Z} \hat{F}_{q}$, we have
\begin{align}
    & \| \operatorname{Loc}^{(k)}_X \E_{(\rr+ \cdot)} F(X, \varphip+\zeta ; \cdot) - \E_{(\rr+ \cdot)} F(X, \varphip+\zeta ; \cdot)\|_{\vec{h}, T_{j+1}(\bar X, \varphip)}  \nnb
   & \qquad\qquad 
   \leq \| \operatorname{Loc}^{(k)}_X \E_{(\rr+ \cdot)} \hat{F}_0(X, \varphip+\zeta ; \cdot) - \E_{(\rr+ \cdot)} \hat{F}_0 (X, \varphip+\zeta ; \cdot)\|_{\vec{h}, T_{j+1}(\bar X, \varphip)} \nnb
   & \qquad \qquad \qquad 
   + \sum_{q \in \Z \backslash \{ 0 \}}  \norm{ \E_{(\rr+ \cdot)} \hat{F}_q (X, \varphip+\zeta ; \cdot) }_{\vec{h},  T_{j+1}(\bar X, \varphip)}  .
\end{align}
First by Proposition~\ref{prop:contraction_estimates_external_field}~(2) and Lemma~\ref{lemma:Loc_minus_Tay}, 
\begin{align}
& \| \operatorname{Loc}^{(k)}_X \E_{(\rr+ \cdot)} \hat{F}_0(X, \varphip+\zeta ; \cdot) - \E_{(\rr+ \cdot)} \hat{F}_0 (X, \varphip+\zeta ; \cdot)\|_{\vec{h}, T_{j+1}(\bar X, \varphip)}  \nnb
& \leq 
\| \operatorname{Loc}^{(k)}_X \E_{(\rr+ \cdot)} \hat{F}_0(X, \varphip+\zeta ; \cdot) - \bar{\Tay}_k^{\varphip} \E_{(\rr+ \cdot)} \hat{F}_0 (X, \varphip+\zeta ; \cdot)\|_{\vec{h},T_{j+1}(\bar X, \varphip)} 
\nnb
& \qquad + \| \bar{\Rem}_k^{\varphip} \E_{(\rr+ \cdot)} \hat{F}_0 (X, \varphip+\zeta ; \cdot)\|_{\vec{h},T_{j+1}(\bar X, \varphip)} \nnb
& \leq 
C L^{-(k+1)} (\log L)^{\frac{k+1}{2}} (A')^{-|X|_j} \norm{F}_{\vec{h}, T_j, A'} G_{j+1} (\bar{X}, \varphi) 
\end{align}
Secondly by Proposition~\ref{prop:contraction_estimates_external_field}~(1),
applied on the terms with $\hat{F}_q$ ($q\neq 0$), we have
\begin{align}
& \norm{ \E_{(\rr+ \cdot)} \hat{F}_q (X, \varphip+\zeta ; \cdot) }_{h,T_{j+1}(\bar X, \varphip)} 
\nnb
& \qquad \qquad\qquad \qquad \leq C e^{2 \sqrt{\beta}qh} e^{-(q-1/2)\beta\Gamma_{j+1}(0)} (A')^{-|X|_j} \norm{F}_{\vec{h}, T_j, A'} G_{j+1} (\bar{X}, \varphip) .
\end{align}
\end{proof}

\subsection{Proof of Proposition~\ref{prop:largeset_contraction_external_field}}

\label{sec:prop_largeset_contraction_external_field_proof}

\begin{lemma} \label{lemma:largeset_contraction_building_block}
Let $X\in \cP^c_{j+1}$ and $A' \geq A_0 (L)$ where $A_0 (L)$ is some function only polynomially large in $L$.  Then
\begin{align}
\sum_{Y\in \Conn_j \backslash \cS_j}^{\bar{Y} = X} (A')^{-|Y|_j} \leq (C L^2 (A')^{-(1+\eta)} )^{|X|_{j+1}} 
\end{align}
for some $C >0$ and $\eta >0$.
\end{lemma}
\begin{proof}
This is a bound already proved in the lines between \cite[(6.115)--(6.118)]{dgauss1}.

\end{proof}

\begin{proof}[Proof of Proposition~\ref{prop:largeset_contraction_external_field}]

By definition of $\mathbb{S}$-operator,  
\begin{align}
\big\| \mathbb{S} \big[ \E_{(\rr + \cdot)} [ F 1_{Y \not\in \cS_j}] \big] (X, \varphip) \big\|_{\vec{h}, T_{j+1} (X, \varphip)} 
\leq \sum_{Y\in \Conn_j \backslash \cS_j}^{\bar{Y} = X}  \norm{ \E_{(\rr, \cdot)} F(Y, \varphip + \zeta ; \cdot) }_{\vec{h}, T_j (X, \varphip)}
,
\end{align}
but by Lemma~\ref{lemma:Et_F_bound}, this is bounded by some multiple of
\begin{align}
G_{j+1} (X, \varphip)  \sum_{Y\in \Conn_j \backslash \cS_j}^{\bar{Y} = X} 2^{|Y|_j} \norm{F(Y)}_{\vec{h}, T_j (Y)} 
\leq G_{j+1} (X, \varphip) \norm{F}_{\vec{h}, T_j , A'} \sum_{Y\in \Conn_j \backslash \cS_j}^{\bar{Y} = X} (A'/2)^{-|Y|_j} 
.
\end{align}
We then just need to apply Lemma~\ref{lemma:largeset_contraction_building_block} and take $A'$ sufficiently large (depending on $L$) to conclude.
\end{proof}

\section*{Acknowledgements}
We are grateful to Roland Bauerschmidt and Pierre-Fran\c{c}ois Rodriguez for helpful discussions and Roland Bauerschmidt for his comments on the manuscript.
JP was supported by the Cambridge doctoral training centre Mathematics of Information and partially supported by the European Research Council under the
European Union's Horizon 2020 research and innovation programme
(grant agreement No.~851682 SPINRG).
He gratefully acknowledges the support and hospitality of the
University of British Columbia and of the Pacific Institute for the Mathematical Sciences in Vancouver during part of this work.

\bibliography{all}
\bibliographystyle{plain}

\end{document}